\documentclass[reqno,11pt]{amsart}

\usepackage{amscd,amssymb}
\usepackage[cmtip,arrow,curve,matrix]{xy}

\usepackage{fullpage}

\numberwithin{equation}{section}
\newtheorem{Theorem}[equation]{Theorem}

\newtheorem{Lemma}[equation]{Lemma}
\newtheorem{Proposition}[equation]{Proposition}
\newtheorem{prop}[equation]{Proposition}
\newtheorem{Corollary}[equation]{Corollary}

\renewcommand{\thesubsection}{\thesection.\Alph{subsection}}

\theoremstyle{definition}

\newtheorem{Definition}[equation]{Definition}
\newtheorem{Remark}[equation]{Remark}
\newtheorem{Example}[equation]{Example}

\DeclareMathOperator{\Aut}{Aut}
\newcommand{\Ahat}{\widehat{A}}
\newcommand{\A}{A}

\newcommand{\Borel}{\mathcal{B}}
\newcommand{\bstringc}{BString_{\C}}

\newcommand{\buoneHeps}{BU(1)^H_{\eps}}
\newcommand{\bulrp}[1]{BU\lrp{#1}}

\newcommand{\busix}{\bulrp{6}}
\newcommand{\Bc}{H_{\T}}
\newcommand{\Bh}{H^{\T}}
\newcommand{\buoneAeps}{BU(1)^A_{\eps}}

\newcommand{\CatE}{\mathcal{E}}

\DeclareMathOperator{\ch}{ch}
\newcommand{\ChB}{c^\Borel}
\newcommand{\cO}{\mathcal{O}}
\newcommand{\cI}{\mathcal{I}}
\newcommand{\cOA}{\cO_{\A}^{\wedge}}

\newcommand{\cochars}{\check{T}}
\newcommand{\Chdot}{c_{\bullet}}
\newcommand{\ChBdot }{c^{\Borel}_{\bullet}}

\newcommand{\CoordDiv}{\mathcal{D}}
\DeclareMathOperator{\colim}{colim}
\newcommand{\C}{\mathbb{C}}
\newcommand{\CTate}{\Tate{C}}
\newcommand{\CTateh}{\Tate{\widehat{C}}}

\newcommand{\cK}{\mathcal{K}}
\newcommand{\cF}{\mathcal{F}}
\newcommand{\cKF}{\cK_{\cF}^{\wedge}}
\newcommand{\cp}{\C P^{\infty}}

\newcommand{\deltaA}{\delta_{\A}}
\newcommand{\deltaBA}{\delta^{\Borel}_{\A}}
\newcommand{\deltaBAp}{(\delta_{\A}')^{\Borel}}
\newcommand{\deltaBApp}{(\delta_{\A}'')^{\Borel}}
\DeclareMathOperator{\diag}{diag}
\DeclareMathOperator{\Div}{Div}
\newcommand{\DL}{\uln{\omega}}

\newcommand{\cE}{\mathcal{E}}
\newcommand{\e}{0}
\newcommand{\elr}[1]{E\langle #1 \rangle}
\newcommand{\ecf}{E\mathcal{F}}
\newcommand{\etf}{\tilde{E}\mathcal{F}}
\newcommand{\efp}{E\mathcal{F}_{+}}
\newcommand{\eqdef}{\overset{\text{def}}{=}}
\newcommand{\eps}{\epsilon}
\newcommand{\EJ}{{EC}}
\newcommand{\EJc}{EC_{\T}}

\DeclareMathOperator{\Ext}{Ext}

\DeclareMathOperator{\Gr}{Gr}
\newcommand{\Gah}{\widehat{\mathbb{G}}_{a}}
\newcommand{\Gmh}{\widehat{\mathbb{G}}_{m}}

\newcommand{\Gre}{\mathcal{M}}

\newcommand{\h}{\mathfrak{h}}
\newcommand{\Hom}{\mathrm{Hom}}
\newcommand{\heq}{\simeq}

\newcommand{\iso}{\cong}
\newcommand{\id}{\mathrm{id}}

\newcommand{\J}{C}
\newcommand{\Jhat}{\widehat{\J}}
\newcommand{\Jlr}[1]{\J\langle #1 \rangle}

\newcommand{\Clr}[1]{\C\langle #1 \rangle}

\newcommand{\KTate}{\Tate{K}}
\DeclareMathOperator{\Ker}{Ker}

\newcommand{\lrp}[1]{\{ {#1}\}}

\newcommand{\Loo}{\mathcal{L}}
\newcommand{\LooAtr}{\gamma}

\DeclareMathOperator{\map}{map}
\newcommand{\mstringc}{MString_{\C}}
\newcommand{\musix}{MU\langle 6\rangle}

\DeclareMathOperator{\sym}{S}

\newcommand{\nonequivBU}{\overline{BU}}
\newcommand{\nonequivBSU}{\overline{BSU}}
\newcommand{\nonequivkU}{\overline{ku}}
\newcommand{\nonequivK}{\overline{K}}
\renewcommand{\O}{\mathcal{O}}
\newcommand{\OA}{\O^{\wedge}_{\A}}

\newcommand{\tU}{\mathcal{V}}
\newcommand{\tSU}{S\mathcal{V}}
\newcommand{\tUd}{\mathcal{V} (d)}
\newcommand{\tSUd}{S\mathcal{V} (d)}
\newcommand{\tStringc}{Str\mathcal{V}}
\newcommand{\tStringcW}{\tStringc_{W}}

\DeclareMathOperator{\Pic}{Pic}
\newcommand{\plus}{+}
\newcommand{\qf}{\phi}

\DeclareMathOperator{\Rep}{Rep}
\DeclareMathOperator{\SRep}{SRep}

\newcommand{\stringc}{String_{\C}}

\newcommand{\T}{\mathbb{T}}
\newcommand{\tap}{\tilde{a}'}
\newcommand{\tensor}{\otimes}
\DeclareMathOperator{\Thom}{Thom}
\newcommand{\trans}{T}

\newcommand{\lra}{\longrightarrow}

\newcommand{\NonEqBU}{\overline{BU}}

\newcommand{\piT}{\pi^{\T}}
\newcommand{\psb}[1]{[\![ #1 ]\!]}
\newcommand{\propp}{\sideset{}{''}\prod}
\newcommand{\propri}{\sideset{}{'}\prod}
\newcommand{\PT}{\Phi^{\T}}
\newcommand{\ptspace}{\ast}

\newcommand{\Q}{\mathbb{Q}}
\newcommand{\qqed}{\qed}

\DeclareMathOperator{\rank}{rank}
\newcommand{\restr}[1]{|_{#1}}

\newcommand{\sm}{\wedge}

\newcommand{\ten}[1]{\otimes #1}

\DeclareMathOperator{\spec}{spec}
\newcommand{\spp}{\sideset{}{''}\sum}
\DeclareMathOperator{\spf}{spf}
\newcommand{\spri}{\sideset{}{'}\sum}
\newcommand{\spX}{\mathfrak{D}}
\DeclareMathOperator{\st}{|}
\DeclareMathOperator{\Symm}{Symm}
\newcommand{\SectE}{\Gamma}
\newcommand{\ShfE}{\mathcal{F}}
\newcommand{\SpcE}{\mathbf{F}}
\newcommand{\stror}{\Sigma}

\newcommand{\Tate}[1]{#1_{\mathrm{Tate}}}
\newcommand{\tTate}{t}
\newcommand{\tm}{\tilde{m}}
\newcommand{\ta}{\tilde{a}}

\renewcommand{\th}{\text{th}}
\newcommand{\univ}{\mathcal{U}}
\newcommand{\uln}[1]{\underline{#1}}

\newcommand{\xla}[1]{\xleftarrow{#1}}
\newcommand{\xra}[1]{\xrightarrow{#1}}
\newcommand{\XT}{X^{\T}}
\newcommand{\XA}{X^{\A}}
\newcommand{\X}[1]{\mathfrak{X}_{#1}}

\newcommand{\Y}{\mathfrak{Y}}
\newcommand{\Z}{\mathbb{Z}}
\newcommand{\ann}{\mathrm{ann}}
\newcommand{\cN}{\mathcal{N}}
\newcommand{\elrnA}{E\langle \neg A \rangle}
\newcommand{\etnA}{\tilde{E}\langle \neg A \rangle}
\newcommand{\elrA}{E\langle A \rangle}
\newcommand{\elrB}{E\langle B \rangle}

\begin{document}

\title[Equivariant sigma genus]{Circle-equivariant classifying spaces
and the rational equivariant sigma genus}

\author[Ando]{Matthew Ando}
\address{Department of Mathematics \\
The University of Illinois at Urbana-Champaign \\
Urbana IL 61801 \\
USA} \email{mando@math.uiuc.edu}

\author[Greenlees]{J.P.C.Greenlees}

\address{Department of Pure Mathematics, Hicks Building, Sheffield S3
7RH. UK.}  \email{j.greenlees@sheffield.ac.uk}

\thanks{We began this project when we were participants in the program
\emph{New contexts for stable homotopy} at the Isaac
Newton Institute in Autumn 2002.  Some of the work was carried out while
Ando was a participant in the program ``Topological
Structures in Physics'' at MSRI.   Ando was partially supported by NSF
grants DMS-0306429 and DMS---0705233, and Greenlees by EPSRC grant
EP/C52084X/1}

\date{Version 4.2, June 2010}

\begin{abstract}
The circle-equivariant spectrum $\mstringc$ is the equivariant
analogue of the cobordism spectrum $\musix$ of stably almost complex
manifolds with $c_1=c_2=0$. Given a rational elliptic curve $\J$, 
there is  a ring $\T$-spectrum $\EJ$ representing 
the associated $\T$-equivariant elliptic cohomology \cite{ellT}. The 
core of the present paper is the construction, when 
$\J$ is a complex elliptic curve, of a map of ring
$\T$-spectra
\[
       \mstringc \to \EJ
\]
which is the rational equivariant analogue of the sigma orientation of
\cite{AHS:ESWGTC}. We support this by a theory of characteristic classes
for calculation, and a conceptual  description in terms of algebraic geometry. 
In particular, we prove a conjecture the first author made in \cite{Ando:AESO}.
\end{abstract}

\maketitle

\tableofcontents

\section{Introduction} \label{sec:introduction}

In this article we construct a circle-equivariant version of the sigma
orientation of Ando-Hopkins-Strickland \cite{AHS:ESWGTC}, taking
values in the equivariant elliptic cohomology $\EJ$ constructed by the
second author in \cite{ellT}, at least in the case of a complex
elliptic curve $C=\C/\Lambda$.
More precisely, we let $\T$ denote the circle group,
  $BU$ denote the classifying space for stable
$\T$-equivariant complex vector bundles, and $BSU$
denote the classifying space for stable $\T$-equivariant complex
vector bundles with trivial determinant. It turns out that $BSU$
is the cover of $BU$ trivializing the first Borel Chern class
$\ChB_1$. Now let $\bstringc$ be the cover
of $BU$ trivializing the first and second Borel Chern classes
$\ChB_1$ and $\ChB_2$; we call a virtual complex $\T$-vector bundle  
with a
lift of its classifying map to $\bstringc$ a  ``$\stringc$ bundle.''
Let $\mstringc$ be the
associated bordism spectrum.  We construct a map of ring $\T$-spectra
\[
     \mstringc\lra \EJ
\]
which specializes to the sigma orientation of Ando-Hopkins-Strickland
\cite{AHS:ESWGTC} in Borel-equivariant elliptic cohomology.

Our argument offers several improvements over the papers
\cite{AndoBasterra:WGEEC,Ando:AESO}, which
construct a canonical and natural Thom isomorphism  for $\stringc$
bundles (and their real analogues)
over compact $\T$-spaces in Grojnowski's equivariant elliptic
cohomology \cite{Grojnowski:Ell-new}.  For one thing, our use of the
spectrum $\EJ$ of \cite{ellT} entitles us to work directly with the
classifying spaces for equivariant bundles and their Thom spectra.
More importantly, we are able for the first time to give a simple and
conceptual formula for the Thom class of a $\stringc$-bundle.
Briefly, for $a\in C$ and a $\T$-space $X$, let
\[
     X^{a}= \begin{cases}
  X^{\T[n]} & \text{ if $a$ has finite order $n$}\\
  X^{\T}    & \text{ otherwise}.
\end{cases}
\]
The long exact sequence \eqref{eq:33} shows how to assemble
$\EJc^{*} (X)$ from the groups
\[
H_{\T}^*(X^a;\cO^{\wedge}_{\J,a})
\]
for $a\in \J,$   where the ``coordinate data'' of the elliptic curve
$\J$ are used to give $\cO^{\wedge}_{\J,a}$ the structure of an $H^{*}
(B\T)$-algebra (see Section \ref{sec:prop-equiv-ellipt}).
If $V$ is a virtual $\T$-equivariant vector bundle over $X$, then
the groups relevant for $\EJc^{*} (X^{V})$ are
\[
    H_{\T}^{*} ((X^V)^{a};\cO^{\wedge}_{\J,a}).
\]
The Thom class $\stror (V)$ near $a$ must then be a unit multiple of $ 
\Thom_{\T}(V^{a}),$
the Thom class of $V^{a}$ in
Borel-equivariant cohomology associated to the Weierstrass sigma
function (see \S\ref{sec:sigma-orientation}).  In order to assemble a
Thom class for $V$ as $a$ varies, we expect  that
\begin{equation} \label{eq:69}
     \stror (V)_{a} = \Thom_{\T} (V^{a}) e_{\T} (V/V^{a}),
\end{equation}
where $V/V^{a}$ is defined so that
\[
     V\restr{X^{a}} \iso V^{a}\oplus (V/V^{a}),
\]
and $e_{\T}$ denotes the Borel-equivariant Euler class associated to
the Weierstrass sigma function.

One of the virtues of our approach is that our Thom class is
given precisely by the formula \eqref{eq:69}; see \eqref{eq:20}
and Theorem \ref{t-th-sigma-equiv}.  The reader is invited to compare  
these
formulae with the formulae following Theorem 9.1 of \cite{Ando:AESO}
or (6.11), (6.17), and (6.23) of \cite{AndoBasterra:WGEEC} to get an  
idea of the
improvement \eqref{eq:69} represents.    It
remains to show that, if $V$ is a $\stringc$-bundle, then the proposed
Thom class $\stror (V)$ has  the necessary properties.  We do this in
Section \ref{sec:proof-prop-phi-A}.  The argument uses characteristic
classes which were, in some sense, the main discovery of
\cite{Ando:AESO}, but by working universally we give  a better account
of them and so put them to more effective use.

In Part \ref{sec:ellipt-cohom-bsu} we use our analysis of the
equivariant elliptic cohomology of $BSU (d)$ to prove a version of the
conjecture in \cite{Ando:AESO,Ando:Australia}, giving a conceptual  
construction
of the sigma orientation, for elliptic curves of the form
$\J=\C/\Lambda$.   To a $\T$-space $X$ we associate a sort of ringed  
space
  $\SpcE (X)$ over $C$ (actually over a diagram approximating $C$),
which determines $\EJc^{*} (X)$ in the
universal case. Associated to a complex vector bundle
$V$ over $X$, we construct a line bundle $\ShfE (X^{V})$ over $\SpcE
(X)$,  which determines $\EJc^{*} (X^{V})$ in the universal case.
The  fact that $\ShfE (X^{V})$ is a line bundle (i.e., locally trivial)
follows from the Thom isomorphism for ordinary Borel cohomology, and if
the line bundle is globally trivial, there is a Thom isomorphism in
elliptic cohomology.

Let $T$ be the usual maximal torus of $SU
(d)$, with Weyl group $W$,  and let $\cochars=\Hom (\T,T)$ denote the  
lattice
of cocharacters.  Looijenga
\cite{Looijenga:RootSystems} used the second Chern class to construct
a line bundle $\Loo = \Loo (c_{2})$
over $(\cochars\otimes_{\Z}\J)/W.$  The Weierstrass sigma function
determines a section $\sigma$ of $\Loo$, and so a trivialization
of $\Loo\otimes \cI$, where $\cI$ is the ideal sheaf of zeroes of
$\sigma$. As explained in
\cite{Ando:EllLG}, characters of representations of level $k$ of the  
loop group
$LSU (d)$ give sections of $\Loo^{k}.$

We show that a $\T$-equivariant $SU (d)$-bundle $V$ over $X$  
determines a
pull-back diagram
\[
\xymatrix{
{\Loo (V) \otimes \ShfE (X^{V})}
  \ar[r]
  \ar[d]
&
{\Loo \otimes \cI}
  \ar[d]
\\
{\SpcE (X)}
  \ar[r]^-{h}
  \ar@/_1pc/[u]_{\sigma (V)}
&
{(\cochars\otimes \J)/W,}
   \ar@/_1pc/[u]_{\sigma}
}
\]
where $\Loo (V) = h^{*}\Loo$ and $\sigma (V) = h^{*}\sigma.$
In particular $\sigma (V)$ is a trivialization of $\Loo (V)\otimes
\ShfE (X^{V}).$

It turns out that if $V_{0}$ and $V_{1}$ are two $\T$-equivariant $SU
(d)$-bundles over $X$ such that
\[
   \ChB_{2} (V_{0}-V_{1}) = 0,
\]
(i.e., so that $V_{0}-V_{1}$ admits a $String_{\C}$-structure) then  
there is an isomorphism
\[
        \Loo (V_{0}) \iso \Loo (V_{1}),
\]
so $\sigma (V_{0})\otimes \sigma (V_{1})^{-1}$ is a trivialization of
\[
       \frac{\Loo (V_{0})\otimes \ShfE (X^{V_{0}})}
            {\Loo (V_{1})\otimes \ShfE (X^{V_{1}})} \iso
            \ShfE (X^{V_{0}-V_{1}}).
\]
This is our Thom class; see Theorem \ref{t-th-anal-geom-sigma-orientation}.

We note that Jacob Lurie \cite{Lurie:Ellsurvey} has,  independently
of this paper and \cite{Ando:AESO}, announced a proof of the analogous
integral result for his oriented derived elliptic spectra.  Our  
results may be
viewed as a classical analogue his work, relating it to the invariant
theory of \cite{Looijenga:RootSystems} and  highlighting the role of
the Weierstrass sigma function.  In \cite{GeGr}, Gepner and the second
author explain the relationship between Lurie's
equivariant derived elliptic spectra and the $\T$-equivariant elliptic
spectra constructed in \cite{ellT}. We believe the two 
constructions of the equivariant sigma orientation are consistent on 
theories where both are defined.

Our work on this project has led us to a clearer understanding of the
relationship between the $\T$-equivariant elliptic cohomology theories
constructed by Grojnowski \cite{Grojnowski:Ell-new} and Greenlees \cite 
{ellT}.
In both cases, $\EJc^{*} (X)$ is assembled from the groups
$H_{\T}^*(X^a;\cO^{\wedge}_{\J,a})$
for $a\in \J.$  In order to make sense of this expression, one must
give $\cO^{\wedge}_{\J,a}$ the structure of an $H^{*} (B\T)\iso \Q[z]$-algebra
\footnote{The generator of $H^2(B\T)$ is denoted $z$ here because we  
are thinking
of it as a complex function on $\C$. When we think of it as the first  
Chern class
of the canonical bundle we write $c$ for the same generator}.
Grojnowski does this for a complex elliptic curve in the form
$\C/\Lambda$, using the covering
\begin{equation} \label{eq:30}
     \C \to \C/\Lambda,
\end{equation}
the structure of $\O_{\C}$ as an $H^{*} (B\T)$-algebra, and  
translation in the elliptic curve.   One of the points of \cite{ellT}
is that we may construct an elliptic cohomology theory given 
suitable coordinate data on the elliptic curve; for each $n\geq 1$ we 
need a function $t_n$ vanishing to the first order on the points of 
$C$ of exact order $n$.
We discuss this in 
Section \ref{sec:prop-equiv-ellipt}, particularly before Lemma
\ref{t-le-completion} and in Remark \ref{rem-1}.


The paper is divided into three parts.

Part 1 is about equivariant classifying spaces and equivariant  
characteristic
classes in general, and much of it holds integrally.
In Section \ref{sec:classifying-spaces-g}, we begin the study of the
classifying spaces for equivariant vector bundles which arise in this
work.  In Sections \ref{sec:characteristic} and \ref{sec:char-class-circle} 
we discuss characteristic classes for these bundles.  In
Section \ref{sec:cohomology-covers-bu}  we use these characteristic  
classes
to describe the Borel-equivariant ordinary cohomology of our
classifying spaces. We make repeated  use of a Universal Coefficient
Theorem for Borel cohomology, which we discuss in the appendix.

Part 2 focuses on elliptic cohomology, introduces the sigma orientation
and establishes the Thom isomorphism.
In Section \ref{sec:prop-equiv-ellipt} we recall from \cite{ellT} the
properties of equivariant elliptic cohomology associated to an  
elliptic curve
over a $\Q$-algebra which we need for our work. We then turn to complex
elliptic curves. In Section \ref{sec:edale} we  recall the basic facts  
about the
Weierstrass sigma function, and use it to give the formula for the Thom
isomorphism over $\T$-fixed spaces. The behaviour for points
with finite isotropy is stated in Proposition \ref{t-pr-phi-A-indep-lift},
and proved in Section \ref{sec:proof-prop-phi-A}. Together, these give
the Thom isomorphism: the main result is Theorem \ref{t-th-sigma-equiv}.

In Part 3 we reformulate the results of Part 2 in geometric terms.
We explain our Thom isomorphism using
the analytic geometry of the elliptic curve $C$ and the invariant
theory of \cite{Looijenga:RootSystems}, proving the conjecture of
\cite{Ando:AESO} in this case.  We also rephrase some of
these ideas in terms of the \emph{algebraic} geometry of $C$, avoiding  
the
use of analytic methods.  We hope
that these ideas will eventually lead to an algebraic version of our  
results.
Section \ref{sec:ellipt-cohom-sheav} gives a convenient sheaf theoretic
formulation of the separation of behaviour over $\T$-fixed points
and generic points of the curve from points with finite isotropy and
torsion points on the curve. In Section \ref{sec:analyt-geom-sigma}
we give the geometric
interpretation of the situation over the $\T$-fixed points of $BSU(d)$,
  and in Section  \ref{sec:ptsoffiniteorder} we extend this to all of  
$BSU(d)$. Finally
in Section \ref{sec:alg-geom-sigma} we give a moduli interpretation
in terms of divisors.

The appendix describes a universal coefficient theorem for Borel
homology and cohomology.

\part{Equivariant classifying spaces and characteristic classes}

In this part
we discuss equivariant classifying spaces and characteristic classes
from several different points of view.  In Section \ref 
{sec:classifying-spaces-g},
we discuss the classifying spaces both via moduli and through specific
models. In Section \ref{sec:characteristic}
we discuss characteristic classes via the splitting principle and  
formal roots.
In Section \ref{sec:char-class-circle}
we apply the earlier sections to calculate the cohomology of the first
few covers of $BU$.

\section{Classifying spaces for equivariant vector bundles}
\label{sec:classifying-spaces-g}

Let $G$ be a compact Lie group.  In this section, we review various
aspects of the classifying spaces for $G$-equivariant complex vector
bundles.
Until the end of Section
\ref{sec:characteristic},  we allow $G$ to be an arbitrary compact Lie  
group,
but our applications use the special case $G=\T$, and we
specialize to that case when it is convenient to do so.  Much of this
material is well-known; see for example \cite{MR1413302}.

\subsection{The classifying space for equivariant complex vector
bundles of finite rank} \label{sec:class-spac-equiv}

Just as in the non-equivariant case, we may pass between $U (n)$-free
$G\times U (n)$-spaces, or $G$-equivariant principal $U (n)$-bundles,
and $G$-equivariant complex vector bundles of rank $n$.  This gives
two models for their classifying $G$-space $BU (n)$.

On the one hand, $BU (n)$ can be constructed as the
quotient $EU (n)/U (n)$, where $EU(n)$ is a $G\times U (n)$ space with
the property that, for all $K\subset G\times U (n)$,
\begin{align*}
    EU (n)^{K} & \heq  \ptspace \text{ if } K\cap U (n) = 1  \\
    EU (n)^{K} & = \emptyset  \text{ otherwise.}
\end{align*}
On the other hand, $BU (n)$
can be modelled as the Grassmannian $\Gr_{n}
(\univ)$ of $n$-dimensional subspaces of a complete complex
$G$-universe.  Thus
\[
BU(n)= EU(n)/U(n)\heq\Gr_n(\univ).
\]
We have omitted the $G$ from the notation for
$BU (n)$, because for $H\subseteq G$, the $H$-space underlying $BU (n) 
$ is the classifying
$H$-space for $H$-equivariant $U (n)$-bundles, as one can check using
either description of $BU (n).$

\begin{Remark}  \label{rem-7}
Note that if $X$ is a $G$-space, and $H$ is a subgroup, then $N_{G}H/H 
$ acts
on $X^{H}$. In particular, if $G=\T$ and $H=\A$ is a finite subgroup,
then $\T/\A$ acts on $BU (n)^{\A}$.
\end{Remark}

\subsection{The classifying space for stable bundles}
\label{sec:class-space-stable}

We will need to have a clear understanding of the stabilization
process.  For this, we let $\univ$ denote a complete complex
$G$-universe, and $V,W,X, \ldots $ denote finite dimensional
subrepresentations of dimensions $v,w, x, \ldots$.

Let $BU(V)=\Gr_v(\univ \oplus V)$, and let $\tU_V$ be the
tautological bundle over this. These spaces form a direct system with
structure maps
\[
BU(V)=\Gr_v(\univ \oplus V)\stackrel{\oplus W} \lra \Gr_{v+w}(\univ
\oplus V \oplus W) =BU(V\oplus W).
\]
Let
\[
BU=\colim_W BU(W).
\]
We let $\tU$ denote the universal bundle over $BU$, so that $\tU
|_{BU(W)}=\tU_W-W$.  The $G$-space $BU$ classifies stable vector
bundles of virtual dimension 0, and the $G$-space $BU \times \Z$
classifies arbitrary stable vector bundles.

\subsection{Fixed points.}

It is
straightforward to identify the $H$-fixed points of $BU(n)$, $BU (W)$,
and $BU$. 
We do this two
ways, by analyzing the Grassmannian model and by analyzing the
homotopy functor represented by $BU (n)^{H}.$

If $H$ is a compact Lie group, we write $H^{\vee}$ for the set of
isomorphism classes of simple (complex) representations of $H$, so
that if $A$ is an abelian group then $A^{\vee}\iso\Hom (A,\T)$ is its
group of characters.

For a representation $V$ of $H$ we make the following definitions.  We
write $U(V)$ for the group of unitary automorphisms of $V$, and
we write $Z (V)$ for the centralizer of $H$ in $U (V)$, so
\[
        Z (V) = \Aut_{H} (V) = \{x\in U(V) | x h x^{-1} = h \text{ for
all } h\in H \}.
\]
For $S\in H^{\vee}$, we define $V_{S}$ to be the $S$-isotypical
summand, so
\[
      V\iso \bigoplus_{S\in H^{\vee}} V_{S},
\]
and we set
\[
       d_{S,V} = \dim \Hom (S,V).
\]
By Schur's Lemma we have an isomorphism of $H$-modules
\[
    \Hom_{H} (S,V) \otimes S \iso V_{S},
\]
where $H$ acts trivially on $\Hom_{H} (S,V),$ and an isomorphism of  
groups
\begin{equation}\label{eq:3}
     \Aut_{H} (V) \iso \prod_{S\in H^{\vee}}\Aut (\Hom_{H} (S,V)) \iso
\prod_{S\in H^{\vee}} U (d_{S,V}).
\end{equation}
The set of isomorphism classes of $n$-dimensional representations of
$H$ is
\[
  \Hom (H,U (n))/\text{conjugacy}.
\]
It is convenient to choose a set of representatives
\[
\Rep_{n} (H) \subseteq \Hom (H,U (n)).
\]

An $H$-fixed point of $\Gr_{w} (\univ \oplus W)$ is an $H$-module of
rank $w$, and so we have the function
\[
    BU (W)^{H} = \Gr_w (\univ\oplus W)^{H} \to \Rep_{w} (H)
\]
which sends a point to the representative of its isomorphism class.    
The
function is surjective since $\univ$ is complete, and the
codomain is discrete, so for $V\in \Rep_{w} (H)$ we define
\[
\Gr_V^H(\univ \oplus W) \subseteq\Gr (\univ\oplus W)^{H}
\]
to be the component mapping to $V$.  Specifying a point of
$\Gr_{V}^{H} (\univ\oplus W)$ is equivalent to specifying a
point of $\Gr_{V_{S}}^{H} (\univ_{S}\oplus W_{S})$ for each $S\in
H^{\vee}$.

\begin{Proposition} \label{t-pr-BU-n-H}
For any compact Lie group $H$ there is an equivalence of
nonequivariant spaces
\begin{equation}\label{eq:2}
BU (n)^H \heq \coprod_{V\in \Rep_{n} (H)} BZ(V)= \coprod_{V\in
\Rep_{n} (H)} \prod_{S \in H^{\vee}} BU (d_{S,V}).
\end{equation}
For the Grassmannian $BU (W)^{H},$ we have
\[
BU(W)^H=\Gr_w(\univ \oplus W)^H=\coprod_{V \in \Rep_w(H)}\Gr_V^H(\univ
\oplus W)
\]
and
\[
\Gr_V^H(\univ \oplus W) = \prod_{S\in H^{\vee}} \Gr_{V_S}^H(\univ_S
\oplus W_S) \heq \prod_{S \in H^{\vee}} BU(\Hom_H(S,V)).
\]
\end{Proposition}


\begin{proof}[First proof] The displayed equalities for the  
Grassmannian model
give proofs.  The only equivalence which has not already been
spelled out is the last.  Over
\[
     \Gr_{V_{S}}^{H} (\univ_{S}\oplus W_{S})
\]
we have, forgetting the action of $H$, a contractible principal $\Aut  
(V_{S})$
bundle.  The sub-group of automorphisms commuting with $H$ is just $U
(\Hom_{H} (S,V))$.
\end{proof}

\begin{proof}[Second proof]  Since the $H$-space underlying $BU (n)$
classifies $H$-equivariant principal $U (n)$-bundles, it is clear that
$BU (n)^{H}$ classifies $H$-equivariant principal $U
(n)$-bundles over $H$-fixed spaces $Z.$  It suffices to consider one
component at a time, so let us suppose we are given
such a bundle $\pi: P\to Z$, with $Z$ connected.

Recall that
\[
    \Aut (P/Z)\iso \Gamma (P\times_{U (n)} U (n)^{c} \rightarrow Z),
\]
where $U (n)^{c}$ denotes $U(n)$ with the adjoint action.
It follows that an action of $H$ on $P/Z$ is given
by a section $s$ of
\begin{equation}\label{eq:1}
      P\times_{U (n)} \Hom (H,U (n)^{c}) \to Z.
\end{equation}
The set $\Hom (H,U (n)^{c})$ is discrete, and so the function
\[
m: Z\xra{s} P\times_{U (n)} \Hom (H,U (n)^{c}) \xra{}
* \times_{U (n)} \Hom (H,U (n)^{c})
\xla{\iso} \Rep_{n} (H)
\]
is locally constant, and so constant.

Let $Z (m)$ be the centralizer
\[
Z (m) = \{x\in U (n) | xm (h) x^{-1} = m (h) \text{ for all }h\in H\},
\]
and let
\[
     Q = \{ p\in P | s (\pi (p)) = \overline{(p,m)} \},
\]
where $\overline{(p,m)}$ denotes the class in the Borel
construction.  Then
\[
\pi\restr{Q}: Q \rightarrow Z
\]
is a principal $Z (m)$ bundle, classified by a map
\[
    f: Z \to BZ (m).
\]
It follows that
\[
BU (n)^H \heq \coprod_{V\in \Rep_{n} (H)} BZ(V),
\]
and the more detailed description in \eqref{eq:2} follows from the
isomorphism~\eqref{eq:3}.
\end{proof}

\begin{Remark}\label{rem-5}
In \S\ref{sec:reduct-splitt-princ} we use this analysis of $BU
(n)^{H}$ to give a splitting principle for $\T$-equivariant vector  
bundles.
\end{Remark}

Passing to limits we find the fixed points in the stable case.  We
define $JU (H)$ to be the ideal of virtual representations of rank
$0$ in the complex representation ring $R (H).$

\begin{Proposition} \label{t-pr-BU-H}
There is an equivalence of nonequivariant spaces
\begin{equation}\label{eq:8}
BU^H \simeq JU(H) \times \prod_{S \in H^{\vee}}BU,
\end{equation}
where the product is topologized as the direct limit of the finite
products.
\end{Proposition}

\begin{proof}[First proof] The stabilization map is
\[
BU(V)^H=\coprod_{U \in \Rep_v(H)}\Gr_U^H(\univ \oplus V)
\stackrel{\oplus W}\lra \coprod_{U' \in \Rep_{v+w}(H)}\Gr_{U'}^H(\univ
\oplus V\oplus W)=BU(V\oplus W)^H.
\]
Thus the components of $BU^{H}$ are labelled by virtual representations
of rank $0$, the component $\Gr_{U}^{H} (\univ\oplus V)$ being labelled
by $U-V$.  The stabilization map is the product over $S\in H^{\vee}$ of
\[
BU (d_{S,U}) \iso \Gr_{U_S}^H(\univ_S \oplus V_S)  \rightarrow
\Gr_{U_{S}\oplus W_{S}} (\univ_{S}\oplus V_{S}\oplus W_{S}) \iso
BU (d_{S,U+W}).
\]
For each factor, the colimit is $BU.$
\end{proof}

\begin{proof}[Second proof] We classify virtual $H$-equivariant
bundles $V$ on $H$-fixed spaces $Z$.   It suffices to consider one
component at a time, and so we suppose that $Z$ is connected.  Schur's
Lemma provides a decomposition
\[
  V \iso \bigoplus_{S\in H^{\vee}}    \Hom_{H} (S,V) \otimes  
S,
\]
where now
\[
    \Hom_{H} (S,V)
\]
is a virtual bundle of rank $d_{S,V}$ (say).
We then have the map
\begin{equation} \label{eq:51}
   f: Z \xra{}\prod_{S \in H^{\vee}} BU
\end{equation}
which on the $S$ factor classifies the virtual
bundle of rank $0$
\[
\Hom_{H} (S,V)-d_{S,V}\eps,
\]
where $\eps$ is the trivial complex line bundle of rank one, with
trivial $H$-action.

If $\tU_{S}$ denotes the tautological bundle of rank $0$ over the
$S$ factor in \eqref{eq:51}, then
\[
    f^{*} \left(\sum \tU_{S}\otimes S\right) \iso V - \sum d_ 
{S,V}S.
\]
To recover $V$, then, we must add the element $\sum
d_{S,V}S\in R (H)$.  This shows that
\[
     (BU\times \Z)^{H} \iso R (H) \times  \prod_{S\in H^{\vee}}
     BU,
\]
with the universal bundle  over the $(\sum_{S} d_{S}S)$
summand being
\[
       \sum_{S} (\tU_{S}+d_{S}) \otimes S.
\]
$V$ has virtual dimension zero if and only if
\[
\sum_{S}
d_{S,V} \rank S = 0,
\]
and so
\[
     BU^{H}\iso JU (H) \times \prod_{S\in H^{\vee}}BU.
\]
\end{proof}

\subsection{Classifying spaces for $SU$-bundles}
\label{sec:class-space-stable-SU}

Next we consider the classifying space $BSU(n)$ of
$n$-dimensional bundles with determinant $1.$  This can be constructed
as $ESU(n)/SU(n)$ where $ESU(n)$ is the universal $SU(n)$-free $G
\times SU(n)$-space.  Alternatively, there is a fibration
\[
BSU(n) \lra BU(n) \lra BU(1)
\]
of $G$-spaces, where the map $BU(n) \lra BU(1)$ classifies the
determinant.  The $H$-fixed points can be calculated in the same
manner as in Proposition \ref{t-pr-BU-n-H}.  Let
\[
\SRep_{n} (H)\subset \Hom (H,SU (n))
\]
be a set of representatives for $\Hom (H,SU (n))/\text{conjugacy}$; and
for $V\in \SRep_{n} (H)$ let $Z (V)$ be its centralizer in $SU (n)$.
The analysis leading to Proposition \ref{t-pr-BU-n-H} gives the  
following.

\begin{Proposition}\label{t-pr-BSU-n-H}
For any compact Lie group $H$ there is an equivalence
\begin{equation}\label{eq:26}
BSU (n)^H \heq \coprod_{V\in \SRep_{n} (H)} BZ(V).
\end{equation} \qed
  \end{Proposition}

Once again we may form the stable classifying space
$BSU$ as a direct limit
\[
BSU\eqdef\colim_n BSU(n),
\]
where the limit is now formed over addition of a cofinal collection of
representations with determinant the trivial 1-dimensional
representation $\eps$, such as those of form $V\oplus V^*$.  Again
there is a fibration
\[
BSU \lra BU \lra BU(1).
\]

Taking $H$-fixed points we have
\[
BSU^H \lra BU^H \lra BU(1)^H.
\]
If $\buoneHeps$ is the component of $BU(1)^H$ corresponding
to the trivial representation,
\[
JU_{2} (H) = \{V \in JU (H) | \det V \iso \eps \}
\]
is the subgroup of $JU (H)$ consisting of virtual representations with
trivial determinant, and $BU^H_\eps$ is the set of components of $BU^{H}$
corresponding to representations with determinant $\eps$, then
we have a fibration
\begin{equation} \label{eq:64}
BSU^H \lra BU^H_\eps \lra \buoneHeps
\end{equation}
with connected base, and an equivalence
\begin{equation}\label{eq:63}
BU^H_\eps \heq JU_2(H) \times \prod_{S\in H^{\vee}} BU.
\end{equation}
Again, the components of $BSU^H$ are all
equivalent, and, taking components of zero, there is a fibration
\[
BSU^H_0 \lra BU^H_0 \lra BU(1)^H_{\eps}
\]
of connected spaces.

\subsection{The tower over $BU \times \protect \Z$.}
\label{subsec:coverdiagram}

In the next two sections we study characteristic classes for
equivariant vector bundles, in light of the preceding analysis of their
classifying spaces.   One  reason to do so is better to understand the
spaces $\bulrp{2k}$ over $BU\times \Z$ defined by the  vanishing of
the Borel Chern classes $\ChB_{0}, \ChB_{1},$ and $\ChB_{2}.$  It is
the Thom spectrum associated to  $\bulrp{6}$ which maps to elliptic
cohomology.

It is perhaps surprising that the vanishing of Borel Chern classes
plays such an important role in the relationship to elliptic
cohomology.  We note that the spaces $\bulrp{2k}$ also occur as
representing spaces for the equivariant version of connective
$K$-theory constructed in \cite{MR2027899}; see
\S\ref{sec:spectrum-mstringc}.   This equivariant version of
connective $K$-theory is complex orientable, and its coefficient ring
classifies multiplicative equivariant formal group laws for products
of two topologically cyclic groups (and in particular for the circle
and all its subgroups).

Nonequivariantly, $BU=BU\langle 2\rangle$ is the 1-connected cover of
$BU \times \Z$, $BSU=BU\langle 4\rangle$ is the fibre of
\[
        BU \xra{c_{1}} K (\Z,2),
\]
and $\bstringc =BU\langle 6\rangle$ is the fibre of the second Chern
class
\[
        BSU \xra{c_{2}} K (\Z,4).
\]
Borel cohomology classes, that is elements of $H^{n} (X\times_{G}EG),$
correspond to $G$-maps
\[
     f: X\lra \map (EG,K (\Z,n)).
\]
We define spaces $\bulrp{2k}$ by the following diagram, in which
the indicated horizontal arrows are Borel Chern classes, and each
vertical arrow is the fibre of the following horizontal arrow.
\begin{equation}\label{eq:4}
\xymatrix{ {\bstringc} \ar@{=}[r] & {\bulrp{6}} \ar[d]
\\
{BSU} \ar[r]^-{\heq} & {\bulrp{4}} \ar[r]^-{\ChB_{2}} \ar[d] & {\map
(EG,K(\Z,4))}
\\
{BU} \ar@{=}[r] & {\bulrp{2}} \ar[r]^-{\ChB_{1}} \ar[d] & {\map (EG,K
(\Z,2))}
\\
{BU\times \Z} \ar@{=}[r] & {\bulrp{0}} \ar[r]^-{\ChB_{0}} & {\map
(EG,K(\Z , 0)).}  }
\end{equation}
We have used the notation $\bulrp{2k}$ instead of $BU\langle 2k
\rangle$ because the spaces in question are not equivariantly connected.
We show in \S\ref{sec:comp-bsu-bulrp4} that $BU$ and $BSU$ occur as
indicated in \eqref{eq:4}.  We {\em define} $\bstringc$ to be $\bulrp 
{6}.$

To analyze $H$-fixed points, we use the equivalence
\begin{equation} \label{eq:6}
\map (EG, K(\Z, n))^H\simeq \map (BH, K(\Z,n)) \simeq
\prod_{i=0}^nK(H^i(BH),n-i).
\end{equation}

\begin{Proposition}\label{t-pr-A-fixed-fibration-diag}
Taking $H$-fixed points in the diagram \eqref{eq:4} yields a diagram
\[
\xymatrix@!=1pc{ {\bstringc^H} \ar[d]
\\
{BSU^H} \ar[rr]^-{\{c_2^0,c_2^2, c_2^3, c_2^4\}} \ar[d] & &
**[r]{K(H^0(BH),4)\times K(H^2(BH),2)\times K(H^3(BH),1)\times K(H^4 
(BH),0)}
\\
{BU^H} \ar[d] \ar[rr]^-{\{c_1^0, c_1^2\}} & & **[r]{K (H^0(BH),2
)\times K(H^2(BH),0)}
\\
{BU^H \times \Z} \ar[rr]^-{c^0_0} & & **[r]{K (H^0(BH),0)} }
\]
in which again each vertical arrow is the fibre of
the following horizontal one. \qed
\end{Proposition}

We are particularly interested in the case that $G=\T$ is the circle,  
and
$H=\A\subseteq \T$ is a closed subgroup.  Such groups $\A$ have integral
cohomology only in even degrees, so the factor $K(H^3(BA),1)$ is  
trivial,
and $c_2^3=0$. We will describe the resulting
  maps $c_k^i$ as characteristic classes in
Lemma \ref{t-le-equiv-ch-classes-for-covers}.

\subsection{Universal bundles}
\label{subsec:univbundles}

Because of the large number of universal spaces and universal bundles  
over
them, we have found it useful to include a brief summary of our  
conventions.

We write $\tUd$ for the universal $d$-plane  bundle over $BU(d)$,
and $\tSUd$ for its restriction to $BSU(d)$, which is the universal
$SU(d)$-bundle.
We write $\tU$ for the universal stable bundle over $BU$, which is
a virtual bundle of virtual rank $0$.  Its restriction $\tSU$ to $BSU$
is universal for virtual bundles of rank $0$ and determinant $1$.  Its
pull-back $\tStringc$ to $\bstringc$ is  universal for
$\stringc$-bundles. We have used distinct notation
for $\tStringc$ because of the central role it plays in our work.

In the nonequivariant case, there is a single preferred map
\[
     BSU (d) \to BSU
\]
classifying the virtual bundle $\tSUd - d$, and $BSU$ is the colimit
of these maps as $d$ varies.  In the equivariant
case, we have for each complex represention $W$ of $\T$ with
determinant $1$ the maps
\begin{align*}
i_{W}':  BU (d) &\rightarrow  BU \\
i_{W}: BSU (d) & \rightarrow BSU
\end{align*}
classifying $\tUd-W$ and $\tSUd-W$ respectively.   We let $\bstringc
(W)$ be the pull-back in the diagram
\begin{equation}\label{eq:61}
\begin{CD}
\bstringc (W) @>>> BSU (d) @>>> BU (d) \\
@VVV @VV i_{W} V @VV i_{W}' V \\
\bstringc @>>> BSU @>>> BU.
\end{CD}
\end{equation}
We let $\tStringcW$ be the virtual bundle over $\bstringc (W)$ such that
over the diagram  \eqref{eq:61} we have a diagram of pulled-back
virtual bundles
\[
\begin{CD}
\tStringcW @>>> \tSUd - W @>>> \tUd - W \\
@VVV @VVV @VVV \\
\tStringc @>>> \tSU @>>> \tU.
\end{CD}
\]

As we discuss in \S\ref{sec:class-space-stable} and
\S\ref{sec:class-space-stable-SU}, as $d$ and
$W$ vary, we obtain a
compatible filtered direct system of $BSU(d)$ approximating $BSU$.  It
follows that as $d$ and $W$ vary we obtain a diagram of spaces
$\bstringc (W)$ approximating $\bstringc$.

\section{Characteristic classes of equivariant bundles}
\label{sec:characteristic}

In this section we briefly discuss characteristic classes for a
general compact Lie group of equivariance, and we show that $BSU\heq
\bulrp{4}.$ In the next section we analyze more closely the case of a
circle.

\subsection{Nonequivariant Chern classes}

We write $c_{i}$ for the usual Chern class in $H^{2i}BU$, so
\[
H^*(BU)=\Z [c_1, c_2, \ldots ].
\]
We write
\[
\Chdot (V) = 1+c_1(V) +c_2 (V) +\cdots
\]
for the total Chern class, and recall that it is exponential in the
sense that
\[
\Chdot (V \oplus W)=\Chdot (V) \cdot \Chdot (W),
\]
so we may extend $\Chdot$ to virtual vector bundles by the formula
\[
\Chdot (V-W) = \Chdot (V) \Chdot (W)^{-1}.
\]

\begin{Remark}
\label{c2addonkerc1} It is convenient to record the behaviour of $c_1$
and $c_2$ on a difference of actual bundles. It is immediate that
$c_1$ is additive, so that
$$c_1(U - V)=c_1(U)-c_1(V).$$ For $c_2$ there is a correction term:
$$c_2(U-V)=c_2(U)-c_2(V)-c_1(V)c_1(U-V), $$ but this simplifies to
additivity when $c_1(U-V)=0$.
\end{Remark}

Let $X$ be a $G$-space, and suppose that $V$ is a complex $G$-bundle
$V$ of rank $n$ over $X$.   If it happens that $X$ is $H$-fixed,
then we have an isomorphism of $H$-bundles
\[
       V \iso \bigoplus_{\alpha\in H^{\vee}} \Hom
(\alpha,V)\otimes \alpha,
\]
where on the one hand $\alpha$ is trivial as a non-equivariant bundle
over $X$ and on the other $\Hom (\alpha,V)$ denotes maps of
$H$-representations, and  carries a trivial
$H$-action.  We may define Chern classes $c_{i}^{\alpha} (V)$ by the  
formula
\begin{equation} \label{eq:66}
     c_{i}^{\alpha} (V) \eqdef c_{i} (\Hom (\alpha,V)).
\end{equation}

If $V$ is a virtual complex vector bundle of rank $0$ over $X$, then
it is classified by a map
\[
      [V]: X\to BU.
\]
If $H$ acts trivially on $X$, then we write
\[
      [V,H]: X\to BU^{H}
\]
for the indicated factorization.  The decomposition
\[
     BU^{H}\iso JU (H)\times \prod_{\alpha\in H^{\vee}} BU,
\]
of Proposition \ref{t-pr-BU-H} gives an isomorphism
\begin{equation} \label{eq:65}
     H^{*} (BU^{H}) \iso \prod_{V\in JU (H)}
\Z\psb{c_{1}^{\alpha},c_{2}^{\alpha},\dotsc | \alpha \in H^{\vee}}.
\end{equation}
The notation in \eqref{eq:65} is consistent with the notation in
\eqref{eq:66} in the sense that
\[
c_i^{\alpha}(V) = [V,H]^*(c_i^{\alpha});
\]
where on the right $c_i^{\alpha}$ is
taken from the appropriate factor of $BU^H$.  The double brackets
in \eqref{eq:65} refer to the completion arising from
the fact (see Proposition \ref{t-pr-BU-H}) that the topology on the
product in the description of $BU^H$  is the direct limit of
finite products.   So the infinite sum
\[
     \sum_{\alpha\in \T^{\vee}} c_{1}^{\alpha} \in H^{2} (BU^{\T})
\]
is allowed, but the sum
\[
      \sum_{n} c_{n}^{\alpha}
\]
makes no sense.

\subsection{Borel Chern classes} \label{sec:borel-chern-classes}

The $G$-Borel construction on $V$ is a virtual complex
vector bundle $V \times_{G} EG$ over $X\times_{G} EG$,
classified by
\begin{equation} \label{eq:72}
[[V ]]:X \times_{G} EG \lra BU.
\end{equation}

The Borel Chern classes of
$V$ are defined to be
$$\ChB_i(V)\eqdef[[V]]^*(c_i).$$

If $X$ is $H$-fixed, then there is a standard way to relate the Borel
Chern classes to the $c_{i}^{\alpha}$.  Notice that there is an
isomorphism
\begin{equation} \label{eq:52}
    B: \Rep_{1} (H) = \Hom (H,\T)\xra{\cong} [BH,B\T] = H^{2} (BH),
\end{equation}
via which we have, for $\alpha\in H^{\vee}$,
\begin{equation} \label{eq:77}
    \ChB_{1} (\alpha) = B\det\alpha.
\end{equation}

\begin{Lemma} \label{t-le-c-1-B-and-c-1-alpha}
If $V$ is an equivariant $G$-bundle over an $H$-fixed space $X$,
then in $H^{2} (X\times BH)$ we have
\begin{equation}\label{eq:5}
\ChB_1 (V)=\sum_{\alpha\in H^{\vee}} \rank (\alpha)
c_1^{\alpha}(V)+ \rank (\Hom (\alpha,V)) B\det(\alpha).
\end{equation}
\end{Lemma}

\begin{proof}
By reducing to the universal case $X=BU (n)^{H}$, we may assume that
$H^{1}X=0$.   Recall that we have the isomorphism of $H$-bundles over
$X$
\[
     \bigoplus_{\alpha\in H^{\vee}} \Hom (\alpha,V)\otimes \alpha
\iso V.
\]
This gives
\[
   V\times_{H}EH  \iso \bigoplus_{\alpha\in H^{\vee}} \Hom
(\alpha,V)\otimes (\alpha\times_{H} EH)
\]
since $H$ acts trivially on $\Hom (\alpha,V)$.  Taking determinants
gives
\begin{equation} \label{eq:76}
     \det (V\times_{H}EH) \iso \prod_{\alpha\in H^{\vee}} \det (\Hom
(\alpha,V) \otimes (\alpha\times_{H} EH)).
\end{equation}
By definition, $c_{1}^{\alpha} (V) = c_{1} (\Hom (\alpha,V)),$
so the result follows from \eqref{eq:77}, \eqref{eq:76}, and the formula
\[
       c_{1} (V\otimes W) = \rank (V) c_{1} (W) + c_{1} (V) \rank (W).
\]
\end{proof}

\subsection{$BU$ as a split $G$-space}
\label{sec:bu-as-split}

The Borel Chern classes have another less familiar description which
will be useful in \S\ref{sec:spectrum-mstringc}.
Let $\NonEqBU$ be the stable Grassmannian associated to the trivial
$G$-universe $\univ^{G}$, so that $\Z\times \NonEqBU$ is a
representing space for non-equivariant $K$-theory.  The inclusion
\[
     \univ^{G} \to \univ
\]
induces an equivariant map
\[
      \eta: \NonEqBU \to BU,
\]
which is easily seen to be a non-equivariant weak equivalence: this is
a space-level expression of the fact that the equivariant complex
$K$-theory spectrum is a split ring spectrum.
It follows that the induced map
\[
     \eta^{*}: H^{*}_{G} (BU) \to H^{*}_{G} (\NonEqBU)\iso H^{*}  
(\NonEqBU\times BG)
\]
is an isomorphism.

If $\tU$ denotes the tautological bundle over $BU$, then the map
$[[\tU]]$ in  \eqref{eq:72} can be regarded as a map
\[
    [[\tU]]:   BU\times_{G} EG \to \NonEqBU,
\]
and it is easy to check that the diagram
\[
\xymatrix{
{\NonEqBU \times BG}
  \ar[dr]^-{\text{projection}}
  \ar[d]
\\
{BU\times_{G} EG}
   \ar[r]
&
{\NonEqBU}
}
\]
commutes up to homotopy.  Thus we have the following.

\begin{Proposition}\label{t-pr-borel-and-splitting}
The Borel Chern classes $\ChB$ are uniquely characterized by the fact
that, under the splitting
\[
    \eta: \NonEqBU\to BU,
\]
they pull back to the ordinary Chern classes.  That is, for each $k$
we have
\[
     \eta^{*} \ChB_{k} = c_{k}
\]
in $H^{2k} (\NonEqBU\times BG).$ \qed
\end{Proposition}

\subsection{Comparison of $BSU$ and $\bulrp{4}$}

\label{sec:comp-bsu-bulrp4}

We explain how the $G$-spaces $BU\times \Z, BU $ and $BSU$ fit into a
diagram as displayed in \S\ref{subsec:coverdiagram}.  Since $\map (EG,
K(\Z,0))\simeq K(\Z, 0)$, $\ChB_{0}$ is a bijection on components, and
so the space $BU$ is the fibre of $c_0^B$.

For the next stage, observe that if $L$ is the tautological line
bundle over $BU (1)$, then
\[
     \ChB_{1} (L) \in H^{2} (BU (1)\times_{G} EG) \iso [BU (1),\map
(EG,K (\Z,2))].
\]
The definition of the first Borel Chern class implies that
the diagram
\begin{equation} \label{eq:7}
\xymatrix{{BU} \ar[r]^-{\det} \ar[dr]_{\ChB_{1}} &
{BU (1)} \ar[d]^{\ChB_{1} (L)}
\\
& {\map (EG,K (\Z,2))}
}
\end{equation}
commutes.

\begin{Proposition}
\cite{StrAfg}
\label{t-pr-bsu-bulrp4}
The $G$-map
\[
  BU (1) \lra \map (EG,K (\Z,2))
\]
corresponding to $\ChB_{1} (L)$ is a weak
equivalence, and so induces a weak equivalence
\[
     BSU \heq \bulrp{4}.
\]
\end{Proposition}

\begin{proof}
Proposition \ref{t-pr-BU-n-H} shows that, for each
compact subgroup $H\subseteq G$,
\[
  BU (1)^{H}\simeq \Rep_{1} (H)\times K (\Z,2).
\]
At the same time, since $H$ is compact, $H^{1} (BH)=0$, and we have
\[
\map (EG,K (\Z,2))^{H} \simeq K (H^{2} (BH),0) \times K (\Z,2).
\]
In terms of these isomorphisms, Lemma \ref{t-le-c-1-B-and-c-1-alpha}
shows that
\[
(\ChB_{1})^{H} = B\times \id: \Rep_{1} (H)\times K (\Z,2) \to K (H^{2}
(BH),0) \times K (\Z,2),
\]
where $B$ is the
isomorphism \eqref{eq:52}.
\end{proof}

\begin{Remark}  This gives another proof that the natural map
\[
    \Pic_{G} (X) \to H_{G}^{2} (X;\Z)
\]
is an isomorphism, where $\Pic_G(X)$ is the group of equivariant line
bundles over $X$ (Atiyah and Segal \cite{MR2172633}).
\end{Remark}

\subsection{The spectrum $\mstringc$}
\label{sec:spectrum-mstringc}

Associated to the spaces $BU,$ $BSU,$ and $\bstringc$ over $BU$
we have Thom spectra $MU$, $MSU$ and $\mstringc$.   The spectrum $MU$
is easily seen to be an  $E_{\infty}$ ring  spectrum since it comes
with an action of the linear isometries operad.  We turn to $MSU$
and $\mstringc$.

\begin{Proposition}\label{t-pr-mstring-einfty}
The spectra $MSU$ and  $\mstringc$ are $E_{\infty}$-ring spectra.
\end{Proposition}

\begin{proof}
It suffices to show that $BSU$ and $\bstringc$ are infinite loop
spaces over $BU,$ and so it suffices to show that the Borel Chern
classes $\ChB_{1}$ and $\ChB_{2}$ arise from maps of spectra.

In \cite{MR2027899}, the second listed author defined the
$G$-equivariant connective $K$-theory spectrum $ku$ to be the pull-back in
the right square of the diagram 
\begin{equation} \label{eq:74}
\xymatrix{
{\nonequivkU}
  \ar[r]
  \ar[dr]
&
{ku}
  \ar[r]
  \ar[d]
&
{K}
  \ar[d]
\\
&
{\map_{\ptspace} (EG_{\plus},\nonequivkU )}
  \ar[r]
&
{\map_{\ptspace} (EG_{\plus}, K).}
}
\end{equation}
Here $\nonequivkU$ (respectively $\nonequivK$) is the inflation of the
non-equivariant connective $K$-theory spectrum (respectively the
equivariant periodic
$K$-theory spectrum), and the bottom arrow is induced by the
composition
\[
    \nonequivkU \rightarrow \nonequivK \rightarrow K,
\]
obtained using the fact that periodic complex $K$-theory is
split.   Note that the construction implies that, as we have
indicated, $ku$ is split.

As explained in \cite{MR2027899}, by looping down the diagram
\eqref{eq:74} we obtain diagrams of the form
\[
\xymatrix{
{\nonequivBU}
  \ar[r]
  \ar[dr]
&
{BU}
  \ar[r]
  \ar[d]^{\alpha_{1}}
&
{BU}
  \ar[d]
\\
&
{\map_{\ptspace} (EG_{\plus},\nonequivBU)}
  \ar[r]
&
{\map_{\ptspace} (EG_{\plus}, BU)}
}
\]
and
\[
\xymatrix{
{\nonequivBSU}
  \ar[r]
  \ar[dr]
&
{BSU}
  \ar[r]
  \ar[d]^{\alpha_{2}}
&
{BU}
  \ar[d]
\\
&
{\map_{\ptspace} (EG_{\plus},\nonequivBSU)}
  \ar[r]
&
{\map_{\ptspace} (EG_{\plus}, BU)}
}
\]
Already this exhibits $BSU$ as an infinite loop space over $BU$, but
to compare to the tower for $\bulrp{4}$, we observe that the map
\[
    \nonequivBU  \to BU \xra{\alpha_{1}} \map_{\ptspace } (EG_{\plus}, 
\nonequivBU)
    \xra{} \map_{\ptspace} (EG_{\plus},K (\Z,2))
\]
represents $c_{1}\otimes 1$ in $H^{2} (\nonequivBU
\times BG)$, and so Proposition \ref{t-pr-borel-and-splitting} implies
that the composition
\[
    BU \xra{\alpha_{1}} \map_{\ptspace} (EG_{\plus},\nonequivBU)
    \xra{} \map_{\ptspace} (EG_{\plus},K (\Z,2))
\]
represents $\ChB_{1}$, as required.  An analogous argument shows
that the composition
\[
     BSU \xra{\alpha_{2}} \map_{\ptspace } (EG_{\plus},\nonequivBSU)  
\xra{}\map_{\ptspace}
     (EG_{\plus},K (\Z,4))
\]
represents $\ChB_{2}$.  In particular, this is an infinite loop
map, and so exhibits $\bstringc$ as  an infinite loop space over $BSU.$
\end{proof}


\section{Characteristic classes for $\T$-vector bundles}
\label{sec:char-class-circle}

In this section, we focus on the special case $G=\T$, and we write $\A$
for a general closed subgroup of $\T$.
We have two goals.  The first is to give an equivariant form of the
splitting principle, so that in \S\ref{sec:class-delta_a} we can
identify some characteristic classes of $\A$-equivariant bundles over
$\A$-fixed spaces.  The second is to record the calculation of the
Borel Chern classes.  These will be used throughout the remainder of
the paper.


Because we use multiplicative notation in $\A^\vee$ and additive  
notation
in $H^2(B\A)$ we write
\[
\log : \A^\vee \stackrel{\cong}\lra H^2(B\A;\Z)
\]
for the first Chern class isomorphism between them.


We write $z$ for the generator of $H^{2}(B\T)$ corresponding to the  
natural representation of $\T$, and also for its restriction to $H^{*} 
(B\A)$.
Over a $\T$-fixed base, we always have
\[
H^*(X\times_{\T} E\T)=H^*(X) \otimes H^{*}(B\T) \iso H^*(X)[z],
\]
and our notation will reflect this.  If $\A=\T[n]$ and $X=X^{\A}$, we
still have
\[
    H^{*} (X\times_{\A}E\A)\iso H^{*}(X)[z] / nz,
\]
provided $H^{*}(X)$ is concentrated in even degrees.

\subsection{Reductions and the splitting principle}
\label{sec:reduct-splitt-princ}

In this section we describe the cohomology rings $H^{*} (BU
(n)^{\A})$ and $H^{*} (BSU (n)^{\A};\Q)$ using the splitting  
principle.  We
start with $U (n).$  Since
$\A$ is abelian, we may choose our representatives $m\in \Rep_{n}
(\A)$ of $\Hom (\A,U (n)^{c})/U (n)$ to be of the form
\begin{equation}\label{eq:53}
      m: \A\to T,
\end{equation}
where $T$ is the maximal torus of diagonal matrices.  If $m$ is such a
homomorphism, then its centralizer
\[
      Z (m) =\{g\in U (n) | g m g^{-1} = m \}\iso \prod_{\alpha \in
      \A^{\vee}} U (\rank \Hom (\alpha,m))
\]
is a product of unitary matrices.  In particular, it is connected,
with maximal torus $T$.  We define $W (m)$ to be the Weyl group of $Z
(m)$ with respect to the torus $T$; it is a subgroup of the Weyl group
$W$ of $T$ in $U (n)$.  With these choices, Proposition
\ref{t-pr-BU-n-H} takes the following form.

\begin{Proposition} \label{t-pr-BU-n-A}
\[
    BU (n)^{\A}\heq \coprod_{m\in \Rep_{n} (\A)} BZ (m),
\]
and so
\[
    H^{*}BU (n)^{\A} \iso \prod_{m\in \Rep_{n} (\A)} H^{*} (BT)^{W
(m)}.\qed
\]
\end{Proposition}

\begin{Example} \label{ex-1}
Any homomorphism $\T\to U (n)$ is conjugate to one of the form
\[
      z \mapsto m (z) = \diag (z^{m_{1}},\dotsc ,z^{m_{1}},
z^{m_{2}},\dotsc ,z^{m_{2}},\dotsc , z^{m_{k}},\dotsc ,z^{m_{k}}),
\]
where the $m_{i}$ are integers, $m_{i}< m_{j}$ for $i<j$, $m_{i}$
occurs $d_{i}$ times, and
\[
        \sum_i d_{i} = n.
\]
Then $Z (m)$ is the group of block-diagonal matrices $\prod_i U
(d_{i})$, with maximal torus $T$ and Weyl group $\prod_i
\Sigma_{d_{i}}$.
\end{Example}

Recall that in the isomorphism of $\A$-bundles
\[
     V\iso \bigoplus_{\alpha\in \A^{\vee}} \Hom (\alpha,V)\otimes
\alpha,
\]
$\A$ acts trivially on $\Hom (\alpha,V)$, while the bundle underlying
$\alpha$ is a topologically trivial line bundle.  Thus if
\[
     \Hom (\alpha,V) \iso L_{1} \oplus \dotsb \oplus L_{d}
\]
as a non-equivariant bundle, then
\[
      \Hom (\alpha,V)\otimes \alpha \iso L_{1}\otimes \alpha
\oplus\dotsb \oplus L_{d}\otimes \alpha
\]
as a bundle with $\A$-action.  Proposition \ref{t-pr-BU-n-A} implies
the following form of the splitting principle.

\begin{Lemma} \label{t-pr-splitting}
Let $V$ be an $\A$-equivariant vector bundle over an $\A$-fixed space
$X$.  The splitting principle holds in the sense that there is another
$\A$-fixed space $X'$ and a cohomology monomorphism $X'\lra X$ so that
over $X'$ we may write
\[
V \iso \bigoplus_{\alpha\in \A^{\vee}} \bigoplus_{i=1}^{d_{\alpha}}
L_{\alpha,i} \otimes \alpha,
\]
where $L_{\alpha,i}$ is a line bundle with trivial action, and
$\alpha$ describes the $\A$-action.
Moreover in the universal case,  if
$x_{\alpha,i}=c_{1}L_{\alpha,i}$, then the image of $H^{*}X$ in
$H^{*}X'$ consists of the expressions in the
$x_{\alpha,i}$s
which are
invariant under the evident action of
\[
     \prod_{\alpha} \Sigma_{d_{\alpha}}.
\]
\qed
\end{Lemma}

Proposition \ref{t-pr-BU-n-A}, like its parent, is phrased in terms of
a choice of representatives for
\[
\Hom (\A,U (n)^c)/U(n).
\]
It will be important to have a more invariant expression, which
also applies to $SU (n)$.  So let $\Gamma$ stand for one of
these groups, and let $T$ be a maximal torus, with Weyl group $W$.

Suppose that
\[
\pi: P \lra X
\]
is an $\A$-equivariant principal $\Gamma$-bundle, over a trivial $\A$-space
$X$.  The action of $\A$ on $P$ corresponds to a section 
\[
     s: X \lra P\times_{\Gamma}\Hom (\A,\Gamma^{c}).
\]
giving a function
\[
    f: X \lra P \times_{\Gamma}\Hom (\A,\Gamma^{c}) \lra \Hom (\A, 
\Gamma^{c})/\text{conjugacy}.
\]

\begin{Definition}
A \emph{reduction} of the action of $\A$ on the principal bundle $P$  
over $X$
  is a function
\[
     m: \pi_{0}X \xra{}\Hom (\A,T)
\]
making the diagram
\[
\begin{CD}
X @> f >> \Hom (\A,\Gamma^{c})/\Gamma \\
@VVV @AAA \\
\pi_{0}X @> m >> \Hom (\A,T)
\end{CD}
\]
commute.  Note that a reduction always exists, because the right
vertical arrow is a surjection of discrete spaces.
\end{Definition}

This definition is convenient for analyzing principal $\Gamma$-bundles  
over
not-necessarily connected spaces.  In the following discussion,
though, we suppose that $X$ is connected, leaving the modifications
for general $X$ to the reader.

Let $Z (m)\subseteq \Gamma$ be the centralizer of $m$ in $\Gamma$.  It  
is
important to note the following.

\begin{Lemma}\label{t-le-Z-m-conn}
With $\Gamma =U(n) $ or $SU(n)$, for any $m:\A\to T,$ the centralizer  
$Z (m)$ is connected, with maximal torus $T$.
\end{Lemma}

\begin{proof}
For $\Gamma=U (n)$ this is clear, since $Z (m)$ is a product of unitary
groups (see Example \ref{ex-1}).  For $SU (n)$, it is a result of Bott
and Samelson \cite{BS:atmss,BottTaubes:Rig} that for any simply  
connected
compact Lie group $\Gamma$, the centralizer of any element is connected.
The maximal torus is $T$, since $T$ is maximal in $\Gamma$.
\end{proof}

Let $W (m)$ be the Weyl group of $Z (m)$; it is a subgroup of $W$.
Any other reduction $m': \A\to T$ is of the form
\[
m'  = wm
\]
where $w\in W$, and
\[
      wm = m
\]
if and only if $w\in W (m).$

The reduction $m$ determines a principal $Z (m)$-bundle $Q (m)$ over
$X$, by the formula
\[
     Q (m) = \{p \in P | s\pi (p) = \overline{(p,m)} \}.
\]
This is classified by a map
\[
      g_{m}: X \lra BZ (m).
\]
By the splitting principle,
\[
    H^{*} (BZ (m);\Q) \iso H^{*} (BT;\Q)^{W (m)},
\]
and so an element $\Xi$ of the right hand side gives an element
\[
       g_{m}^{*}\Xi \in H^{*} (X;\Q).
\]

\begin{Proposition} \label{t-pr-splitting-princ-messy}
Let $\Gamma=U (n)$ or $SU (n)$ as above.  Let $\A$ be
a closed subgroup of $\T$.  Then $H^{*} (B\Gamma^{\A};\Q)$ is  
isomorphic to
the ring
$$\Hom_W(\Hom (\A , T), H^{*}(BT;\Q))$$
of $W$-equivariant functions. More explicitly, it consists
of functions
\[
     \Xi: \Hom (\A,T) \to H^{*} (BT;\Q)
\]
such that
\begin{enumerate}
\item for each $m\in \Hom (\A,T)$, $\Xi (m) \in H^{*} (BT;\Q)^{W
(m)}$;
and
\item for $w\in W$,
\[
        \Xi (m) = w^{*}\Xi (wm) \in H^{*} (BT;\Q)^{W (m)}.
\]
\end{enumerate}
In particular, any such function $\Xi$ determines a characteristic
class of $\A$-equivariant complex vector bundles over $\A$-fixed spaces,
by the formula
\[
         \Xi (V) = g_{m}^{*} \Xi (m),
\]
where $m: \pi_{0}X \to \Hom (\A,T)$ is any choice of reduction of the
action of $\A$ on $V/X$.  For $\Gamma=U (n)$, the analogous statements  
for
integral cohomology are true as well.
\end{Proposition}

\begin{proof}
Another choice of reduction $m'$ determines $Z (m'), Q (m'),$ and
$g_{m'}$ as above, and there is an element
\[
    w \in W (m')\backslash W/ W (m),
\]
determined by the formula
\[
     m' = wm \in \Hom (\A,T)
\]
and making the diagram
\[
\xymatrix{ {X} \ar[r]^-{g_{m}} \ar[dr]_-{g_{m'}} & {BZ (m)} \ar[d]^{w}
\\
& {BZ (m')} }
\]
commute.  Thus if
\[
\Xi (m) = w^{*}\Xi (m') \in H^{*} (BT)^{W
(m)},
\]
then
\[
     (g_{m'})^{*}\Xi (m') = g_{m}^{*} \Xi (m) \in H^{*}X.
\]
\end{proof}

\begin{Remark}
The main ingredient in the argument is the splitting principle for $BZ
(m),$ so one needs to know that $Z (m)$ is a connected compact Lie
group.  Thus the result of Bott and Samelson \cite{BS:atmss} implies
that the Proposition holds rationally for any simply-connected compact
Lie group.
\end{Remark}

\begin{Remark}
The results of this section and of Proposition \ref{t-pr-BU-n-H} say
that the components of $B\Gamma^{\A}$ are labelled by elements of
\[
     \Hom (\A,\Gamma^{c})/\Gamma,
\]
where $\Gamma^c$ denotes $\Gamma$ as a $\Gamma$-space with the  
conjugation action.
A choice of representative $m: \A\to \Gamma$ identifies the  
corresponding
component with $BZ (m)$.  One way to work with $B\Gamma^{\A}$, then,  
is to
fix a set of representatives.  Elsewhere in this paper, particularly
from Section \ref{sec:edale} onwards, it is
essential \emph{not} to do so, because we must understand the
behaviour of our characteristic classes under restriction
\[
    B\Gamma^{\T}\to B\Gamma^{\A},
\]
which leads us to consider diagrams
like
\[
\xymatrix{
{\T}
  \ar[r]^{\tm}
&
{T}
\\
{\A}
  \ar[u]
  \ar[ur]_{m}
}.
\]
Our approach is to give formulae which work
for any $m: \A\to T$ and which are compatible with the action of $W$ by
conjugation.   Proposition \ref{t-pr-splitting-princ-messy} tells us
how to do this.  When we write
that a homomorphism
$m: \A\to \Gamma$ ``labels a component of $B\Gamma^{\A}$'', we mean  
that we use
$m$ to identify its component with $BZ (m)$.
\end{Remark}

\subsection{Chern classes of $\T$-bundles} \label{sec:chern-classes}

Our calculation of Chern classes uses the splitting principle (Lemma
\ref{t-pr-splitting}) to deduce the general case from the following
result, which is a specialization of Lemma
\ref{t-le-c-1-B-and-c-1-alpha}.

\begin{Lemma}
\label{c1Tfixed} If $L$ is a line bundle over an $\A$-fixed space, and
if $\alpha\in \A^{\vee}$, then
\[
\ChB_1 (L\otimes \alpha )=c_1(L)+ \log(\alpha)\cdot z.
\]
\end{Lemma}

\begin{proof}
The only point is to observe that, under the decomposition
\[
     BU^{\A}\heq JU (\A)\times \prod_{\beta \in \A^{\vee}} BU,
\]
of Proposition \ref{t-pr-BU-H}, the map classifying $L\otimes \alpha$  
maps to the $\alpha$ factor of
$BU$ as the map classifying $L$, and to the other factors trivially.
That is,
\[
     c_{1}^{\alpha} (L\otimes \alpha) = c_{1}L,
\]
while
\[
     c_{1}^{\beta}L = 0
\]
for $\beta\neq \alpha$.
\end{proof}

Now suppose that $V$ is an $\A$-equivariant vector bundle over an
$\A$-fixed space, and that after pulling back along a cohomology
monomorphism $X'\to X$ we have
\[
      V\iso L_{1}\otimes \alpha_{1} + \dotsb + L_{n}\otimes \alpha_{n}.
\]
Then
\begin{equation}\label{eq:27}
\ChBdot (V)=\prod_i(1+c_1(L_i)+\log(\alpha_i) z).
\end{equation}
This gives a calculation of the first and second Borel Chern classes.
In order to state the result, we introduce the following quantities.
Suppose that $m= (m_{1},\dotsc ,m_{d})$  and $m'= (m_{1}',\dotsc ,m_ 
{d}')$ are
arrays of elements of $B\iso \Z$ or $\Z/n$ (in our applications,
$m, m' \in \Hom (\A,T)$).  Let
\begin{align*}
    \qf (m) & \eqdef -\sum_{i<j} m_{i} m_{j} \\
    I (m,m') & \eqdef -\sum_{i\neq j} m_{i} m_{j}'.
\end{align*}
Similarly, if $(x_{1},\dotsc,x_{d})$ are elements of a $B$-module $X$,
then
\[
     I (m,x) \eqdef -\sum_{i\neq j} m_{i} x_{j}.
\]

We have chosen the signs of $\qf$ and $I$ so that the right hand sides
appear with positive sign in the following.

\begin{Lemma}\label{t-le-phi-I}
\begin{enumerate}
\item $\qf$ is quadratic, $I$ is symmetric and bilinear, and
\[
     \qf (m+m') = \qf (m) + I (m,m') + \qf (m').
\]
\item If $\sum m_{i}= 0$, then
\[
       I (m,x) = \sum_{i} m_{i}x_{i}.
\]
\item If $\sum m_{i}=0$ then, then
\[
        2 \qf (m) = \sum_{i} m_{i}^{2}.
\]\qed
\end{enumerate}
\end{Lemma}

\begin{Lemma}  \label{t-le-equiv-ch-classes-for-covers}
Writing $m_i=\log (\alpha_i)$ and $x_i=c_1(L_i)$, we have
\[
c_1^{B}(V)=c_1(V)+(\sum_i m_i )\cdot z
\]
and
\[
c_2^{B}(V)=c_2(V) - I (m,x) z - \qf (m) z^{2}.
\]
In particular,
\begin{align*}
    c_{1}^{0} (V) & = c_{1} (V) \\
    c_{1}^{2} (V) & = \sum_{i} m_{i} \\
    c_{2}^{0} (V) & = c_{2} (V) \\
    c_{2}^{2} (V) & = -I (m,x)\\
    c_{2}^{4} (V) & = -\qf (m).
\end{align*}
If $c_1^{B}(V)=0$ then
\begin{equation}\label{eq:28}
    c_{2}^{2} (V) = -\sum m_{i} x_{i}.
\end{equation}
If $c_1^{B}(V)=0$ and $\A=\T$, then
\begin{equation}\label{eq:29}
    c_{2}^{4} (V) = -\frac{1}{2}\sum_{i} m_{i}^{2}.
\end{equation}
\end{Lemma}

\begin{proof}
The expressions for $\ChB_{1}$ and $\ChB_{2}$ follow easily from the
product formula \eqref{eq:27}.  If $\ChB_{1} (V) = 0$, then
$\sum_{i} m_{i} = 0$, and the formula \eqref{eq:28} follows from
Lemma \ref{t-le-phi-I}.   Finally, if we note that in the universal
case $2$ is not a zero divisor, \eqref{eq:29} also follows from Lemma
\ref{t-le-phi-I}.
\end{proof}

Applying Lemma \ref{t-le-equiv-ch-classes-for-covers} to the universal
bundle $\tU$ over $BU^A$, using the splitting of Proposition \ref{t-pr-BU-H},
we have the following, which will be useful in Section \ref 
{sec:cohomology-covers-bu}.

\begin{Proposition} \label{t-pr-cij-univ}
Let
\[
      V = \sum_{\alpha\in \A^{\vee}} d_{\alpha} \alpha
\]
be an element of $JU (\A).$  In the $V$ factor of
\[
         H^{*} (BU^{\A}) \iso \prod_{W\in JU (\A)}
\Z\psb{c_{1}^{\alpha},c_{2}^{\alpha},\dotsc | \alpha \in \A^{\vee}},
\]
we have
\begin{align*}
       c_{1}^{0} (\tU) & = \sum_{\alpha} c_{1}^{\alpha} \\
       c_{1}^{2} (\tU) & = \sum_{\alpha} d_{\alpha} \log (\alpha).
\end{align*}
If $V\in JU_{2} (\A)$, then in the $V$ factor of $BSU^{\A}$, we have  
(for any fixed ordering on $A^{\vee}$)
\begin{align*}
        c_{2}^{0} (\tU) & = \sum_{\alpha} c_{2}^{\alpha} +
                           \sum_{\alpha<\beta} c_{1}^{\alpha}c_{1}^ 
{\beta}\\
       c_{2}^{2} (\tU) & = \sum_{\alpha} \log (\alpha) c_{1}^{\alpha}.
\end{align*}
\end{Proposition}

\begin{proof}
Let $\tU_{\alpha}$ denote the universal bundle (of rank $0$) over the
$\alpha$ factor of $BU$ in \eqref{eq:8}.  Then the universal bundle
over the $V$ component of $BU^{\A}$ is  (see the second proof of
Proposition \ref{t-pr-BU-H})
\[
       \tU =  \sum_{\alpha} \tU_{\alpha} \otimes \alpha + \sum_{\alpha}
	d_{\alpha} \alpha.
\]
The formulae in the Proposition follow from this and Lemma
\ref{t-le-equiv-ch-classes-for-covers}.
\end{proof}

\section{The cohomology of connective covers of $BU \times \protect\Z$}
\label{sec:cohomology-covers-bu}

The long exact sequence \eqref{prop:Hasse} we use to calculate
the $\T$-equivariant elliptic cohomology of $X$ involves the (rational)
Borel (co)homology of $X^{\A}$.  In this section we carry out the
calculation for $X=\bulrp{2k}$ with $k=0,1,2,3$.  The main point is that
the ordinary cohomology of the fixed set is
concentrated in even degrees, so the Serre spectral sequence for the
Borel cohomology collapses.  For the case $k=3$ this requires the use of
  rational coefficients.

\subsection{Components and simple connectivity}
\label{subsec:pi0pi1}

Let $G$ be a compact Lie group.  We recall that $JU(G)$ is the
augmentation ideal of representations of virtual dimension zero in
the complex representation ring $RU(G)$ of $G$.

\begin{Lemma}
If $\A$ is finite cyclic or the circle group then
\[
\pi_0(BU\lrp{2k}^\A)=JU(\A)^{k} \mbox{ for } k=0,1,2,3.
\]
All components of $BU\lrp{2k}^\A$ are homotopy equivalent.
\end{Lemma}

\begin{Remark} \label{rem-6}
For general groups $G$ it is more natural to expect
$$\pi_0(BU\lrp{2k}^G)=JU_{k}(G) \mbox{ for } k=0,1,2,3,$$
where $JU_k(G)$ is the ideal generated by the
representation-theoretic Chern classes $c_l(V)$ for $l \geq k$.  If
$G$ is abelian then $JU_k(G)=JU(G)^k$.
\end{Remark}

\begin{proof}
The equivalence between the components comes from the H-space
structures.  We carry out the $\pi_{0}$ calculations.
Since all homotopy groups of $\map (B\A, K(\Z ,2n))$ are in even
degree, we have
\[
0 \lra \pi_0(\bulrp{2k+2})\lra \pi_0(\bulrp{2k})
\stackrel{c_k^{2k}}\lra H^{2k}(B\A)
\]
as in the diagram of Proposition \ref{t-pr-A-fixed-fibration-diag}.
By definition $c_0^0$ is the dimension and surjective, and so
$\pi_0(BU^\A)=JU(\A)$.

Lemma \ref{t-le-equiv-ch-classes-for-covers} implies that
\[
c_1^2: JU(\A)=\pi_0(BU^\A) \lra H^2(B\A)=\A^{\vee}
\]
is the determinant.  This is surjective, and $\pi_0(BSU^\A)$ is the
ideal $JU_2'(\A)$ consisting of elements of $JU(\A)$ with determinant
1. It is easy to check
\[
JU(\A)^2\subseteq JU_2(\A)\subseteq JU_2'(\A),
\]
and it remains to show that $JU_2'(\A) \subseteq JU(\A)^2$.

We may do this explicitly as follows (the argument is due to Neil
Strickland).  First note that an arbitrary element $x$ of $JU_2'(\A)$
is of the form
$$x=\alpha_1 \oplus \cdots \oplus \alpha_s - \beta_1\oplus \cdots
\oplus \beta_{s}, $$ where the $\alpha_{i}$ and $\beta_{i}$ are the
classes of one-dimensional representations, and
\[
\prod \beta_{i} = \prod \alpha_{i}.
\]
Notice that
\[
    (1-\beta_{1}) (1-\beta_{2}) + (1-\beta_{1}\beta_{2}) (1-\beta_{3})
     = 2 - \beta_{1} - \beta_{2} -\beta_{3} + \beta_{1}\beta_{2}\beta_ 
{3} \\
\]
This has the generalization
\[
    \sum_{j=1}^{s-1} (1-\prod_{k=1}^{j}\beta_{k}) (1-\beta_{j+1}) =
(s-1) - \beta_{1} - \dotsb - \beta_{s} + \prod_{j=1}^{s} \beta_{j}.
\]
Let us write $n (\beta_{1},\dotsc ,\beta_{s})$ for this element of
$JU(\A)^{2}$.  Then
\[
     x = n (\beta_{1},\dotsc ,\beta_{s}) - n (\alpha_{1},\dotsc
,\alpha_{s}) \in JU (\A)^{2}.
\]

Finally we have
\begin{equation}\label{eq:9}
c_{2}^4: JU^2(\A)=\pi_0(BSU^\A)\lra H^4(B\A)=\Symm^2(\A^{\vee}).
\end{equation}
%
When $\A$ is a compact abelian group, the isomorphism
\[
     RU (\A)\iso \Z[\A^{\vee}],
\]
identifies $JU (\A)$ with the augmentation ideal $I (\A^{\vee}),$  
and it is not difficult to check
that the map
\[
       c_{2}^{4}: JU (\A)^{2}\to H^{4}B\A
\]
factors as
\begin{equation}\label{eq:10}
\xymatrix{ {I (\A^{\vee})^{2}} \ar[r]^-{c_{2}^{4}} \ar[dr]_-{\text 
{can.}}
& {\Symm^{2}\A^{\vee}} \ar[d]
\\
& {I (\A^{\vee})^{2}/I (\A^{\vee})^{3},} }
\end{equation}
where the vertical map is the one induced by the fact that, for any
abelian group $B$, the map
\[
     B\times B \to I (B)^{2}/I (B)^{3}
\]
sending $(x,y)$ to the class of $(1-x) (1-y)$ is symmetric and
bilinear.  For any abelian group, this vertical map
is an isomorphism \cite[Theorem 8.6]{MR537126}, and so the kernel of
the horizontal map is $I (\A^{\vee})^{3}\iso JU (\A)^{3}.$
\end{proof}

\begin{Remark} For general $G$ and $k=2$ we may prove the result
indicated in Remark \ref{rem-6} as follows.
Using the fact $JU_2$ is an ideal we may assume $\det U=\det V=1$ and
hence $x=(U-n)+(n-V)$ is a sum of two elements of $JU_2'$. Now $U-n$
is the pullback from $SU(n)$ of $\tilde{U}-n$ where $\tilde{U}$ is the
natural representation, and this is the pullback of $\tilde{U}-(\delta
+ n-1)$ from $U(n)$, where $\delta$ is the determinant of
$\tilde{U}$. This universal case follows from the calculation of
$ku_*^{U(n)}$ in \cite{MR2027899}.
\end{Remark}

\begin{Lemma}
If $\A$ is finite cyclic or the circle group, all components of
$\bulrp{2k}^\A$ are simply connected.
\end{Lemma}

\begin{proof}
The non-equivariant simple connectivity of $BU$ is well known, and
implies that of the components of $BU^\A \times \Z $ and $BU^\A$.  For
$BSU^\A$ it follows from the surjectivity of
\[
c_1^0: BU^{\A}\to K (H^{0} (B\A),2)
\]
in $\pi_2$, which is a consequence of Lemma \ref{c1Tfixed} and the
well-known non-equivariant case. For $\bulrp{6}^\A$ it suffices to show
the surjectivity of
\[
\pi_{2} ( c_2^2): \pi_{2}BSU^{\A}\to \pi_{2}K (H^{2} (B\A),2) = H^{2}
(B\A).
\]
Notice that the map
\[
    B\T\to K(H^{2}B\T,2)
\]
corresponding to a generator of $H^{2}B\T$ is an equivalence, and that
the natural map
\[
   K (H^{2}B\T,2) \to K (H^{2}B\A,2)
\]
is an epimorphism in $\pi_{2}$.  In particular, for every element
$x\in \pi_{2}K (H^{2}B\A,2) = H^{2}B\A$, there is a line bundle $L$ over
$S^{2}$ such that the map
\[
     S^{2} \xra{L} B\T \to K (H^{2}B\A,2)
\]
represents $-x$.

Recall that we have chosen a generator $z$ of $H^{2}B\T$, and let
$\alpha \in \A^{\vee}$ be a generator, so that
\[
       B\alpha^{*} z \in H^{2}B\A
\]
is a generator (which we will also call $z$).  Now consider the
virtual $\A$-bundle
\[
    W = (1-L) (1-\alpha) = 1 - L - \alpha + L\otimes \alpha.
\]
over $S^{2}$.  Its Borel Chern class is
\[
\ChBdot (W) = \frac{1-x+z}{(1-x) (1+z)} = 1 + x z + \text{degree }6.
\]
In particular, it is an $SU$ bundle whose $c_{2}^{2}$ component is
$x$.  Thus we have a commutative diagram
\[
\xymatrix{ {BSU^{\A}} \ar[r]^-{c_{2}^{2}} &
{K (H^{2} (B\A),2),} \\
{S^{2}} \ar[u]^-{W} \ar[ur]_-{x} }
\]
showing that $c_{2}^{2}$ is surjective in $\pi_{2}.$
\end{proof}

\subsection{Homology and cohomology of fixed points of $BSU$}
\label{sec:homol-cohom-fixed-bsu}

The cohomology ring of $BU$ is well known to be polynomial on the
Chern classes, so that
$$H^*(BU^\A)=\prod_{V\in JU (\A)} \Z [c_1^{\alpha}, c_2^{\alpha}, \ldots
\st \alpha \in \A^{\vee}].$$ We note that the usual calculation of
$H_{*}(BU)$ in the nonequivariant case generalizes to give
\[
H_*(BU^\A)=\Symm H_{*} (BU (1)^{\A}) = \Z [\beta_1^{\alpha},
\beta_2^{\alpha}, \ldots \st \alpha \in \A^{\vee}] [JU(\A)],
\]
where $\beta_{i}^{\alpha}$ is the basis dual to $(c_{1}^{\alpha})^{i}$
in $H^{*}(BU (1)^{\A})$.

For $BSU$ we
consider the fibre sequence $BSU^\A \lra BU^\A_\eps \lra \buoneAeps$ of
\eqref{eq:64} and the decomposition 
\[
BU^A_\eps \heq JU_2(A) \times \prod_{\alpha \in A^{\vee}} BU
\]
of \eqref{eq:63}.  
If $V \in JU_{2} (A)$, we write $BU^{A}_{V}$ for the corresponding
component of $BU^{A}_{\eps}.$  Recall from Proposition
\ref{t-pr-cij-univ} that the generator of $H^*(\buoneAeps)$ acts as $\Sigma_{\alpha}
c_1^{\alpha}$. 

It follows that $H^*(BU^\A_V)$ is flat as a module
over $H^*(\buoneAeps)$ for each $V$ and hence the Eilenberg-Moore
spectral sequence gives
$$H^*(BSU^\A_V)=\Z \tensor_{H^*(\buoneAeps)} H^*(BU^\A_V).$$ This means
that each component of $BSU^\A$ has polynomial cohomology in even
degrees. Dually, in homology we may deal with all components at once
to find $$H_*(BSU^\A)= \Hom_{H^*(\buoneAeps)}(\Z , H_*(BU^\A_\eps)).$$
The Serre spectral sequence gives exactly the same calculation if we
consider the fibration
\[
K(\Z,1) \lra BSU^\A \lra BU^\A_\eps.
\]



\subsection{Homology and cohomology of fixed points of $\busix$}
\label{sec:homol-cohom-fixed-1}

To continue, we suppose that $\A$ is either a finite cyclic group $C$ or
the circle group $\T$, so that $H^3(B\A)=0=H^{1} (B\A)$.  Thus have a  
fibration
\[
\busix^\A \lra BSU^\A \lra K(H^0(B\A), 4) \times K(H^2(B\A), 2) \times
K(H^4(B\A), 0).
\]
As for $BSU$, we may trim away the component group in the base by
letting $BSU^\A_S$ consist of the
components indexed by representations in
\[
JU(\A)^{3} = \Ker (\pi_{0}BSU^{\A} \lra H^{4}B\A ).
\]
Then we have a fibration
\[
\busix^\A \lra
BSU^\A_S \lra K(H^0(B\A), 4) \times K(H^2(B\A), 2)
\]
with connected base.   All components are equivalent, and we have a
fibration
\[
\busix^\A_0 \lra BSU^\A_0 \lra K(H^0(B\A), 4) \times K(H^2(B\A), 2)
\]
of connected spaces, where the subscript 0 indicates the component
of the 0 bundle.

For the purposes of this paper, it is sufficient to work over the
rationals, so that $K(H^0(B\A), 4)$ and $K(H^2(B\A), 2)$ have
polynomial cohomology.

We deal separately with the case $\A=C$ is finite and the case
$\A=\T$. In the first we even have $H^{2} (BC;\Q) = 0$, and so the
rational fibration
\[
\busix^C \lra BSU^C_S \lra K(\Q ,4).
\]
Now $H^*(K(\Q, 4); \Q)=\Q [c_2^0],$ where $c_2^0$ acts as its name
suggests by
\[
\sum_{\alpha}c_2^{\alpha}+ \sum_{\alpha < \beta}
c_1^{\alpha}c_1^{\beta} = \sum_{\alpha}c_{2}^{\alpha} - \frac{1}{2}
\sum_{\alpha} (c_{1}^{\alpha})^{2}
\]
for any fixed ordering on $\A^{\vee}$ (see Proposition
\ref{t-pr-cij-univ}).   Since $c_2^0$ can be chosen as a polynomial
generator of $H^*(BSU^C_S)$ we obtain
\begin{align*}
H^*(\busix^C)&= \Q \tensor_{\Q [c_2^0]} H^*(BSU^C_S)
=H^*(BSU^C_S)/(c_2^0) \\
H^*(\busix^C_{0})&= \Q \tensor_{\Q [c_2^0]} H^*(BSU^C_0)
=H^*(BSU^C_0)/(c_2^0)\\
H_*(\busix^C_0)&=\Hom_{\Q[c_2^0]}(\Q,
H_*(BSU^C_0)).
\end{align*}
We note that the cohomology ring of each component
is polynomial and in even degree.

When $\A=\T$ we have
$$\busix^{\T} \lra BSU^{\T}_S \lra K(\Q,4)\times K(\Q ,2),$$ and note
that $H^*(K(\Q, 4)\times K(\Q,2); \Q)=\Q [c_2^0, c_2^2]$ where $c_2^0$
and $ c_2^2$ act as indicated in Proposition \ref{t-pr-cij-univ}.
Since $c_2^0 $ and $c_2^2$ generate a tensor factor $\Q
[c_2^0,c_2^2]$ we obtain
$$H^*(\busix^\T_V)= \Q \tensor_{\Q [c_2^0,c_2^2]} H^*(BSU^\T_V)
=H^*(BSU^\T_S)/(c_2^0,c_2^2)$$ and $$
H_*(\busix^\T_V)=\Hom_{\Q[c_2^0,c_2^2]}(\Q, H_*(BSU^\T_S)).$$ Once
again, the cohomology ring of each component is polynomial and in even
degrees.

\subsection{Borel homology and cohomology}
\label{sec:borel-homol-cohom}

The long exact sequence
we use to calculate the elliptic cohomology of $\bulrp{2k}$ involves
the rational Borel (co)homology of $\bulrp{2k}^{\A}$.  We continue the
convention that $\A$ is a finite cyclic group or the
circle, and we continue to work with rational coefficients.

\begin{Proposition}
\label{prop:rigideven} For $k \leq 3$, the $\A$-fixed point spaces of
$X=BU\lrp{2k}$ have cohomology in even degrees. Each component has
polynomial cohomology.  As for Borel homology and cohomology, there
are isomorphisms (non-canonical unless $\A=\T$)
\begin{align*}
\Bc^*(X^\A) & =H^*(X^\A)[z] \\
\intertext{and}
\Bh_*(X^\A)& =H_*(X^\A) \otimes H_*(B\T).
\end{align*}
Moreover, the natural map
\begin{equation} \label{eq:68}
     H_{\T}^{*} (X^{\A}) \rightarrow H_{\T}^{*} (X^{\T})
\end{equation}
is injective.
\end{Proposition}

\begin{proof}
All the spaces $X^\A=BU\lrp{2k}^\A$ for $k\leq 3$ have components whose
cohomology is polynomial and in even degrees. Accordingly the Serre
spectral sequence calculating the $\T$-equivariant Borel cohomology of
one component collapses and shows the Borel cohomology is isomorphic
to a tensor product of $H^*(B\T)$ and the polynomial cohomology
ring. When $\A$ is finite, this involves {\em choosing} lifts of the
polynomial generators of $H^*(X^\A_V)$ to $H_{\T}^*(X^\A_V)$.   Since
the map $\XT \lra X^\A$ is injective in cohomology, it follows
that \eqref{eq:68} is as well.
Similar arguments show that the Borel homology spectral sequence also
collapses to give the
isomorphism of $H^{*} (B\T)$-modules
$$\Bh_*(X^\A)\cong H_*(X^\A) \otimes H_{*} (B\T).$$
\end{proof}




\part{Elliptic cohomology and the sigma orientation}

In this part we turn towards elliptic cohomology and the sigma genus.
First, in Section \ref{sec:prop-equiv-ellipt} we introduce notation
for discussing the geometry of an elliptic curve, before summarizing
the relevant properties of equivariant elliptic cohomology from
\cite{ellT}. The sigma genus is most easily introduced for $\T$-fixed
spaces and generic points on the curve, because the role of the topology
and geometry is largely  unlinked:  we discuss this in Section \ref 
{sec:edale}.
Finally, in Section \ref{sec:proof-prop-phi-A}
we turn to the more subtle question of how to deal with points with
finite isotropy and torsion points on the elliptic curve; this is
where intricate vanishing and transformation properties of sigma
functions enter into the proof.

\section{Properties of equivariant elliptic cohomology}
\label{sec:prop-equiv-ellipt}

\subsection{Geometry of the elliptic curve}
\label{sec:geom-ellipt-curv}

In this section we summarize the relevant properties of the
$\T$-equivariant elliptic cohomology defined in \cite{ellT}.  We begin
by introducing notation to describe the elliptic curve.   Let $\J$ be a
rational elliptic curve
\[
\xymatrix{
{\J}
  \ar[r]_-{p}
&
{S}
  \ar@/_2ex/[l]_{\e}
}
\]
with identity $\e$ and structure map $p$, over an affine $\Q$-scheme
$S$.

We write $\cO$ for the structure sheaf of $\J$, and for a divisor $D$,
the sheaf $\cO (D)$ consists of functions with $\Div (f)+D \geq
0$. We write $\cK$
for the constant sheaf of meromorphic functions on $\J$ with poles only
at points of finite order.

For any $n \geq 1$ we write $\J[n]=\ker (n: \J \lra \J)$ for the  
subgroup
of points of order dividing $n$ and $\Jlr{n}$ for the scheme of points  
of
exact order $n$.  It is convenient to index certain divisors by
representations of $\T$. Given a representation $V$ with $V^{\T}=0$ we
write $V=\sum_n a_nz^n$, and take $D(V)=\sum_n a_n \J[n]$. Thus we
have
\[
\cK =\colim_{V^{\T}=0}\cO (D(V)).
\]
Next, we write $\cK_n$ for
the functions regular on $\Jlr{n}$ (this is also the
   ring $\cO_{\Jlr{n}}$),
  and $T_n\J=\cK /\cK_n$ for the sheaf
of principal parts of functions on $\Jlr{n}$; this can also be
described as the local cohomology group $H^1_{\Jlr{n}}(\J)$.
We set
\[
      TC =\bigoplus_n T_{n}\J.
\]
For a finite subgroup $\A=\T[n]$, it is the topology of $X^{\A}$
which controls the behaviour of the equivariant elliptic cohomology of
$X$ near $\Jlr{n}$.  As a consequence we adopt the convention that
\begin{align*}
      \Jlr{\A} & = \Jlr{|\A|}\\
      \cK_{\A} & = \cK_{n} \\
      T_{\A}\J & = T_{n}\J \\
      \OA &  = \O^{\wedge}_{\Jlr{A}}.
\end{align*}

We write $\Omega = \Omega_\J$ for the sheaf of K\"ahler differentials,  
and
$\Omega_\J^{\ten{d}}$ for its $d^{\th}$ tensor power.  We write
\[
     \DL = p_{*}\Omega \iso \e^{*}\Omega
\]
for the $\O_{S}$-module of invariant differentials.  If $t$ is a
coordinate on $C$ (i.e. $\e^{*}dt$ is a generator of $\e^{*}\Omega$),
then we write $Dt$ for the associated invariant differential.

Our analysis will involve expressions of the form $f (Dt)^{k}$, where  
$f$ is
a meromorphic function on $\J$;
this may be regarded as a section of
\[
     \cK \otimes_{\cO} p^{*} \DL^{\ten{k}} \iso  \cK \otimes_{\cO}  
\Omega_\J^{\ten{k}}.
\]
We write $\DL^{*}$ for the graded sheaf of tensor powers, and  
similarly for
$\Omega_C^*$. Since $C$ is a one-dimensional group, $Dt$ trivializes
$\Omega_C$, and multiplication by $Dt$
gives periodicity for any graded
sheaf $M \otimes_{\cO}\Omega_C^*$.

We will identify the constant sheaf $\cK$ and its twists by
differentials with their modules of global sections.  That is, we will
generally not distinguish in our notation between
\[
\cK \otimes_{\cO} \Omega_\J^* \iso
\cK \otimes_{\cO} p^{*} \DL^{*}
\]
and
\[
\Gamma (\cK \otimes \Omega_\J^*) \iso
\Gamma (\cK \otimes_{\cO} p^{*}\DL^{*}).
\]

\subsection{Coordinate data}
\label{sec:coordinate-data}

We recall that to give a $\T$-equivariant elliptic spectrum we specify  
not
only an elliptic curve $\J$ but also ``coordinate data'':  this is a
collection of functions $t_{1},\dotsc ,t_{s},\dotsc$, where $t_{s} \in
\O_{\Jlr{s}}^{\wedge}$ and vanishes to first order at each point of
$\Jlr{s}.$

One way to do this was introduced in \cite{ellT}.  We choose a
section $t_{1}$
of $\cK_{\J}$ which
is a coordinate at the identity of $\J$.  Note that setting
$\CoordDiv=\Div (t_{1})$ and $\omega = Dt_{1}$
is a bijection between such sections and pairs $(\CoordDiv,\omega)$
consisting of a non-vanishing invariant differential $\omega$
and a divisor $\CoordDiv=\sum n_{P} (P)$ satisfying
\begin{enumerate}
\item  $\deg \CoordDiv = 0$
\item  $\sum^{C} [n_{P}] (P) = 0$
\item $n_{P} = 0$ unless $P$ is a point of finite order of $\J$
\item $n_{\e} = 1.$
\end{enumerate}
The first three conditions imply that there is a meromorphic function
$t_{1} \in \cK$ with $\Div t_{1} = \CoordDiv$; the last condition
implies that $t_{1}$ is a coordinate at the identity.

In this paper we work with a complex elliptic curve $C\iso
\C/\Lambda$, in which case, if $\Bar{P}$ is a choice of lifts to $\C$
of the points of $\CoordDiv$, then we can take
\[
     t_{1} (z) = \prod_{P} \sigma (z-\Bar{P})^{n_{P}},
\]
where $\sigma$ is the Weierstrass sigma function (in the form described
in Subsection \ref{sec:sigma-function} below).

Next, for  $s>1$ we define $t_s$ to be the meromorphic function
with the properties
\begin{enumerate}
\item $\Div (t_{s}) = \Jlr{s} - |\Jlr{s}| (\e)$
\item $(t_{1}^{|\Jlr{s}|}t_{s}) (\e) = 1$.
\end{enumerate}
In our complex case, if $\overline{\Jlr{s}}\subset \C$ is a set of  
lifts of
the points of $\Jlr{s}$, then
\[
       t_{s} (z) = \lambda_{s} \frac{\prod_{p \in \overline{\Jlr{s}}} 
\sigma (z-p)}
                                    {\sigma (z)^{|\Jlr{s}|}},
\]
where $\lambda_{s}$ is a constant easily expressed in terms of values
of the sigma function.

However we choose the coordinate data, the important point is that we  
use
$t_s/Dt_{1}$ to make $T_s\J\tensor \omega_\J^*$ into a torsion $\Q
[c]$-module:  for $f\otimes \omega \in T_{s}\J\otimes
\omega_{\J}^{*}$, we take
\[
    c^{k} f \otimes \omega = t_{s}^{k}f \otimes (Dt_{1})^{-k}\omega.
\]

\subsection{Spheres and line bundles}

Let $V$ be a virtual complex representation of $\T$, and suppose that
$V^{\T} = 0$.  The spectrum $\EJ$ is constructed so that
\[
      \EJ^{i}_{\T} (S^{V}) = H^{i} (\J; \cO (-D (V)))
\]
for $i=0,1$.  Twisting by a trivial representation of rank $1$ is
equivalent to a double suspension, and as in the non-equivariant case
this introduces a twist by the K\"ahler differentials, giving
\begin{align*}
   \EJ^{i-2d}_{\T}(S^V)&=H^i(\J;\cO (-D(V))\otimes (\Omega_\J^1)^ 
{\otimes d}) \\
   \EJ_{2d-i}^{\T}(S^V)&=H^i(\J;\cO (D(V))\otimes (\Omega_\J^1)^{\otimes
d}).
\end{align*}

\subsection{Localization and completion}  Elliptic cohomology
satisfies a localization theorem and a completion theorem.

Let $\cF$ be the family of finite subgroups of $\T$.
Recall that there is a universal space $\ecf$ for $\T$-spaces with
isotropy in $\cF$, characterized by the fact that its fixed points under
finite subgroups are contractible, whereas $\ecf^{\T}=\emptyset$.
This is related to the join $\etf =S^0*\ecf$ by the cofibre
sequence
$$\efp \lra  S^0 \lra   \etf . $$

It is convenient to use the models
$$\ecf = \bigcup_{V^{\T}=0}S(V) \mbox{ and }
\etf = \bigcup_{V^{\T}=0}S^V ,$$
where $S(V)$ is the unit sphere in $V$ and $S^V\cong S^0 *S(V)$ is the
one-point compactification of $V$.
The usefulness of these spaces arises since for any based
$\T$-space $X$, the inclusion $X^{\T}\to X$ induces a weak equivalence
$$X^{\T} \sm \etf \stackrel{\simeq}\lra X \sm \etf .$$
The corresponding statement holds for spectra if we use geometric
fixed points, but we restrict to spaces so we can retain familiar
notation.

\begin{Lemma} \label{t-le-localization}
For any $\T$-space $X$ we have
$$\EJ_*^{\T}(X \sm \etf)=H_*(X^{\T} ; \cK \tensor \Omega_\J^*), $$
and the corresponding result holds in cohomology under the assumption
that $X$ is a finite complex. The corresponding statement
holds for spectra if we use geometric fixed points.
\end{Lemma}

\begin{proof}
Since $\etf = \colim_{V^{\T}=0}S^V$ and $X \sm \etf \simeq X^{\T} \sm  
\etf$
we easily deduce the result from the values on spheres. Indeed,
$\colim_V \cO (D(V))=\cK$, so that
  $$\EJ_{2d}^{\T}(\etf)= \cK \tensor (\Omega_\J^1)^{\otimes d}.$$
\end{proof}

Before stating the completion theorem, we pause briefly to summarize
the relationship between Borel homology and cohomology, which is
described in more detail in Appendix \ref{sec-UCT}.
Given a
graded module $M$ over $H^{*}(B\T)=k[c]$ for a field $k$, we can form  
the Borel cohomology
$H^{*}_{\T} (X;M)$, and in certain cases there are simple descriptions.
(This notation means the cohomology theory represented
by the module over the Borel spectrum, and not a Brown-Comenetz type
theory as in \cite{ellT}; the distinction is explained further in
  Remark \ref{notationconflict}).
If $M$ is flat and $X$ is a finite $\T$-CW complex, we have
\[
  H^*_{\T}(X;M)\cong   H^{*}_{\T} (X)\tensor_{H^{*}B\T}M.
\]
As usual, homological and cohomological gradings of the same
module are related by $M_{k} = M^{-k}$, and $c$ is of cohomological
degree 2 (i.e., homological degree $-2$). It is  also useful to consider
  the torsion $H^{*}(B\T)$-module $M[c^{-1}]/M$.  If $c$ is not
a zero-divisor in $M$, there is a natural map
\begin{equation} \label{eq:73}
   \kappa:  H_{\T}^{p} (X;M) \xra{} \Hom_{H^{*}B\T}^{p}
     (H^{\T}_{*}X;\Sigma^{-2}(M[c^{-1}]/M)).
\end{equation}
If $M$ is a free module, then $M[c^{-1}]/M$ is injective,
so we have a natural transformation of cohomology theories, and it
is easy to check then that the map is $c$-adic completion. Since  
completion
of $M$ does not affect $M[c^{-1}]/M$, we see that if $M$ is the  
completion
of a free module then $\kappa$ is an isomorphism.

For example, if $M=H^{*}(B\T)$, then
\[
         \Sigma^{-2}(M[c^{-1}]/M)\iso H_{*}(B\T),
\]
and this is injective, so we have the isomorphism
\[
     H^{*}_{\T} (X;H^*(B\T)) \iso \Hom_{H^{*}B\T}^{*} (H^{\T}_{*}(X);H_ 
{*}(B\T)).
\]


For example, let $\OA$ be the formal
completion of $\cO$ at $\Jlr{A}$, and let $M = \OA\otimes \omega^{*}$
be considered as an $H^{*}B\T$-algebra via
\[
      c \mapsto t_{\A}\otimes (Dt)^{-1}.
\]
Then
\[
    (M[c^{-1}]/M) \iso T_{\A}\J\otimes \omega^{*},
\]
and this is an injective $\Q[c]$-module.  Thus Example (\ref 
{BorelcOATA})
shows that we have
\[
     H^{*}_{\T} (X;\OA\otimes \omega^{*}) \iso
     \Hom_{H^*(B\T)}(H_{*}^{\T} (X), T_{\A}\J\otimes \omega^{*}).
\]

\begin{Lemma} \label{t-le-completion}
For any $\T$-space $X$
$$\EJ_{\T}^*(X\sm \efp)
\iso \prod_{\A} H_{\T}^{*} (X^{\A}; \OA\otimes \omega_{\J}^{*}).$$
If $H_*^{\T}(X^\A)=H_*(X^\A)\otimes H_*(B\T)$ for all finite $\A
\subset \T$ then
$$\EJ_{\T}^*(X\sm \efp) \cong \prod_\A H^*(X^\A; \cO_\A^{\wedge}\tensor
\omega_\J^*).$$
The corresponding statement
holds for spectra if we use geometric fixed points.
\end{Lemma}

\begin{proof}
The first statement amounts to the fact that the homotopy of
$\EJ \sm \Sigma \efp$ is injective as a module over $H^*(B\T)$
with coefficients $T\J \tensor \omega_\J^*$. Now we use the
fact that there is a rational splitting $\efp \simeq \bigvee_A \elr{\A}$
corresponding to $T\J\simeq \bigoplus_\A T_\A\J$, and that $[X , \elr 
{\A}
\sm Y]^{\T}= [X^\A , \elr{\A} \sm Y]^{\T}$. Passing to the summand
corresponding to $\A$, the $H^*(B\T)$-module structure on rings of
functions is through $t_{|A|}/Dt$. The second statement follows since  
the
short exact sequence $$0 \lra \cK_\A \lra \cK \lra T_\A\J \lra 0$$ gives
an isomorphism $$\Hom_{H^*(B\T)}(H_*(B\T), T_\A\J\tensor \omega_\J^*)=
\Ext_{H^*(B\T)}(H_*(B\T), \cK_\A\otimes
\omega_\J^*)=\cO_\A^{\wedge}\tensor \omega_\J^*.$$
(See Appendix A for further details.)
\end{proof}

\subsection{Periodicity}
\label{sec:periodicity}


It is sometimes convenient to define the ``periodic ordinary
cohomology spectrum'' by the formula
\[
     HP = \bigvee_{k\in \Z} \Sigma^{2k} H
\]
This spectrum has the feature that
\[
      \spf HP^{0}(\cp )\iso \Gah,
\]
while
\[
      HP^{0}(S^{2}) \iso \pi_{2}(HP) \iso \Gamma ( \omega_{\Gah}).
\]
Note that the coordinate data used to construct $\EJ$ determine an
isomorphism
\[
      \Jhat \iso \Gah
\]
carrying $Dt_{1}$ to the standard generator of $\Gah,$
and so inducing an isomorphism
\[
    H^{*} (X;R\otimes \omega^{*})\iso HP^{*} (X;R).
\]
We shall find it convenient
simply to \emph{define} $\omega^{*}$-periodic cohomology as
\[
    HP^{*} (X;R)\eqdef H^{*} (X;R\otimes \omega^{*}_{\J}).
\]

With this notation, the localization and completion isomorphisms above
become
\begin{align*}
\EJ_*^{\T}(X \sm \etf)&\iso HP_*(X^{\T} ; \cK)\\
\EJ_{\T}^*(X\sm \efp)& \iso \prod_\A HP_{\T}^*(X^\A; \cO_A^{\wedge}).
\end{align*}

\subsection{The Hasse square}
\label{sec:hasse-square}

The localization and completion theorems combine to give an extremely
useful long exact sequence, relating equivariant elliptic cohomology
to Borel cohomology and the elliptic curve. It takes
(i) information from the $\T$-fixed point space, generic on the curve
and (ii) information from the $A$-fixed point space in a neighbourhood
of the points of order $|A|$ on the curve and splices them together.
The idea that points with isotropy of order $n$ in topology are
associated to points of order $n$ on the curve is a recurrent central  
theme.
Topological and geometric information interacts very little over
$\T$-fixed spaces, but much more over points with finite isotropy.

This sequence is \cite[15.3]{ellT}, but we have used  Remark
\ref{BorelcOATA} to give it in a more geometrically transparent form.


\begin{Proposition}
\label{prop:Hasse} For any $\T$-space $X$ there is a long exact
sequence
\begin{multline}\label{eq:56} \cdots \lra \EJ^n_{\T}(X) \lra HP^n 
(\XT ; \cK
) \times \prod_\A HP_{\T}^n(\XA; \cO_A^{\wedge})
\lra \\ HP^n(\XT ; \cKF)
\lra \EJ^{n+1}_{\T}(X) \lra \cdots,
\end{multline}
natural in $X$, where $\cKF = \prod_\A \cOA \tensor \cK$.  If
we are given an isomorphism
$H_*^{\T}(X^\A) \cong H_*(X^\A)\otimes H_*(B\T)$, then we obtain
an isomorphism
\begin{equation} \label{eq:33}
HP_{\T}^n(X^\A; \cO_A^{\wedge}) \cong
    HP^n     (X^\A; \cO_\A^{\wedge}).
\end{equation}
The corresponding statement holds for spectra provided we use  
geometric fixed points.
\end{Proposition}

\begin{Remark} \label{rem-1}
\begin{enumerate}
\item
The first displayed map is a ring homomorphism when $X$ is a space.
\item For the spaces we care most about, $H^*(X^{\T})$ and
$H_{\T}^*(X^\A)$ are in even degrees for all $\A$, so that the long  
exact
sequence
degenerates to give a pullback square of rings for $\EJ^*_{\T}(X)$.
\item  We remark that we may arrange that the map
\[
HP^n(\XT ; \cK) \times \prod_\A HP_{\T}^n(\XA; \OA) \rightarrow \\
   HP^n(\XT ; \cKF)
\]
from the long exact sequence is the obvious one.
On the $\XT$ factor, it is induced by the natural map
\[
      \cK \rightarrow \cKF.
\]
On the $X^{\A}$ factor, we may arrange that it is restriction along
\[
    \XT   \rightarrow \XA,
\]
composed with the natural map
\[
     \OA \rightarrow  \cKF.
\]
To make sense of this, we arrange that both remaining terms may be  
interpreted
as Borel cohomology. Indeed, we may make $\cK_{\cF}^{\wedge} \tensor  
\omega_{\J}^*$
into  a module over $H^*(B\T)$ by letting $c$ act through
$t_{|A|}/Dt$ in the $A$-factor, and as such we have
$$H^n(\XT; \cKF \tensor \omega_\J^*)\cong H^n_{\T}(\XT ; \cKF \tensor  
\omega_\J^*).$$
It will appear from the proof that the map is as stated by naturality
of the completion theorem.

The exact sequence \eqref{eq:56} suggests that
$\EJ^{*}_{\T} (X)$ is related to the cohomology of a sheaf on $\J$: the
$HP^n(\XT ; \cK)$  factor concerns the  behaviour of
a section generically on $\J$, while the $HP_{\T}^n(\XA; \cO_A^ 
{\wedge})$ factors concern the behaviour in small
neighborhoods of the points of finite order.  We shall study the
string orientation from this point of view in
Part \ref{sec:ellipt-cohom-bsu}.
\item The Borel cohomology groups which appear in
\eqref{eq:56} are essentially those which describe Grojnowski's
sheaf-valued theory (in the case of a finite complex).  Note that
Grojnowski treats the case of an elliptic curve of the form
$\C/\Lambda,$ and uses the projection
\[
    \C \to \C/\Lambda
\]
and translation in the elliptic curve to give $\OA$ the structure of
an $H^{*}B\T$-algebra.  One of the innovations of \cite{ellT} is to
handle the algebraic case, using the functions $t_{|A|}$ to make $\OA$
into an $H^{*}(B\T)$-algebra.
\end{enumerate}
\end{Remark}

\begin{proof}
Any $\T$-spectrum $E$ occurs in the Tate homotopy pullback
square $$\begin{array}{ccc} E &\lra & E \sm \etf \\ \downarrow
&&\downarrow\\ F(\efp, E)&\lra & F(\efp, E ) \sm \etf \end{array}$$
where $\cF$ is the family of proper subgroups, and applying $F(X,
\cdot )$ we obtain the homotopy pullback square
$$\begin{array}{ccc}
F(X,E) &\lra & F(X,E \sm \etf) \\
\downarrow &&\downarrow\\
F(X \sm \efp, E)&\lra & F(X,F(\efp, E ) \sm \etf). \end{array}$$ Note
that
$$[X,Y \sm \etf]^{\T}_*=[\PT X, \PT Y]_*=[X^{\T}, \PT Y]_*, $$
  so that both the right
hand terms can be expressed in terms of the geometric fixed points of
$X$.  In the case that $E$ is elliptic cohomology, we apply the
localization theorem to see that $\piT_* (F(X,\EJ \sm \etf))=H^*(X^ 
{\T}; \cK \otimes \omega_\J^*)$ and
the completion theorem to see that
$$\piT_*(F(X \sm \efp , \EJ))=\EJ^*_{\T}(X \sm \efp)=\prod_\A H_{\T}^* 
(X^\A;
\cO_\A ^{\wedge}).$$ \end{proof}

\section{The sigma orientation} \label{sec:edale}

In this section we describe the construction of our Thom class
for the tautological bundle $\tStringc$ over $\bstringc.$  We implement the
strategy for bundles over $\T$-fixed spaces, by showing how to
use the Weierstrass sigma function to construct a Thom class.
Details for spaces which are not fixed are deferred to Section
\ref{sec:proof-prop-phi-A}.

\subsection{The sigma function}

\label{sec:sigma-function}

First of all, we write $\sigma$ for the expression
\begin{align*}
    \sigma (w,q)  =  (w^{1/2} - w^{-1/2})  \prod_{n\geq 1}
\frac{(1-q^{n}w) (1-q^{n}w^{-1})}{(1-q^{n})^{2}}.
\end{align*}
We can consider  $\sigma$ as a function of $(z,\tau)\in \C\times\h$,  
where
$\h$ is the upper half-plane, by
setting
\begin{align*}
     w^{r} &= e^{rz} \\
     q^{r} &= e^{2\pi i r\tau}
\end{align*}
for $r\in \Q$.  It is convenient to consider $\sigma$
sometimes as a function of $w$, writing the first argument
multiplicatively, and sometimes as a function of $z$, writing the
first argument additively.  We'll adopt the convention that the
second argument ($\tau$ or $q$) indicates the form of the first
argument.  Thus with our notation
\[
      \sigma (e^{z},q) =\sigma (z,\tau)
\]

The function $\sigma$ is holomorphic, vanishes only at lattice points,  
and
has the following properties.
\begin{align}
      \sigma (z,\tau)  & = z + o (z^{2}) \label{eq:11} \\
      \sigma (-z,\tau) & = -\sigma (z,\tau) \label{eq:12} \\
      \sigma (z+2\pi i l + 2\pi i k \tau,\tau) &= (-1)^{l+k} e^{-k z -  
\pi i
k^{2}\tau} \sigma (z , \tau)
      \label{eq:16} \\
      \sigma (wq^{k},q) & = (-1)^{k} w^{-k} q^{-\frac{k^{2}}{2}} \sigma
(w,q). \label{eq:13}
\end{align}

\subsection{The Witten genus and the sigma orientation} \label 
{sec:witten-genus}

We use the expansion of $\sigma$ in terms of $z$ in \eqref{eq:11} to
determine an exponential orientation for complex vector bundles.  More
precisely, if $V$ is a complex vector bundle over $X$, then there is a
Thom class
\begin{equation}\label{eq:15}
\Thom (V)= \Thom^{\sigma} (V) \in  H^{*} (X^{V};\C),
\end{equation}
characterized by the property that, if
\[
\Chdot (V) = \prod (1+x_{i}),
\]
then the Euler class associated to $\Thom (V)$ is
\begin{equation}\label{eq:17}
     e (V,\tau) \eqdef \prod \sigma (x_{i},\tau).
\end{equation}

As explained in \cite{MR1189136} (see also \cite{Witten:EllQFT} and
\cite{AHS:ESWGTC}), the $q$-form of $\sigma$ is the $K$-theory
characteristic series of a multiplicative  orientation
\[
    \sigma: MSU \to K\psb{q}
\]
for $SU$ bundles in
integral $K$-theory, with coefficients in $\Z\psb{q}$.  To explain
this, we introduce the total complex exterior and symmetric powers
\begin{align*}
     \Lambda_{t} (V) & = \sum_{i\geq 0}t^{i} \Lambda^{i} (V) \\
     S_{t} (V) & = \sum_{i\geq 0} t^{i} S^{i} (V).
\end{align*}
Using the identities ($V$ and $W$ are vector bundles, and $L$ is a
line bundle)
\begin{align*}
    \Lambda_{t} (V\oplus W) & \iso \Lambda_{t} (V)\Lambda_{t} (W)\\
    S_{t} (V\oplus W) & \iso S_{t} (V) S_{t} (W) \\
    S_{t} (L) &= 1 +tL + t^{2}L^{2} + \dotsb  \iso \Lambda_{-t} (L)^ 
{-1},
\end{align*}
we can define
\begin{align*}
    S_{t} (V-W) &= S_{t} (V)\Lambda_{-t} (W)\\
    \Lambda_{t} (V-W) &= \Lambda_{t} (V)S_{-t} (W)
\end{align*}
and so consider $S_{t}$ and $\Lambda_{t}$ to be exponential operations
\[
      K (X)\to K (X)\psb{t}.
\]

With this notation, the Euler
class of $V=L_{1}\oplus \dotsb \oplus L_{d}$ associated to the
orientation $\sigma$ is
\[
     e (V,q)  = \prod_{i} \sigma (L_{i},q) = \Delta_{-1} (V)\otimes  
\bigotimes
\Lambda_{-q^{n}} (V-\rank V) \bigotimes \Lambda_{-q^{n}}
(\Bar{V}-\rank V).
\]
That is, the orientation given by $\sigma$ is a twist of the $\Ahat$
orientation of Atiyah-Bott-Shapiro.  The associated genus is of an
$SU$-manifold $M$ is
\[
    \Ahat (M; \bigotimes_{n\geq 1} \sym_{q^{n}} (V-\rank V) \otimes  
\sym_{q^{n}}
(\Bar{V}-\rank V)),
\]
which is known as the \emph{Witten genus}.

In \cite{AHS:ESWGTC}, the authors define an \emph{elliptic spectrum}  
to be a
triple $(E,C,t)$, where $E$ is an even periodic ring spectrum (and so
complex-orientable), $C$ is an elliptic curve over $\pi_{0}E$, and $t$
is an isomorphism of formal groups
\[
         t: \spf E^{0}\cp  \iso \widehat{C}.
\]
They show that the data of an elliptic spectrum determine a map of
(non-equivariant) ring spectra
\[
    \sigma (E,C,t): \musix \to E.
\]
called the \emph{sigma orientation}.

The Tate curve is an elliptic curve $\CTate$ over $\Z\psb{q}$ which
provides an arithmetic model for the multiplicative uniformization of
a complex elliptic curve as
\[
         C \iso \C/\Lambda \iso \C^{\times}/q^{\Z},
\]
where $\Lambda= 2\pi i \Z + 2\pi i \tau \Z$ and $q=e^{2\pi i\tau}$.
It comes with an isomorphism of formal groups
\[
       \tTate: \Gmh \iso \CTateh,
\]
so $\KTate\eqdef (K\psb{q}, \CTate,\tTate)$ is an elliptic spectrum.
It turns out \cite[\S2.6,2.7]{AHS:ESWGTC} that the sigma orientation
of $\KTate$ is just the restriction to $\musix$ of the orientation
above: that is, the diagram
\[
\xymatrix{
{\musix}
  \ar[dr]^{\sigma (\KTate)}
  \ar[d]
\\
{MSU}
  \ar[r]_{\sigma}
&
{K\psb{q}}
}
\]
commutes.

\subsection{The Borel equivariant sigma orientation}
\label{sec:sigma-orientation}

If $V$ is a $G$-equivariant $SU$-bundle, then there is similarly an
equivariant Thom class
\begin{equation}\label{eq:18}
    \Thom_{G} (V)\eqdef \Thom (V\times_{G}EG) \in H^{*}_{G} (X^{V}),
\end{equation}
and we write
\begin{equation}\label{eq:19}
     e_{G} (V) \eqdef \zeta^{*}\Thom_{G} (V)\in H^{*}_{G} (X)
\end{equation}
for the associated Euler class.   In this
section we record some formulae for this class and some related
characteristic classes, in the case of the circle group.  In
Sections \ref{sec:edale} and \ref{sec:proof-prop-phi-A}, we use these  
formulae
to construct the equivariant sigma orientation.

Suppose that $V$ is an $\T$-bundle over an $\T$-fixed space, given as
\begin{equation} \label{eq:22}
    V \iso L_{1}\otimes \alpha_{1}  \oplus \dotsb \oplus L_{d}\otimes
\alpha_{d},
\end{equation}
where $L_{i}$ is a complex line bundle with Chern class $x_{i}$ and
$\alpha_{i}\in \T^{\vee}.$ Let $m_{i} = \log \alpha_{i} \in \Z.$  Then
\begin{equation}\label{eq:23}
     e_{\T} (V,\tau) = \prod_{i} \sigma (x_{i} + m_{i} z,\tau),
\end{equation}
where $z=c_{1}L\in H^{2} (B\T).$  Considering
$z$ to be the identity map of the complex plane defines maps
\[
    H^{*}(B\T) =\C [z] \rightarrow \O_{\C}\rightarrow \C \psb{z}=HP^0(B 
\T)
\]
and we observe that $e_{\T} (V)$ defines an element of
\[
      HP^{*} (X; \O_{\C}) \subseteq HP^*(X;\C \psb{z})\cong
H^{*} (X)\psb{z} \iso HP^{*} (X\times B\T).
\]
When working multiplicatively, we set $w=e^{z}.$


The manipulations that follow are more manageable if we adopt vector
notation, and abbreviate
\[
x= (x_{1},\dotsc,x_{d}).
\]
Similarly we'll write $u_{i}=e^{x_{i}}$ and
\[
    u = (e^{x_{1}},\dotsc ,e^{x_{d}}).
\]
If $x$ is such a vector, we define
\[
\sigma (x,\tau) \eqdef \prod_{j} \sigma (x_{j},\tau)
\]
and similarly for $\sigma (u,q),$ so $\sigma (u,q) = \sigma (x,\tau)$
as in the ``scalar'' case.  Then if
\[
   V \iso L_{1}\oplus\dotsb \oplus L_{d}
\]
with
\[
     x_{i} = c_{1}L_{i}
\]
and
\[
x = (x_{1},\dotsc ,x_{d}),
\]
then
\[
    e (V,\tau) = \sigma (x,\tau).
\]
If
\[
     u = (L_{1},\dotsc ,L_{d})
\]
then the corresponding $K$-theory Euler class is
\[
    e (V,q) = \sigma (u,q);
\]
and these are related by
\[
       e (V,\tau) = \ch e (V,q),
\]
where $\ch$ is the Chern character.

This notation is particularly helpful when we have to deal with the
equivariant Euler class.  Let $T\subset SU (d)\subset U (d)$ be the
standard maximal torus, and
let
\[
    \cochars = \Hom (\T,T)
\]
be its lattice of cocharacters: so
\[
     T = \{\diag (w_{1},\dotsc ,w_{d}) | \prod w_{i} = 1 \},
\]
and
\[
     \cochars \iso \{m \in \Z^{d} | \sum m_{i} = 0\}.
\]

Define
\begin{align*}
     I: & \cochars \times \cochars \to \Z \\
     \qf: & \cochars \to \Z
\end{align*}
by the formulae
\begin{align*}
     \qf (m) & = \frac{1}{2} \sum m_{i}^{2} \\
     I (m,m') & = \sum_{i} m_{i}m_{i}'.
\end{align*}
The important points about $\qf$ and $I$ are
\begin{equation} \label{eq:25}
\begin{split}
     \qf (0) & = 0 \\
     \qf (km) & = k^{2} \qf (m) \\
     \qf (m+m')& = \qf (m) + I (m,m') + \qf (m') \\
     \qf (wm) & = \qf (m)\\
     I (km,m') & = kI (m,m') = kI (m', m) \text{ etc.} \\
     I (wm,wm') & = I (m,m')
\end{split}
\end{equation}
for $m,m'\in \cochars$, $k\in \Z$, and $w\in W$.

\begin{Remark} Note that Lemma \ref{t-le-phi-I} shows that, for $SU
(d),$ the formulae for $\qf$ and $I$ here agree with those in
\S\ref{sec:chern-classes}.
\end{Remark}

As above, we continue to suppose that $x= (x_{1},\dotsc ,x_{d})$, and
$u = (u_{1},\dotsc ,u_{d})= (e^{x_{1}},\dotsc ,e^{x_{d}})$.
We define
\[
    I (x,m) = \sum m_{j} x_{j}.
\]
and
\[
    u^{I (m)} = \prod u_{i}^{m_{i}},
\]
so that
\[
    u^{I (m)} = e^{I (x,m)}.
\]

If $b$ is a scalar, then the meaning of
\[
     mb = bm = (m_{1}b,\dotsc, m_{d}b)
\]
is clear.  Its multiplicative analogue is
\[
     \beta^{m} = (\beta^{m_{1}},\dotsc ,\beta^{m_{d}});
\]
again these are related by
\[
      e^{mb} = (e^{b})^{m}.
\]

With these notations, the functional equations for $\sigma$ imply the
following.
\begin{Lemma}\label{t-fu-eq-sigma}
Suppose that $\lambda = 2\pi i l + 2\pi i k \tau$, that $x=
(x_{1},\dotsc ,x_{d})$, and $u=e^{x}$.  Suppose that $m\in \cochars$.
Then
\begin{align*}
  \sigma (x+m\lambda,\tau) &= e^{-k I (m,x) - 2\pi i \tau k^{2}\qf
(m)}
  \sigma (x,\tau) \\
  \sigma (uq^{km},q) & = u^{-kI (m)} q^{-k^{2}\qf (m)} \sigma
(u,q). \qed
\end{align*}
\end{Lemma}

\begin{Remark}
The factor of $(-1)^{l+k}$ in \eqref{eq:16} contributes $1$, because
it becomes
\[
    (-1)^{(l+k)\sum m_{i}} = 1.
\]
\end{Remark}

\begin{Remark} \label{rem-4}
To work with virtual vector bundles, we may as well extend our
abbreviations by using \emph{super}-vector notation, and so let
\[
   x = (x^{0},x^{1}), u = (u^{0},u^{1}), m = (m^{0},m^{1}),
\]
etc. stand for \emph{ordered pairs} of quantities as above.  So for  
example
\[
    \sigma (x,\tau) = \frac{\sigma (x^{0},\tau)}{\sigma (x^{1},\tau)}.
\]
\end{Remark}

Our first use of all this notation is to give the following result
about the equivariant Euler class associated to the sigma orientation.
It shows that that characteristic class restriction $\ChB_{1} = 0 =
\ChB_{2}$ suffices to ensure the Euler class descends to a meromorphic
function on the elliptic curve.  This observation goes back
at least to \cite{Witten:EllQFT,BottTaubes:Rig}.
Note that if $V$ is a $\T$-vector bundle over a $\T$-fixed space, then  
it
admits a decomposition
\[
     V \iso V^{\T}\oplus V',
\]
where $V'\iso V/V^{\T}.$

\begin{Proposition} \label{t-pr-sigma-euler-elliptic-busix}
Let $m= (m_{1},\dotsc ,m_{d}):\T \to T$ be a cocharacter,
corresponding to a component $BZ (m)$ of $BSU (d)^{\T},$ and let $\tSUd 
$ be the tautological
$\T$-equivariant vector bundle over this space.  Let $x= (x_{1},\dotsc
,x_{d})$ be the roots of the total Chern class of the tautological
bundle $\tSUd$ over $BSU (d)^{\T}.$ Then
\[
      e_{\T} (\tSUd) (z,\tau) = \sigma (x+ m z,\tau) = \sigma (uw^{m},q)
      \in      HP^{*}_{\T} (BSU (d)^{\T};\O_{\C})
\]
and
\[
      e_{\T} (\tSUd/\tSUd^{\T}) (z,\tau) = \prod_{m_{j}\neq 0} \sigma  
(x_{j}+ m_{j}
      z,\tau) = \prod_{m_{j}\neq 0} \sigma (u_{j}w^{m_{j}},q) \in
      HP^{*}_{\T} (BSU (d)^{\T};\O_{\C})
\]
These elements satisfy
\begin{align*}
      e_{\T} (\tSUd) (z + \lambda,\tau) & =
      \exp \left( -k I (x,m) - k I (m,m) z - 2\pi i k^{2}\qf(m) \tau  
\right)
      e_{\T} (\tSUd) (z,\tau) \\
      e_{\T} (\tSUd/\tSUd^{\T}) (z + \lambda,\tau) & =
      \exp \left( -k I (x,m) - k I (m,m) z - 2\pi i k^{2}\qf(m) \tau  
\right)
      e_{\T} (\tSUd/\tSUd^{\T}) (z,\tau) \\
     \end{align*}
if $\lambda=2\pi i l + 2\pi i k \tau$; in $q$-notation this is
\begin{align*}
      e_{\T} (\tSUd) (wq^{k},q) &=
      u^{-k I (m)} w^{-kI (m,m)} q^{-k^{2}\qf (m)} e_{\T} (\tSUd)
      (w,q) \\
     e_{\T} (\tSUd/\tSUd^{\T}) (wq^{k},q) &=
      u^{-k I (m)} w^{-kI (m,m)} q^{-k^{2}\qf (m)} e_{\T} (\tSUd/\tSUd^ 
{\T})
      (w,q)
\end{align*}
In particular, if $V=V_{0} - V_{1}$ is a $\T$-equivariant bundle over
a $\T$-fixed space $X$, with $\ChB_{1}V = 0 = \ChB_{2} V$, then
\[
e_{\T} (V/V^{\T}) (z,\tau) \in HP^{*}_{\T} (X;\cK_{\J}^{\times})
\]
\end{Proposition}

\begin{proof}
The formula for the equivariant Euler class is just
\eqref{eq:23}.  The transformation formula follows from Lemma
\ref{t-fu-eq-sigma}.  Note that the entries $m_{j}=0$ make no
contribution to $I (x,m)$ or $\qf (m).$   Lemma
\ref{t-le-equiv-ch-classes-for-covers}  then shows that if $V = V_{0}
- V_{1}$ is a $\bulrp{6}$-bundle, then  $e_{\T} (V/V^{\T})$ descends
to $\cK_{\J}$; it remains to show that it is non-zero.  We have
\begin{align*}
    e_{\T} (V/V^{\T}) &=     \frac{e_{\T} (V_{0}/V_{0}^{\T})}
                                 {e_{\T} (V_{1}/V_{1}^{\T})} \\
& = \frac{\prod_{m^{0}_{j}\neq 0} \sigma (x^{0}_{j}+ m^{0}_{j} z,\tau)}
          {\prod_{m^{1}_{j}\neq 0} \sigma (x^{1}_{j}+ m^{1}_{j} z, 
\tau)}.
\end{align*}
Each factor of the product is of the form $\sigma (x +mz,\tau)$.
This is holomorphic, and
the cohomology class $x$ takes integer values on homology classes,
so it has zeroes  only at points of finite order (i.e., when
a multiple of $z$ is a lattice point). Accordingly,
the product takes values which are invertible meromorphic functions.

\end{proof}

\subsection{The Thom class} \label{sec:thom-class}

In this section we give the formula for our Thom class, although the
proof that it works as we say depends on some results in
Section \ref{sec:proof-prop-phi-A}.

Let $\tStringc$ be the tautological bundle over $\bstringc$, so that
\[
     \mstringc = \bstringc^{\tStringc}.
\]
Recall that $\bstringc$ and so also $\mstringc$ has all relevant  
cohomology in
even degrees (Proposition \ref{prop:rigideven}).  In this section and
the next, we will construct and analyze
a class in $\EJc (\bstringc^{\tStringc})$ using   the exact sequence
\begin{multline}\label{eq:79}
0 \rightarrow   \EJ^{2n}_{\T}(\bstringc^{\tStringc}) \lra HP^{2n} 
((\bstringc^{\tStringc})^{\T} ; \cK) \times \prod_\A HP_{\T}^{2n} 
((\bstringc^{\tStringc})^{A}; \OA) \rightarrow \\
   HP^{2n}((\bstringc^{\tStringc})^{\T} ; \cKF)
\rightarrow \EJ^{2n+1}_{\T}(\bstringc^{\tStringc}) \rightarrow 0
\end{multline}
of Proposition \ref{prop:Hasse}.   We shall specify
\[
    \stror_{\T} (\tStringc) \in HP^{0}_{\T} ((\bstringc^{\T})^ 
{\tStringc^{\T}};\cK_{\J})
\]
and, for each finite $\A\subseteq \T$, an element
\[
     \stror_{\A} (\tStringc)\in HP^{0}_{\T} ((\bstringc^{\A})^ 
{\tStringc^{\A}};\OA),
\]
with the property that
\begin{equation} \label{eq:70}
     \stror_{\A} (\tStringc)\restr{(\bstringc^{\T})^{\tStringc^{\T}}}  
= \stror_{\T} (\tStringc)
\end{equation}
in
\[
      HP^{0}_{\T} ((\bstringc^{\T})^{\tStringc^{\T}};\cKF).
\]

The obvious way to produce such a class is to
start with
\begin{equation} \label{eq:71}
      \Thom_{\T} (\tStringc) \in HP^{0}_{\T} (\bstringc^{\tStringc}),
\end{equation}
and then for each $\A$ to pull back along
\[
      (\bstringc^{\A})^{\tStringc^{\A}} \xra{} (\bstringc^{\A})^ 
{\tStringc} \xra{}
      \bstringc^{\tStringc}.
\]
Thus for all $\A$, finite or not, the formula for $\stror_{\A}$ is
\begin{equation}\label{eq:20}
      \stror_{\A} (\tStringc) = \Thom_{\T} (\tStringc^{\A}) e_{\T}  
(\tStringc/\tStringc^{\A}).
\end{equation}
Here $\Thom_{\T} (\tStringc^{\A})$, as defined in \eqref{eq:18}, is  
the Thom
class using $\sigma$ of the $\T$-Borel construction of $\tStringc^{\A} 
$, and
$e_{\T} (\tStringc/\tStringc^{\A})$ is the Euler class, again using $ 
\sigma$, of
the $\T$-Borel construction of the ``complement'' $\tStringc-\tStringc^ 
{\A} \iso
\tStringc/\tStringc^{\A}$.

As written, our formula for $\stror_{\T}$ gives an element of
\[
     HP^{0} ((\bstringc^{\T})^{\tStringc^{\T}};\cK_{\C}),
\]
and the first observation is that it descends to an element of
$HP^{0}_{\T}((\bstringc^{\T})^{\tStringc^{\T}};\cK_{\J})$ as required.

\begin{Lemma}\label{t-le-phi-t}
The class $\stror_{\T} (\tStringc)$ gives an element of
\[
         HP^{0}_{\T} ((\bstringc^{\T})^{\tStringc^{\T}};\cK_{\J}),
\]
and multiplication by $\stror_{\T} (\tStringc)$ is an isomorphism
\[
       HP^{*}_{\T} (\bstringc^{\T};\cK_{\J}) \xrightarrow[\iso]{\stror_ 
{\T}}
               HP^{*}_{\T} ((\bstringc^{\T})^{\tStringc^{\T}};\cK_{\J}).
\]
\end{Lemma}

\begin{proof}
It is an isomorphism because $\Thom_{\T} (\tStringc^{\T})=\Thom  
(\tStringc^{\T})$
is a Thom class,
and in Proposition \ref{t-pr-sigma-euler-elliptic-busix} it was shown  
that
$e_{\T} (\tStringc/\tStringc^{\T})$ is a unit of $HP^{*}_{\T}  
(\bstringc^{\T};\cK_{\J})$.
\end{proof}

We turn now to $\stror_{\A}$ for $\A\subset \T$ the finite subgroup of
order $n$.  Our formula gives an element  of
\[
      HP^{0}_{\T} (\bstringc^{\A})^{\tStringc^{\A}};\cK_{\C})
\]
and we produce from it an element of
\[
    HP^{0}_{\T} ((\bstringc^{\A})^{\tStringc^{\A}}, \cK_{\A})
\]
simply by choosing, for each point $a$ of order $n$, a lift $\ta$,
and then electing to evaluate our element of $\cK_{\C}$ near $\ta$.
Of course the apparent dependence on arbitrary choices is not
satisfactory.  In Section \ref{sec:proof-prop-phi-A}, we shall prove
the following result, which requires a delicate and intricate
analysis of the Euler class functions.

\begin{Proposition}\label{t-pr-phi-A-indep-lift}
The value of $\stror_{\A} (\tStringc)$ at $\ta\in \C$ depends only on  
the
image $a$ of $\ta$ in $\J$.  As such, multiplication by $\stror_{A}
(\tStringc)$ is an isomorphism
\[
       HP^{*}_{\T} (\bstringc^{\A};\cO_{\A}^{\wedge}) \xrightarrow 
[\iso]{\stror_{\A}}
               HP^{*}_{\T} ((\bstringc^{\A})^{\tStringc^{\A}};\cO_{\A}^ 
{\wedge}).
\]
\end{Proposition}

Thus we have the following.

\begin{Theorem} \label{t-th-sigma-equiv}
The classes $\stror_{\A} (\tStringc)$ for $\A\subseteq \T$
assemble to give a class $\stror (\tStringc) \in \EJc (\bstringc^ 
{\tStringc}) = \EJc (\mstringc)$, and
multiplication by $\stror (\tStringc)$ is an isomorphism
\[
    \EJc^{*} (\bstringc) \xrightarrow[\iso]{\stror (\tStringc)}
    \EJc^{*} (\mstringc).
\]
\end{Theorem}

\begin{proof}
Lemma \ref{t-le-phi-t} and Proposition
\ref{t-pr-phi-A-indep-lift} together with the exact sequence
\eqref{eq:79} show that we have assembled an element of $\stror  
(\tStringc)$
of $\EJc (\bstringc^{\tStringc}).$  Moreover, using the exactness of  
\eqref{eq:79},
its analogue for $\EJc (\bstringc),$ the Thom isomorphism in ordinary
cohomology, and the Five Lemma, we may conclude
that multiplication by $\stror (\tStringc)$ is an isomorphism.
\end{proof}

\subsection{Multiplicativity}
\label{sec:multiplicativity}

Theorem \ref{t-th-sigma-equiv} gives a map of spectra
\[
     \stror: \mstringc \lra \EJ.
\]
By Proposition   \ref{t-pr-mstring-einfty}  $\mstringc$
is an $E_{\infty}$ ring spectrum, and we would like to know that map
is multiplicative.

\begin{Theorem}
The map $\stror: \mstringc \lra \EJ$ is a ring map up to homotopy.
\end{Theorem}

\begin{proof}
The product on $\mstringc$ arises from the formula
\[
     (X\times Y)^{V\oplus W} \iso X^{V} \wedge  Y^{W},
\]
and the map $\stror$ is multiplicative because it arises from the
exponential class $\Thom_{\T}.$
\end{proof}

\section{Translation of the Thom class by a point of order $n$}
\label{sec:proof-prop-phi-A}

In this section we assemble a proof of Proposition
\ref{t-pr-phi-A-indep-lift}.  The formula \eqref{eq:20} for
$\stror_{\A}$ gives an element of
\[
     HP_{\T}^{0} ((\bstringc^{\A})^{\tStringc^{\A}};\cK_{\C}),
\]
where $\tStringc$ is the universal bundle,
and we must show that this element descends to $\cK_{\A}$ and is
holomorphic near $a\in \Jlr{\A}.$ This requires detailed understanding
of transformation
properties and  orders of vanishing for functions derived from the  
Weierstrass
$\sigma$ function, and is therefore one of the most intricate parts of  
the
proof.

\subsection{The strategy via translation}

The cohomology ring $H^{*}(B\T)$ is the ring of functions on the
completion of $\C$ at the origin, so to study $\stror_{\A}$ near a point
$\ta\in \Clr{\A} = \pi^{-1} (\Jlr{\A})$, we study
$T_{\ta}^{*}\stror_{\A}$ near zero.  Let $\tStringc$ denote the  
tautological
bundle over $\bstringc$.  The essential problem is to understand
the Euler class
\[
      \zeta^{*} T_{\ta}^{*}\stror_{\A} (\tStringc) \in HP^{*}_{\T}  
(\bstringc^{\A})
\]
near zero. Here $\zeta$ denotes the zero section, and $T$ will denote
translation in $\C$ or $\J.$  Let $\tSU$ denote the tautological bundle
over $BSU$.  Using the description of  $H^{*}(BSU(d)^{\A})$ in
\S\ref{sec:reduct-splitt-princ},  we introduce a characteristic class
\[
    \delta_{\A} (\tSU): \Clr{\A} \xra{} HP^{*}_{\T} (BSU^{\A};\cK_{\C}),
\]
with the property that
\[
     \delta_{\A} (\tSU,\ta) = \zeta^{*}T_{\ta}^{*}\stror_{\A} (\tSU).
\]
The explicit formula for $\delta_{\A}$ makes it possible to
prove Proposition \ref{t-pr-phi-A-indep-lift}.  For example, we show
that $\delta_{\A} (\tStringc,\ta)$
depends only on $\pi (\ta)\in \Jlr{\A}$: that is, we have a
factorization
\[
\xymatrix{
{\Jlr{\A}}
  \ar@{-->}[r]^-{\delta_{\A}} &
{HP^{*}_{\T}\bstringc^{\A}} \\
{\Clr{\A}}
  \ar[u]^{\pi}
  \ar[r]^-{\delta_{\A}}
&
{HP^{*}_{\T}BSU^{\A}.} \ar[u]
}
\]

This argument by translation was introduced by \cite{BottTaubes:Rig},
and its use in Grojnowski's equivariant elliptic cohomology goes back
to Rosu \cite{Rosu:Rigidity}.  The class $\delta_{\A}$ was introduced
in \cite{Ando:AESO}, to show that the translation argument
could be made independent of the complicated choices in
\cite{BottTaubes:Rig,Rosu:Rigidity}.

As we note in the introduction, earlier treatments of the Thom class
required the translation argument even to give a formula for
$\stror_{\A}$.  Our formula \eqref{eq:20} (essentially, \eqref{eq:71})
does not involve the translation argument, and so is much simpler than
earlier formulae.  What we must do is adapt the translation argument
to show that our $\stror_{\A}$ has the required
properties.

\subsection{The class $\deltaA$} \label{sec:class-delta_a}

Suppose that $\A=\T[n]$ is a finite subgroup of the circle.  Let
$\tSUd$ be the tautological bundle  over $BSU (d)$.
For $\ta \in \Clr{\A}$, we define
an element $\delta_A(\tSUd,\ta)$ of $H^{*} (BSU (d)^{\A};\O_{\C})$ as  
follows.
Suppose that
\[
m= (m_{1},\dotsc ,m_{d}): \A\to T
\]
is a homomorphism, corresponding to a component of $BSU (d)^{\A}.$
Thus $m_{i}\in A^{*}\iso \Z/n,$ and
\[
    \sum m_{i}\equiv 0 \mod n.
\]
Suppose that $x_{j}$ are
the Chern roots of $\tSUd$ with respect to this decomposition: that  
is, we
suppose that we have a splitting
\[
\tSUd \iso (L_{1}\otimes \C (m_{1})) \oplus\dotsb \oplus (L_{d}\otimes  
\C (m_{d})),
\]
where
\[
   m_j=\log \alpha_j \mbox{ and }   x_{j} = c_{1} L_{j}.
\]

Let $a$ be a point of $C$ of order $n$, and let $\ta$ be a lift of $a$  
to
$\C$.  Define integers $k$ and $l$ by the formula
\[
     n \ta = \lambda = 2\pi i l + 2\pi i k \tau,
\]
and let $\lambda=2\pi i l + 2\pi i k \tau$.  Let $\tm$ be any
factorization
\begin{equation}\label{eq:58}
\xymatrix{ {\T} \ar[r]^-{\tm} &
{T.} \\
{\A} \ar[u] \ar[ur]_{m} }
\end{equation}
That is, $\tm_{j}$ is an integer lift of $m_{j}.$

In order to give $q$-expansion formulae we also set
\begin{align*}
     u^{r}  & = e^{rx}\\
     \alpha^{r} & = e^{r\ta}
\end{align*}
for $r\in \Q.$

\begin{Definition}\label{def-delta-A-fn}
Let $\deltaA (x, \tm,\ta)$ be the expression
\begin{equation} \label{eq:59}
\begin{split}
    \deltaA (x,\tm,\ta) & \eqdef \exp\bigl(\tfrac{k}{n} I (\tm,x) +
\tfrac{k}{n}\ta\qf (\tm)\bigr)
  \sigma (x+\tm\ta, \tau)\\
  & = \exp \bigl(\tfrac{k}{n}\sum_{j} \tm_{j} x_{j} +
\tfrac{k}{n}\tfrac{\ta}{2}\sum_{j} \tm_{j}^{2}\bigr) \prod_{j} \sigma  
(x_{j}
+ \tilde{m_{j}}\ta,\tau) \\
    & = u^{\frac{k}{n}I (\tm)} \alpha^{\frac{k}{n}\qf (\tm)}\sigma
    (u\alpha^{\tm},q).
\end{split}
\end{equation}
Just as in the case of $\sigma$, it will be convenient to express
$\deltaA$ equivalently in terms multiplicative variables: we'll use
$\alpha=e^{\ta}$ in the last slot to signal multiplicative notation
when necessary: so
\[
    \deltaA (u,\tm,\alpha) = \deltaA(e^{x},\tm,\alpha) = \deltaA (x, 
\tm,\ta).
\]
\end{Definition}

\begin{Lemma}\label{t-le-delta-invt-m}
The expression $\deltaA$ is independent of the choice of lift
$\tm$.  Moreover, it is invariant under the action of $W (m)$, the
Weyl group of $Z (m)$.  As such, it defines a characteristic class of
principal $Z (m)$-bundles.
\end{Lemma}

\begin{proof}
Suppose that $\tm'$ is another lift.  Then
\[
      \tm' = \tm + n \Delta,
\]
where $\Delta\in \Hom (\T,T).$  Then
\begin{align*}
   \deltaA (x,\tm',\ta) = &
\exp \bigl(\tfrac{k}{n}I (\tm+n\Delta,x) + \tfrac{k}{n}\ta\qf (\tm+n 
\Delta)\bigr)
\sigma (x+ (\tm+n\Delta)\ta,\tau)\\
  = & \exp \bigl(\tfrac{k}{n}I (\tm,x) + kI (\Delta,x) +
\tfrac{k}{n}\ta\qf(\tm) + k \ta I (\tm,\Delta) + k n \ta\qf (\Delta)
\bigr)
\sigma (x+ \tm\ta + \Delta \lambda,\tau) \\
= & \exp \bigl(\tfrac{k}{n}I (\tm,x) + kI (\Delta,x) +
\tfrac{k}{n}\ta\qf(\tm) + k \ta I (\tm,\Delta) + k \lambda\qf
(\Delta)
      \bigr) \\
& \exp \bigl(- I (x,k\Delta) - I (\tm,k\Delta) \ta
       - 2\pi i \tau \pi \qf(n\Delta)\bigr)  \sigma (x+\tm\ta,\tau)\\
= & \deltaA (x,\tm,\ta).
\end{align*}
Now suppose that $w\in Z (m)$ so $wm=m$: then
\[
       w\tm = \tm + n \Delta
\]
for some $\Delta\in \cochars$, and a similar argument to the one just
given shows that
\[
    \deltaA (x,\tm,\ta) = \deltaA (wx,w\tm,\ta) = \deltaA
(wx,\tm+n\Delta, \ta) = \deltaA (wx,\tm,\ta).
\]
\end{proof}

As a related matter, it is easy to understand the action of a general
$w\in W$ (i.e. one which does not necessarily fix $m$).

\begin{Proposition}\label{t-pr-delta-char-class}
For $w\in W$, the Weyl group of $SU (d)$, we have
\[
    \deltaA (wx,wm,\ta) = \deltaA (x,m,\ta).
\]
That is, the family of expressions $\delta (x,m,\ta)$ for $m\in \Hom
(\A,T)$ satisfies the hypothesis of Proposition
\ref{t-pr-splitting-princ-messy}, and so assembles to give an element
$\delta_A(\tSUd,\ta)$ of $HP^{*}(BSU (d)^{\A})$. \qed
\end{Proposition}

\begin{Remark} \label{rem-3}
As noted in Remark \ref{rem-4}, the appropriate extension to virtual  
vector bundles is
given by the same formulae, provided we admit $\Z/2$-graded notation.
Thus if $x= (x^{0},x^{1})$, $m= (m^{0},m^{1})$, and $\tm = (\tm^{0},
\tm^{1})$, then
\[
      \deltaA (x,\tm, \ta) = \delta (x^{0},\tm^{0},\ta)/ \delta
(x^{1},\tm^{1},\ta).
\]
\end{Remark}

\begin{Definition}\label{def-delta-A}
If $V$ is a virtual $\T$-equivariant vector bundle over an
$\A=\T[n]$-fixed space $X$, with
\[
\ChB_{1} (V) = 0,
\]
and if $\ta$ is a point of $\C$ over a point $a$ of order $n$ in $C$,
then we write
\[
    \deltaA (V,\ta)
\]
for the class in $H^{*} (X)$ provided by Proposition
\ref{t-pr-delta-char-class}.
\end{Definition}

Now we investigate the dependence of $\deltaA$ on the lift $\ta$.
Suppose $\tap$ is another lift of $a$.  Then there are
$r,s\in \Z$ such that
\[
     \tap = \ta + 2\pi i r + 2\pi i s \tau,
\]
so
\[
     e^{\tap} = e^{\ta}q^{s},
\]
and
\[
     e^{n\tap} = q^{k+ns}.
\]
Let $k'=k+ns$.  Then
\begin{equation}\label{eq:43}
      w (a,q^{\frac{1}{n}}) \eqdef e^{-\ta+ \frac{k}{n}\tau} =
e^{-\tap+\frac{k'}{n}\tau}
\end{equation}
is an $n^{\th}$ root of unity which does not depend on the choice of
lift $\ta$; in fact it is the Weil pairing of $a$ with $q^{1/n}$
\cite[p. 90]{KaMa:AMEC}.

Now let $m: A \to T$ be a homomorphism, labelling the component $BZ
(m)$ of $BSU (d)^{A}$, and let $\tSUd$ be the tautological bundle over  
$BZ
(m).$  Let $\tm: \T \to T$ be any lift of $m$, as in \eqref{eq:58}.
Because $w (a,q^{\frac{1}{n}})$ is an $n^{\th}$ root of unity,
the quantity
\[
    w (a,q^{\frac{1}{n}})^{\qf (m)} \eqdef w (a,q^{\frac{1}{n}})^{\qf
(\tm)}
\]
does not depend on the lift $\tm$.    The dependence of
$\deltaA$ on the choice of lift $\ta$ is given by the following.

\begin{Lemma} \label{t-le-delta-A-dep-ta}
\[
\deltaA(\tSUd,m,\tap) = w (a,q^{\frac{1}{n}})^{s\qf (m)}
\deltaA (\tSUd,m,\ta) \in HP^{*} (BZ (m)^{\A}).
\]
\end{Lemma}

\begin{proof}
Let $\alpha=e^{\ta}.$ In $q$-notation,
\begin{align*}
\deltaA (\tSUd,\tm,\tap) & = u^{\frac{k'}{n}I (\tm)} (\alpha
q^{s})^{\frac{k'}{n}\qf (\tm)}
\sigma (u \alpha^{\tm} q^{s \tm},q) \\
& = u^{\frac{k}{n}I (\tm)} u^{s I (\tm)} \alpha^{\frac{k}{n}\qf
(\tm)} \alpha^{s \qf (\tm)} q^{s \frac{k}{n}\qf (\tm)}
q^{s^{2}\qf (\tm)} u^{-s I (\tm)} \alpha^{-s I
(\tm,\tm)} q^{-s^{2}\qf (\tm)}
\sigma (u\alpha^{\tm},q) \\
& = \deltaA (V,\tm,\ta) \alpha^{- s \qf (\tm)} q^{s
\frac{k}{n}\qf (\tm)} \alpha^{-s I (\tm,\tm)}.
\end{align*}
Noting that
\[
    \qf (\tm) - I (\tm,\tm) = - \qf (\tm),
\]
the last expression becomes
\[
\deltaA (\tSUd,\tm,\ta) \alpha^{-s \qf (\tm)} q^{s
\frac{k}{n}\qf (\tm)} \alpha^{-s I (\tm,\tm)} = w
(a,q^{\frac{1}{n}})^{s \qf (m)} \deltaA (\tSUd,\tm,\ta).
\]
\end{proof}

Recall that $\tStringc$ is the tautological bundle over $\bstringc$.

\begin{Proposition}\label{t-pr-delta-inv-busix}
If $V$ is a virtual $\T$-equivariant $SU$-bundle with $\ChB_{2} (V) =
0$, then class $\deltaA (V,\ta)$ does not depend on the choice of
lift $\ta$ of $a$.  Equivalently, the following statement holds 
 in the universal case:  for any two lifts $\ta$ and $\tap$
of $a$,
\[
     \deltaA (\tStringc,\ta)  = \deltaA
(\tStringc,\tap) \in HP^{0} (\bstringc^{\A}).
\]
\end{Proposition}

%
%

\begin{proof}
Lemma \ref{t-le-equiv-ch-classes-for-covers} implies that, if $m$ is
any reduction of the action of $\A$ on $V$, and if $\tm$ is a lift of
$m$, then
\[
     \qf (\tm) \equiv 0 \mod n,
\]
so $w (a,q^{1/n})^{\qf (m)} = 1.$
\end{proof}


It is important that the class $\deltaA$ has a Borel-equivariant
version as well.  For if $V$ is a $\T$-equivariant bundle over an
$\A$-fixed space $X$, then the $\T$-action preserves the decomposition
into isotypical summands for the $\A$-action
\[
     V \iso \bigoplus_{S \in \A^{\vee}} S \otimes \Hom
(S,V).
\]
and so the reduction $m$ determines a \emph{$\T$-equivariant}  
principal $Z
(m)$-bundle over $X$.  Put another way, the $\T$-action on $BSU
(d)^{A}$ determines one on the component $BZ (m)$, and as such the
map classifying the Borel construction of the tautological bundle
factors as
\begin{equation}\label{eq:44}
\xymatrix{
{E\T\times_{\T} BZ (m)}
  \ar[r]
  \ar[dr]
&
{BZ (m)}
\ar[d]
\\
&
{BSU (d)^{\A}.}
}
\end{equation}
We write
$\deltaBA (V,m,\ta)$
for the resulting Borel class, in $H_{\T}^{*} (X;\O_{\C})$.

It is important to understand the restriction of $\deltaBA$ to
the fixed subspace
$Y=X^{\T}.$  Since $Y\subseteq X$, any reduction
\[
\tm: \pi_{0}Y \rightarrow \Hom (\T,T)
\]
of the action of $\T$ on
$V\restr{Y}$ is a lift of $m$.  If $x = (x_{1},\dotsc ,x_{d})$
are the Chern roots of $V\restr{Y}$, then
\begin{equation}\label{eq:14}
\begin{split}
    \deltaBA (V,\ta)\restr{Y} = \deltaBA (x,\tm,\ta) & =
\exp \bigl(\tfrac{k}{n} I (\tm,x+\tm z) + \tfrac{k}{n}\ta\qf (\tm)\bigr)
  \sigma (x+\tm z + \tm\ta,\tau)
\end{split}
\end{equation}

As promised, we can now show that $\deltaBA$ is the translation
of the equivariant Euler class associated to $\sigma$.

\begin{Proposition}\label{t-pr-delta-fixed-busix}
Let $\tStringc$ be the tautological bundle restricted to the fixed
point set $\bstringc^{\A}$.  Then
\[
   \deltaBA (\tStringc,\ta) = T_{\ta}^{*}e_{\T}(\tStringc) \in HP_{\T}^ 
{0}
   (\bstringc^{\A}; \cK_{\C}).
\]
\end{Proposition}

\begin{proof}
We showed in Proposition \ref{prop:rigideven} that
\[
     H_{\T}^{*} (\bstringc^{\A};\cK_{C}) \rightarrow
     H_{\T}^{*} (\bstringc^{\T};\cK_{C})
\]
is injective, and so it suffices to prove that
\[
        \deltaBA (\tStringc,\ta)\restr{\bstringc^{\T}} =
        T_{\ta}^{*}e_{\T} (\tStringc)\restr{\bstringc^{\T}}.
\]
But under the indicated characteristic class restrictions, we have $I
(\tm,x) = 0$ and $\qf (\tm) = 0$ by Lemma
\ref{t-le-equiv-ch-classes-for-covers}, and so equation \eqref{eq:14}
becomes
\[
    \deltaBA (\tStringc,\ta)\restr{\bstringc^{\T}} = \sigma (x+\tm z +  
\tm \ta,\tau)
    = T_{\ta}^{*}e_{\T} (\tStringc)\restr{\bstringc^{\T}}.
\]
\end{proof}

\subsection{Variants} \label{sec:variants}

We need two variants of $\deltaA$, corresponding to the
decomposition
\[
    V \iso V^{\A}\oplus V',
\]
where $V'\iso V/V^{\A}.$  To give formulae we introduce some
restricted sums and products.

Let
\begin{align*}
     \spri_{i} \tm_{i} y_{i} & = \sum_{\tm_{i}\not\equiv 0 \mod n}
     \tm_{i} y_{i} \\
  \propri_{i} f (y_{i}) & = \prod_{\tm_{i}\not\equiv 0 \mod n} f (y_ 
{i})\\
     I' (m,x) & \eqdef \spri_{i} \tm_{i} y_{i} \\
     \qf' (\tm)&  = \frac{1}{2} \spri_{i} \tm_{i} \tm_{i} \\
     \spp_{i} \tm_{i} y_{i} & =
            \sum_{\tm_{i} \equiv 0 \mod n}  \tm_{i} y_{i} \\
     I'' (m,x) & = \spp_{i} \tm_{i} y_{i} \\
     \qf'' (\tm) &  = \frac{1}{2} \spp_{i} \tm_{i} \tm_{i} \\
     \propp_{i} f (y_{i}) & = \prod_{\tm_{i}\equiv 0 \mod n} f
(y_{i}), \\
\end{align*}
and let
\begin{align*}
     \sigma' (y,\tau) & = \propri_{i}\sigma (y_{i},\tau) \\
     \sigma'' (y,\tau) & = \propp_{i}\sigma (y_{i},\tau) \\
     \deltaA' (x,\tm,\ta) & = \exp (\frac{k}{n}I' (\tm,x) +
\frac{k}{n}\ta \qf' (\tm))
     \sigma' (x+\tm \ta,\tau)\\
     \deltaA'' (x,\tm,\ta) & = \exp (\frac{k}{n}I'' (\tm,x) +
\frac{k}{n}\ta \qf'' (\tm)) \sigma'' (x + \tm \ta,\tau).
\end{align*}

Notice that
\begin{align*}
     I & = I' + I'' \\
     \qf & = \qf' + \qf'' \\
     \deltaA &= \deltaA' \deltaA''.
\end{align*}

Our analysis of $\deltaA$ applies to $\deltaA'$ and
$\deltaA''$ to give the following.

\begin{Proposition}\label{t-pr-delta-A-primes}
The classes $\deltaA' (x,\tm,\ta)$ and $\deltaA'' (x,\tm,\ta)$
are independent of the choice of lift $\tm$, and are invariant under
the action of $W (m)$.  Moreover if $w\in W$, then
\begin{align*}
      \deltaA'(wx,wm,\ta) = \deltaA' (x,m,\ta),
\end{align*}
and similarly for $\deltaA''$.  As such, they assemble to give
elements $\deltaA' (\tSUd,\ta)$ and $\deltaA'' (\tSUd,ta)$ of $HP^{*}  
(BSU (d)^{\A}).$ As $d$ varies, they define stable
exponential classes $\deltaA' (\tSU,\ta)$ and $\deltaA''
(\tSU,\ta)$ in $HP^{*} (BSU^{\A})$.
\end{Proposition}

\begin{proof}
The arguments for $\deltaA$ in \S\ref{sec:class-delta_a} decouple
in this way.  The main point is, if $V$ is an $\A$-bundle or
$\T$-bundle over an $\A$-fixed space, then with respect to the
equivariant decomposition
\[
       V \iso V^{\A} \oplus V',
\]
the ``prime'' parts above correspond to $V'$, while the
``prime-prime'' parts correspond to $V^{\A}$.
\end{proof}

The behaviour of $\deltaA'$ and $\deltaA''$ with respect to
change from $\ta$ to $\tap$ is similar, but there is an additional
subtlety.  First of all, note the following.

\begin{Lemma} \label{t-le-phi-pp-mod-n}
For any lift $\tm$ of $m$, we have
\[
     \qf'' (\tm) \equiv 0 \mod n.
\]
\end{Lemma}

\begin{proof}
Recall that $\qf''$ corresponds to restriction to $V^{\A},$ where each
$\tm_{i}\equiv 0 \mod n.$  We have
\[
    \qf'' (\tm) = -\spp_{i<j} \tm_{i}\tm_{j}.
\]
\end{proof}

\begin{Proposition} \label{t-pr-delta-prime-abar}
If $\tap$ is another lift of $a$, and $\delta$ is defined by
\[
     e^{\tap} = e^{\ta} q^{\delta},
\]
then
\begin{align*}
\deltaA'(\tSUd,m,\tap) & = w (a,q^{\frac{1}{n}})^{\delta\qf' (m)}
\deltaA' (\tSUd,m,\ta)\\
& = w (a,q^{\frac{1}{n}})^{\delta\qf (m)}
\deltaA' (\tSUd,m,\ta) \\
\intertext{and}
\deltaA''(\tSUd,m,\tap) & =\deltaA''(\tSUd,m,\ta)
\end{align*}
in $HP^{*} (BZ (m)).$
In particular, we have a well-defined characteristic class
$\deltaA' (\tStringc,a)$ of $\bstringc$-bundles, and a well-defined
characteristic class $\deltaA'' (\tSU,a)$ of $BSU$-bundles. \qed
\end{Proposition}

The fact that, even for an $SU$-bundle, $\deltaA'' (V,\ta)$ does
not depend on the choice of lift $\ta$ is striking, until it is
discovered to be trivial.  Recall from \eqref{eq:17} that, if $V$ is a
(virtual) vector bundle, then $e (V)$ is its Euler class with respect
to the orientation given by $\sigma$.

\begin{Proposition}\label{t-pr-delta-pp}
Let $\tSUd$ be the tautological bundle over $BSU (d).$ For any choice $ 
\ta$
of lift of $a$,
\[
   \deltaA'' (\tSUd,\ta) = e (\tSUd^{\A}).
\]
and
\[
   \deltaBApp (\tSUd,\ta) = e_{\T} (\tSUd^{\A})
\]
\end{Proposition}

\begin{proof}
Let $\alpha=e^{\ta}$, so
\[
      \alpha^{n} = q^{k}.
\]
Let $\tm$ be a lift of $m: \A\to T$, and define integers $\Delta_{i}$
by the rule
\[
      \Delta_{i} = \begin{cases}
\frac{m_{i}}{n} & m_{i} \equiv 0 \mod n \\
0 & \text{otherwise}.
\end{cases}
\]
Let
\[
    u = (e^{x_{1}},\dotsc ,e^{x_{d}}).
\]
We use $q$-notation.
\begin{align*}
     \deltaA'' (x,\tm,\ta) & = u^{\frac{k}{n}I''
(\tm)}\alpha^{\frac{k}{n}\qf'' (\tm)}
     \sigma'' (u\alpha^{\tm}) \\
& = u^{\frac{k}{n}I'' (n\Delta)}\alpha^{\frac{k}{n}\qf'' (n\Delta)}
     \sigma'' (u\alpha^{n\Delta}) \\
& = u^{kI'' (\Delta)}\alpha^{k n \qf'' (\Delta)}
     \sigma'' (uq^{k\Delta}) \\
& = \sigma'' (u),
\end{align*}
as required.  The equivariant case is similar.
\end{proof}

\begin{Corollary}\label{t-co-trans-sigma-thom}
We have
\[
T_{\ta}^{*}\Thom_{\T} (\tSUd) = \Thom_{\T} (\tSUd)\in
HP^{*}_{\T} ( ( BSU (d)^{\A})^{\tSUd^{\A}};\O_{\C}).
\]
That is, the Thom class is invariant under translation by $\ta$.  
\end{Corollary}

\begin{proof}
Proposition \ref{t-pr-delta-pp} gives the result for the corresponding
Euler class.  In this universal case, the cohomology of the base $BSU
(d)^{\A}$ is a domain, and cohomology of the Thom space is a principal
ideal, so the result for the Euler class gives the result for the Thom
class.
\end{proof}

It is important that $\delta'_{\A}$ has no zero or pole at $0$.

\begin{Proposition}\label{t-pr-delta-a-p-hol}
The class $\deltaBAp (\tSU,\ta)$ gives an element of $H_{\T}^{*}
(BSU^{\A};\O_{\C,0}^{\wedge})^{\times}$.  Moreover, if $\tStringc$ is  
the
tautological bundle over $\bstringc^{\A}$, then
\[
      \deltaBAp (\tStringc,\ta) = T_{a}^{*}e_{\T} (\tStringc/\tStringc^ 
{\A}) \in HP_{\T}^{*} (\bstringc^{\A};\O_{\C,0}^{\wedge})^{\times}
\]
\end{Proposition}

\begin{proof}

First let's consider the situation  over $BSU (d).$
By construction, $\deltaBAp (\tSUd,\ta)$ can be viewed as an element of
$R=HP^{0}_{\T} (BSU (d)^{A};\O_{\C,0}^{\wedge}).$
Proposition \ref{prop:rigideven} implies that there is an isomorphism
$R\iso HP^{0} (BSU (d)^{A})\psb{z}$, and so $f\in R$ is a unit if and
only if with respect to such an isomorphism the coeffiecient of
$z^{0}$ is a unit.  Moreover the reduced cohomology of $BSU (d)^{A}$
is topologically nilpotent, and to it suffices to check this condition
after restriction to $\pi_{0}BSU (d)^{A}.$

Since
\[
        BSU (d)^{\T} \to BSU (d)^{\A}
\]
is surjective on $\pi_{0}$, it follows that
$f$ is a unit in $R$ if and only if the
coefficient of $z^{0}$ in the restriction of $f$
to
\[
HP^{0}_{\T} (BSU (d)^{\T};\O_{\C,0}^{\wedge}) \iso  HP^{0} (BSU (d)^ 
{\T})\psb{z}
\]
is a unit.  For a component of $BZ (\tm)$ of $BSU (d)^{\T}$ labelled  
by a
homomorphism $\tm: \T\to T$, that restriction
is
\begin{align*}
    \deltaBAp (\tSUd,\tm,\ta) (z,\tau) & =
\exp \bigl(\tfrac{k}{n}I' (\tm,x+\tm z) +
\tfrac{k}{n}\ta \qf' (\tm)\bigr)
\prod_{\tm_{j}\not\equiv 0 \mod n} \sigma (x_{j} + \tm_{j}z + \tm_{j} 
\ta,\tau).
\end{align*}
Recall that $\ta$ is a lift of a point $a$ of order $n.$  If $\tm_{j}$  
is
not divisible by $n$, then there is a small
neighborhood $U$ of $0$ such that $\tm_{j} (z+\ta)\not\in\Lambda$ for
$z\in U$: so $\sigma (\tm_{j} z + \tm_{j} \ta)$ is a unit of
$\O_{\C,0}^{\wedge}.$  Now consider the Taylor series expansion
\[
\sigma (x+\tm z + \ta,\tau) = \sigma (\tm z + \tm \ta,\tau) + o (x).
\]
Since $x$ is a power series variable, this is a unit since
$\sigma (\tm +\tm\ta,\tau)$ is a unit.

We conclude that
\[
\deltaBAp
(\tSUd,\ta) \in HP_{\T}^{*} (BSU(d)^{A};\O_{\C,0}^{\wedge})^{\times},
\]
and passing to virtual bundles we see that
\[
\deltaBAp
(\tSU,\ta) \in HP_{\T}^{*} (BSU^{A};\O_{\C,0}^{\wedge})^{\times},
\]
as required.   This proves the first part of the Proposition.
The proof of the second part proceeds exactly as for the case of $ 
\deltaBA$ in
Proposition \ref{t-pr-delta-fixed-busix}: using the fact that
\[
     H^{*}_{\T} (BSU^{A}) \to H^{*}_{\T} (BSU^{\T})
\]
is injective, one reduces to checking the equation after restriction
to $BSU^{\T}$, where it is an easy calculation.
\end{proof}

\subsection{Proof of Proposition \ref{t-pr-phi-A-indep-lift}}
\label{sec:proof-proposition}

We continue to write $\tStringc$ for the tautological bundle over $ 
\bstringc.$

To illustrate the argument we give it first for the Euler class
associated to $\stror_{\A}$.  Let
\[
\zeta: \bstringc^{\A}\to (\bstringc^{\A})^{\tStringc^{\A}}
\]
be the zero
section.  Let $\trans_{\ta}$ denote translation by $\ta$ in $\C$.
Then, as we showed in Proposition \ref{t-pr-delta-fixed-busix},
\[
       \trans_{\ta}^{*}\zeta^{*}\stror_{\A} (z) = \deltaBA (\tStringc, 
\ta) (z),
\]
so to understand the behaviour of $\stror_{\A}$ near $\ta$, it  
suffices to
understand the behaviour of $\deltaBA$ near $0$.  But we have
shown in Proposition \ref{t-pr-delta-inv-busix}  that the class $ 
\deltaBA$ does not depend on
the choice of lift $\ta$ of $a$.  So
\[
       \trans_{\ta}^{*}\zeta^{*}\stror_{\A} (z) = \deltaBA (\tStringc, 
\ta) (z) =
\deltaBA (\tStringc,\tap) (z) = \trans_{\tap}^{*}\zeta^{*}\stror_{\A}  
(z).
\]

The refinement to $\stror_{\A}$ itself is clear, given the preceding
discussion and the fact that, by definition,
\[
       \stror_{\A} = \Thom_{\T} (\tStringc^{\A})e_{\T} (\tStringc/ 
\tStringc^{\A}).
\]
Corollary \ref{t-co-trans-sigma-thom} shows that
\[
       \trans_{\ta}^{*}\Thom_{\T} (\tStringc^{\A}) = \Thom_{\T}  
(\tStringc^{\A})
\]
and so this quantity is independent of $\ta$ and for that matter of $a$.
Meanwhile by Proposition \ref{t-pr-delta-a-p-hol},
\[
       \trans_{\ta}^{*}e_{\T} (\tStringc/\tStringc^{\A}) =
\deltaBAp(\tStringc,\ta),
\]
and Proposition \ref{t-pr-delta-prime-abar} shows that this quantity
is independent of the lift $\ta$.  Thus we have shown that
$\trans_{\ta}^{*}\stror_{\A}$
depends only on $a$, and not the choice of lift $\ta$.

Finally we must show that multiplication by $\stror_{\A}$ is an
isomorphism
\[
       HP^{*}_{\T} (\bstringc^{\A};\cO_{\A}^{\wedge})
       \xrightarrow[\iso]{\stror_{\A}}
       HP^{*}_{\T} ((\bstringc^{\A})^{\tStringc^{\A}};\cO_{\A}^ 
{\wedge}).
\]
Certainly $\Thom_{\T} (\tStringc^{\A})$ is an isomorphism
\[
     HP^{*}_{\T} (\bstringc^{\A}) \rightarrow
     HP^{*}_{\T} ((\bstringc^{\A})^{\tStringc^{\A}}).
\]
In Proposition \ref{t-pr-delta-a-p-hol} we showed that
\[
\trans_{\ta}^{*}e_{\T} (\tSU/\tSU^{\A}) =  \deltaBAp (\tSU,\ta)
\]
is a unit of $HP^{*}_{\T} (BSU^{\A};\O_{C,0}^{\wedge})$.    As $a$
ranges over the points of $\Jlr{A}$, we find that
\[
e_{\T} (\tStringc/\tStringc^{\A}) \in HP^{*}_{\T} ((\bstringc^{\A}); 
\OA)^{\times},
\]
as required.

This completes the proof of Proposition \ref{t-pr-phi-A-indep-lift}.   
\qed

\part{Analytic and algebraic geometry of the sigma orientation}
\label{sec:ellipt-cohom-bsu}

In this part, we give an account of the string orientation in terms of  
the
analytic geometry of the curve $C=\C/\Lambda$.   In
Section \ref{sec:ellipt-cohom-sheav}, we
associate to a $\T$-spectrum $X$ a sort of sheaf  $\ShfE (X)$ on $\J$
(actually it is a sheaf over a diagram approximating $\J$)
whose sections are calculated by an exact sequence like
\eqref{eq:56}.  If $X$ is a space, this is a sheaf of rings, and so
gives rise to a ringed space $\SpcE (X).$  If $V$ is a complex vector bundle
over $X$, the construction of $\SpcE (X)$ is such that the Thom  
isomorphisms
for ordinary Borel cohomology show that $\ShfE (X^V)$ is a line
bundle over $\SpcE (X)$, and if this line bundle is globally
trivial we have a Thom isomorphism in elliptic cohomology.

We are interested in the special case when $X=BSU(d)$ and
$V=\tSUd$ is  the universal bundle over it.
As before, we begin (in Section \ref{sec:analyt-geom-sigma})
by dealing with $\T$-fixed space and generic points on the curve,
which is to say we consider the
subspace $Y=BSU(d)^{\T}$ and its associated ringed space $\mathfrak{Y}$.
In Section \ref{sec:ptsoffiniteorder}, we turn to the analysis
of points with finite isotropy and torsion points on the curve,
and hence build up to  the full space $X=BSU(d)$ and its associated
ringed space $\mathfrak{X}$. Moreover we construct a map from
$\mathfrak{X}$ to the space $(\cochars\otimes_{\Z}\J)/W$, where
$\cochars$ and $W$ are 
the cocharacters and Weyl group of $SU(d)$.  Overall, we construct a
commutative diagram of ringed spaces 
\[
\begin{CD}
\SpcE (BSU (d)^{\T}) @>>> \SpcE (BSU (d)) \\
@V\iso VV @VV \iso V \\
\Y @>>> \mathfrak{X} @>>> (\cochars\otimes_{\Z}\J)/W.
\end{CD}
\]


Next we see that  our elliptic cohomology Thom class gives a section of
the twist of $\ShfE (BSU (d)^{\tSUd})$ by the Looijenga line bundle pulled back
from   $(\cochars\otimes_{\Z}\J)/W$.  Turning to
$\stringc$-bundles, let $V$ be a
complex representation of $\T$ of rank $d$ and determinant $1$, and
as in \S\ref{subsec:univbundles}, let $\bstringc (V)$ be the pull-back
in the diagram
\[
\begin{CD}
\bstringc (V) @>>> \bstringc \\
@VVV @VVV \\
BSU (d) @> \tSUd  - V >> BSU;
\end{CD}
\]
we explained in \S\ref{subsec:univbundles} that $\bstringc$ is a colimit
of such spaces.  We show that the Looijenga bundle becomes trivial
over $\SpcE (\bstringc (V))$,  and it follows that our Thom class  
gives a
trivialization of the line bundle
$\ShfE (\bstringc (V)^{\tStringc})$ over $\SpcE (\bstringc (V)).$
Thus we prove the conjecture in \cite{Ando:AESO,Ando:Australia} in this
setting.

Finally, we conclude in Section \ref{sec:alg-geom-sigma}
by rephrasing the situation in algebraic terms;
we hope that this will eventually lead to an algebraic construction of
the string orientation for equivariant elliptic cohomology theories
associated to arithmetic elliptic curves.

\section{Elliptic cohomology and sheaves of $\O_{\J}$-modules}
\label{sec:ellipt-cohom-sheav}

The sequence \eqref{eq:56} suggests that $\EJ^{*}_{\T} (X)$ is
approximately the cohomology of a sheaf
on $\J$: the $H^n(\XT ; \cK
\tensor \Omega_\J^*)$  factor concerns the
behaviour of a section generically on $\J$, while the $HP_{\T}^n(\XA;
\cO_A^{\wedge})$ factors concern the behaviour in small
neighborhoods of the points of finite order.

  In Sections 19--22 of \cite{ellT}, the second author constructs such a
sheaf, which we describe in \S \ref{subsec:Groj}, however this does not
have the formal properties we need, so in \S \ref{subsec:Espaces} we
construct a suitable completion with better behaviour.

\subsection{The Grojnowski sheaf}
\label{subsec:Groj}

We briefly recall some of the properties of the sheaf $\Gre (X)$ 
(which is  denoted $\mathcal{M}_{\J} F (X,\EJ)$ in \cite{ellT}).
As usual, topological 2-periodicity corresponds to the periodicity given
by multiplication by the invariant differential as in Subsection
\ref{sec:geom-ellipt-curv}, which we refer to as $\omega^*$-periodicity.

\begin{Proposition}
There is a functor $\Gre$ from $\T$-spectra to $\omega^*$-periodic
$\O_{\J}$-modules enjoying the following properties.
\begin{enumerate}
\item If $W$ is a virtual complex representation of $\T$, then
\[
       \Gre (S^{W}) \iso \O_{\J} (-D (W))\otimes \omega^{*}.
\]
\item There is a short exact sequence
\[
      0 \rightarrow \Sigma H^{1} (\J;\Gre (X))
        \rightarrow \EJc^{*} (X) \rightarrow H^{0} (\J;\Gre (X))  
\rightarrow 0.
\]
\item
\[
    \Gre (X^{\T}) \iso H^{*} (X^{\T};\O_{\J} \otimes \omega^{*})
\]
\item
Let $a$ be a point of $\J$ of exact order $n$,
and let $A\subseteq \T$ be the subgroup of
order $n$ ($A=\T$ if $n=\infty$).  For a finite $\T$-space $X$
\[
       \Gre (X)_{a} \iso \EJc^{*} (X^{A})\otimes \O_{\J,a}.
\]
In the case that
\[
H_{*}^{\T} (X^{\A})  \iso H_{*} (X^{\A})\otimes H_{*} (B\T),
\]
then
\[
       \Gre (X)_{a} \iso H^{*} (X^{A}; \cO_{\J,a}\tensor \omega^*).
\]
\end{enumerate}
\end{Proposition}


In the case of finite $X$ and an elliptic curve of the form
$\J=\C/\Lambda$, the sheaf $\Gre (X)$ is equivalent to that of
\cite{Grojnowski:Ell-new}; see \cite[\S22]{ellT}.  As we have already
observed, one of the important innovations of \cite{ellT} is to
realize that this sheaf can be constructed in the case of a rational
elliptic curve, by using the function $t_{|A|}$ rather than the
covering $\C\to \C/\Lambda$ to make $\OA$ into an $H^{*}B\T$-algebra.

We shall be interested
in taking $X=BSU (d)$ or $\bstringc$, in which case it is more
difficult to describe $\Gre (X)$ explicitly.  Instead, we
introduce a variant of $\Gre$, which amounts to working
with the sheaf $\Gre (X^{\T})$ together with the collection of stalks
$\Gre (X)_{a}$ for $a$ of finite order.

The long exact sequence \eqref{eq:56} shows
that the difference between $\EJc^{*} (X)$ and $\EJc^{*} (X^{\T})$ is
local on the elliptic curve: the ``meromorphic'' sections of
$\EJc^{*} (X)$ or $\EJc^{*} (X^{\T})$ are just the meromorphic  
functions on
$\spec (H^{*} (X^{\T})) \times C,$ and the question of whether a
meromorphic function $s$ is holomorphic can be checked one point at a
time.  For $X^{\T}$, this is a question of checking whether for each
$a$ of finite order, $s_{a}$ lives in
\[
HP^{*} (X^{\T};\O^{\wedge}_{a})  \subset HP^{*} (X^{\T};\cKF)
\]
but for
$X$ on which $\T$ acts non-trivially one must specify for
each finite $\A\subset \T$ an element $s_{\A}$ of
\[
  HP^{*}_{\T} (\XA;\OA)
\]
which restricts to $s$ in
\[
    HP^{*} (X^{\T};\O^{\wedge}_{\A}).
\]
We introduce a category of objects called $\CatE$-sheaves over $C$ to  
make this
systematic.  The terminology is intended to suggest a category
$\CatE$ like the category of elliptic curves, but with extraneous
structure removed. Thus an $\CatE$-sheaf over $C$ is like a sheaf
over an object $\CatE (C)$ consisting of the relevant part of $C$.

\begin{Definition}\label{def-cat-E}
The category of \emph{$\CatE$-sheaves} over $C$ is the category in which
an object $\ShfE$ consists of
\begin{enumerate}
\item an $\omega^{*}$-periodic graded $\O_{\J}$-module
$\ShfE_{\T}$;
\item for each finite $\A\subset \T$, an $\omega^{*}$-periodic graded  
module
$\ShfE_{\A}$ over the local ring $\O_A$,  and a map of graded $\O^ 
{\wedge}_{\A}$-modules
\begin{equation} \label{eq:60}
      r_{\A}: \ShfE_{\A}^{\wedge} \to (\ShfE_{\T})^{\wedge}_{\Jlr{A}}.
\end{equation}
\end{enumerate}
A morphism $\ShfE \to \ShfE'$ consists of maps
\[
      \ShfE_{\A} \to \ShfE'_{\A}
\]
for $\A\subseteq \T$, which intertwine the maps \eqref{eq:60}.
\end{Definition}

\begin{Definition}
If $\ShfE$ is an $\CatE$-sheaf over $C$, then a section $s$ of $\ShfE$  
consists
of sections $s_{\A}$ of $\ShfE_{\A}$ for $\A\subseteq \T$, such that for
each finite $\A$,
\[
     r_{\A}s_{\A} = s_{\T}.
\]
We write $\SectE (\ShfE)^{*}$ for the graded group of sections of $ 
\ShfE$.
\end{Definition}

We formalize  the motivating example.

\begin{Example}
If $X$ is a $\T$-spectrum, then $\Gre (X)$ defines an $\cE$-sheaf $\cN  
(X)$
over $C$ by taking
\[
     \cN (X)_{\T} := \Gre (X^{\T})
\]
and
\[
    \cN  (X)_{\A} := \Gre (X)_{\Jlr{\A}}
\]
with structure map
\[
      \Gre (X)^{\wedge}_{\Jlr{\A}} \to \Gre (X^{\T})^{\wedge}_{\Jlr 
{\A}}.
\]
The construction applies equally well to show that any $\omega^*$- 
periodic
sheaf of $\cO_C$-modules over $C$ gives an $\CatE$-sheaf over $C$ by  
restriction.
\end{Example}

\subsection{The completed Grojnowski $\protect \cE$-sheaf of a
$\protect \T$-space}
\label{subsec:Espaces}

Our main example, motivated by Proposition \ref{prop:Hasse},
  is designed to highlight the regularity of various sections we
later construct. It turns out to be a type of completion of the  
Grojnowski
sheaf with convenient formal properties.

\begin{prop}
There is a functor
$$\ShfE : \mbox{$\T$-spaces} \lra \mbox{$\cE$-sheaves}.$$
It is defined by the formulae
\begin{align*}
        \ShfE (X)_{\T} & = HP^{*} (X^{\T};\O_{\J}) \\
        \ShfE (X)_{\A} & = HP^{*}_{\T} (X^{\A}; \OA) \text{ for }\A
        \text{ finite},
\end{align*}
with structure map
\[
     \ShfE (X)_{\A} = HP^{*}_{\T} (X^{\A}; \OA)
\rightarrow HP^{*}_{\T} (X^{\T}; \OA) \iso
             HP^{*} (X^{\T};\O^{\wedge}_{\A}) = (\ShfE (X)_{\T})^ 
{\wedge}_{\Jlr{\A}}.
\]
The functor is a naive 2-periodic $\T$-equivariant cohomology theory
in the sense that it is homotopy invariant, exact and satisfies excision
and the wedge axiom.

There is a map
$$\cN (X)\lra \cF (X)$$
natural in $X$, and it is an isomorphism at $\T$ in that
$\ShfE (X)_{\T} = \Gre (X^{\T})=\cN  (X)_{\T}$.
If $X$ is finite,
then it is completion at $A$ in that
\[
        \ShfE (X)_{\A} \iso HP^{*} (X^{\A}; \O^{\wedge}_{\A}) \iso
        \Gre (X)^{\wedge}_{\Jlr{\A}}.
\]
\end{prop}

\begin{proof}
For behaviour at $\T$, we need only note that $\Gre (X)$ is associated
to the free $EC$-module $F(X^{\T},EC)$ of the same rank as $HP^*(X^ 
{\T};\O_C)$.

For behaviour at $A$ we need to recall that the stalk of the
sheaf $\Gre (X)$ at $A$ consists of the homotopy groups of
$F(X,EC) \sm \etnA$, where $\etnA$ is defined by a cofibre sequence
$$\elrnA \lra S^0 \lra \etnA$$
  and
$$\elrnA = \bigvee_{B \neq A}\elrB, $$
where the wedge is over all finite subgroups $B$ except $A$.
Now consider the  maps
$$F(X,EC) \sm \etnA \lra F(X^A,EC) \sm \etnA \lra F(X^A \sm \elrA, EC)  
\sm \etnA.$$
The homotopy of the left-hand side is
$\Gre (X)_A$, and the homotopy of the right-hand side is
$$ HP^{*}_{\T} (X^{\A}; \OA)= \ShfE (X)_{\A} $$
by the completion theorem. The first of the maps is an equivalence
when $X$ is finite since $X/X^A$ is built from basic cells corresponding
to finite subgroups other than $A$. The second of the maps is   
completion.
\end{proof}

If $X$ is a space, then it will also be convenient to reverse the
arrows and consider the ringed spaces over $\J$ given
by
\begin{equation} \label{eq:45}
\begin{split}
      \SpcE (X)_{\T} & \eqdef (\J,\ShfE (X)_{\T}) \\
      \SpcE (X)_{\A} & \eqdef (\Jlr{A},\ShfE (X)_{\A}).
\end{split}
\end{equation}
In this guise the structure map of $\SpcE (X)$ is a map of ringed spaces
over $\J^{\wedge}_{\Jlr{A}}$
\begin{equation}\label{eq:62}
       (\SpcE (X)_{\T})^{\wedge}_{\Jlr{A}} \xra{}
       \SpcE (X)_{\A}^{\wedge}.
\end{equation}
\begin{Definition}
We shall refer to a collection of spaces $\SpcE_{\A}$ for $\A\subseteq
\T$, equipped with maps \eqref{eq:62}, as an \emph{$\CatE$-space} over  
$C$.
If $\SpcE$ is a $\CatE$-space over $C$, then we write $\O_{\SpcE}$ for  
its
associated $\CatE$-sheaf over $C$.
\end{Definition}

\begin{Remark}
As for $\CatE$-sheaves, any ringed space over $C$ gives rise to a $ 
\CatE$-space
over $C$ by restriction.
\end{Remark}

\begin{Proposition} \label{t-pr-cate-spaces-and-ell}
If $X$ is a $\T$-space with $H^*(\XT)$
concentrated in even degrees
then there is a natural isomorphism
\[
     \SectE (\O_{\SpcE (X)})^0 \stackrel{\cong}\rightarrow  \EJc^{0}  
(X).
\]
\end{Proposition}

\begin{Remark}
If we assume in addition that $H^*(X^A)$ is in even degrees for all
finite subgroups $A$, then    $\SectE (\O_{\SpcE (X)})$ will be entirely
in even degrees, and hence it will be isomorphic to the even part of
$\EJc^*(X)$.
\end{Remark}

\begin{proof}
Consider the diagram
\[
\xymatrix{
{\SectE (\ShfE (X))^0}
  \ar@{>->}[r]
  \ar[d]
&
{HP^{0} (\XT;\O_{\J}) \times   \prod_\A HP^{0}_{\T}(\XA;\O^{\wedge}_ 
{\A})}
  \ar[r]^-{F}
  \ar[d]
&
{HP^{0} (\XT;\cKF)}
  \ar[d]
\\
{\EJ^{0}_{\T}(X)}
  \ar@{>->}[r]
&
{HP^{0}(\XT ; \cK) \times
        \prod_\A HP^{0}_{\T}(\XA;\O^{\wedge}_{\A})}
  \ar[r]^-{G}
&
HP^{0}(\XT ; \cKF)
  \ar@{->}[r]
&
{\EJ^{1}_{\T}(X).}
}
\]
The exactness of the top  row describes $\SectE (\ShfE (X))^0.$  With
our hypotheses, the bottom row is exact, by Proposition
\ref{prop:Hasse}.    The only difference between
the middle terms is the $\O_{\J}$, mapping to $\cK$ in the bottom.
The middle vertical arrow is the obvious map, and it is injective.
It is clear that the left vertical arrow exists, and is
injective.

But in fact, with our hypotheses the two maps labelled $F$ and $G$
above have the same kernel.  Suppose that $s_{\T} \in H^{0}
(X^{\T};\cK)$ and $s_{A} \in H^{0}_{\T} (\XA,\O^{\wedge}_{\A})$ for
$A\subset \T$ comprise an element $s \in \Ker G.$
We have
\[
s_{\T} \in HP^{0} (\XT;\cK) \iso \hom (HP_{0}\XT,\cK),
\]
and to ask whether
\[
s_{\T}\in HP^{0} (\XT,\O_{\J})\iso \hom (HP_{0}\XT,\O_{\J})
\]
is to ask whether, for all $x\in HP_{0}\XT$, $s_{\T} (x)$ has no pole
at any point of finite order.  To check this, it suffices to check at
each point $a$ of finite order that the Laurent series expansion of
$s_{\T} (x)$ at $a$ has no pole.  The collection of these Laurent
series as $x$ varies is an element of
\[
    \hom (HP_{*}\XT,\cKF) \iso HP^{0} (\XT,\cKF)
\]
as indicated above.  To say that $s$ is in the kernel of $G$ is to say
that the image of
\[
    s_{A} \in HP^{0}_{\T} (\XA;\O^{\wedge}_{\A})
\]
under
\[
    HP^{0}_{\T} (\XA;\O^{\wedge}_{\A}) \to HP^{0}_{\T}
    (\XT;\O^{\wedge}_{\A}) \iso HP^{0} (\XT;\O^{\wedge}_{\A}) \iso \hom
    (HP_{0} (\XT),\O^{\wedge}_{\A}) \to  \hom (HP_{0} (\XT),\cKF)
\]
gives the Laurent series expansions of $s_{\T}$ at the points of
$\Jlr{A}$.  Clearly such an $s_{\T}$ is in $HP^{0} (\XT,\O_{\J})$, and
so $s$ can be considered as an element of $\Ker F$, as required.
\end{proof}

\subsection{$\cE$-sheaves over $\T$-fixed spaces}
\label{sec:ce-sheaves-T-fixed}

We recapitulate the discussion of $\EJ^{*}_{\T} (X)$ from the
introduction to this section, in terms of the $\cE$-sheaf $\ShfE (X).$

If $\T$ acts trivially on $Y$, then
\[
     \ShfE (Y)_{\A} \iso (\ShfE (Y)_{\T})^{\wedge}_{\Jlr{\A}},
\]
and so $\ShfE (Y)$ contains no additional information beyond $\ShfE
(Y)_{\T}$.  In terms of spaces, $\SpcE (Y)$ contains no more
information than $\SpcE (Y)_{\T}$.

If $X$ is any $\T$-space, then
\[
     \ShfE (X)_{\T} \iso \ShfE (X^{\T})_{\T},
\]
or equivalently
\[
     \SpcE (X^{\T})_{\T} \iso \SpcE (X)_{\T}.
\]
However, if $\A\subset \T$ is a subgroup of finite order, then
we have only the restriction map
\begin{equation}\label{eq:32}
     \ShfE (X)_{\A} = HP^{0}_{\T}(X^{\A};\O_{\A}^{\wedge})  \to
     HP^{0}_{\T} (X^{\T};\O_{\A}^{\wedge}) \iso  HP^{0}
     (X^{\T};\O_{\A}^{\wedge})=\ShfE (X^{\T})_{\A}.
\end{equation}
Thus the map of $\T$-spaces
\[
    \SpcE (X^{\T}) \to \SpcE (X)
\]
induced by the inclusion of fixed points
is an isomorphism over a generic point of $\J$, but at the point of
finite order, $\SpcE (X)$ is the thickening of $\SpcE (X^{\T})$
given by \eqref{eq:32}.

\section{Analytic geometry of the sigma orientation I:
$\protect \T$-fixed spaces}
\label{sec:analyt-geom-sigma}

In this section we consider $Y=BSU(d)^{\T}$ and construct an $\cE$-space
$\mathfrak{Y}$ and an isomorphism of $\cE$-spaces
\[
       \SpcE (BSU (d)^{\T})\iso \mathfrak{Y}.
\]
As for the Thom space of the universal bundle $\tSUd$,
we construct an ideal sheaf $\cI$ over $\mathfrak{Y}$ which
corresponds to $\ShfE ((BSU (d)^{\T})^{\tSUd})$ under the
displayed isomorphism, and we
explain the sigma orientation in terms of this model.


In Section \ref{sec:ptsoffiniteorder}, we will extend our analysis to
points of $X=BSU (d)$ with finite isotropy and torsion points on the
elliptic curve. This enables us to  construct an $\cE$-space
$\mathfrak{X}$ and an isomorphism of $\cE$-spaces
\[
       \SpcE (BSU (d))\iso \mathfrak{X}.
\]
The ideal sheaf $\cI$ over $\mathfrak{X}$ then corresponds to
$\ShfE (BSU (d)^{\tSUd})$ under this isomorphism, and
the analysis of the sigma orientation extends to the whole space
$BSU(d)^{\T}$.


\subsection{The  $\protect\cE$-space associated to the $\protect
\T$-fixed points of  $BSU(d)$} \label{sec:prot-space-assoc}

For brevity we write $Y=BSU(d)^{\T}$.
Since $\T$ acts trivially on $Y$, we have
\[
     H^{*}_{\T} (Y^{\A}) = H^{*}_{\T} (Y) =  H^{*} (Y\times B\T),
\]
and so $\SectE (\ShfE (Y))$ is given by the exact sequence
\begin{equation} \label{eq:41}
  0 \rightarrow \SectE (\ShfE (Y)) \rightarrow
HP^{0}(Y; \O_{\J}) \times \prod_{\A} HP^{0}(Y;\OA)
\rightarrow
        HP^{0}(Y ; \cKF),
\end{equation}
while $\SectE (\ShfE (Y^{\tSUd}))$ is given by the exact sequence
\begin{equation}\label{eq:31}
  0 \rightarrow \SectE (\ShfE (Y^{\tSUd})) \rightarrow
HP^{0}(Y^{\tSUd^{\T}}; \O_{\J}) \times \prod_{\A} HP^{0}_{\T}
(Y^{\tSUd^{\A}};\OA)
\rightarrow
        HP^{0}(Y^{\tSUd^{\T}} ; \cKF).
\end{equation}
For each component $Z$ of $Y$, $HP^{*} (Z)$ is a domain, and so
each factor in the right of \eqref{eq:31} is a principal ideal of the
corresponding factor in \eqref{eq:41}, generated by
an Euler class.    Our goal is to understand these ideals.

Recall from Proposition
\ref{t-pr-splitting-princ-messy} that elements $\Xi \in
HP^{0} (Y) = HP^{0} (BSU (d)^{\T})$ are given by compatible elements
\[
    \Xi (m) \in HP^{0} (BT)^{W (m)}
\]
where $m$ ranges over $\cochars=\Hom (\T,T)$.  Now
\[
   \spf HP^{0}(B\T) \iso \Gah,
\]
and the projection
\[
      \C \to \J
\]
gives an isomorphism
\[
     \Gah \iso \Jhat,
\]
so we have
\[
      \spf HP^{0} (BT) \iso \cochars \otimes \Jhat.
\]
Thus we may view an
element of $HP^{0} (Y;\O_{\J})$ as a family of
functions
\begin{equation}\label{eq:21}
      (\cochars \otimes \Jhat)/W (m) \times \J\rightarrow \C.
\end{equation}
This suggests that we make the following definitions.

\begin{Definition}\label{def-Y-m}
For $m: \T\to T$, let $\X{m}$ be the space
\[
      \X{m} \eqdef (\cochars\otimes
      \Jhat)/W (m) \times \J.
\]
Note that this is a ringed space over $\J$, so $\O_{\X{m}}$ is an
$\O_{\J}$-algebra.
For $w\in W$, there is an evident isomorphism
\[
     w: \X{m} \to \X{wm},
\]
which is the identity if $w\in W (m)$ so that $w=wm.$  Thus let
\[
     \X{\T}  = \left(\coprod_{m:\T\to T} \X{m} \right)/W.
\]
Let $\Y$ be the $\cE$-space with
\[
     \Y_{\T} = \X{\T},
\]
and
\[
     \Y_{\A} = (\X{\T})^{\wedge}_{\Jlr{\A}}.
\]
\end{Definition}

\begin{Proposition} \label{t-pr-E-BZm-O-scr-X}
For each $m:\T\to T$, labelling a component $BZ (m)$ of $BSU (d)^{\T}$,
there is a canonical isomorphism of $\O_{\J}$-algebras
\[
     \ShfE (BZ (m))_{\T} \iso \O_{\X{m}},
\]
or equivalently of ringed spaces over $\J$
\[
      \SpcE (BZ (m))_{\T}\iso \X{m}.
\]
These assemble to an isomorphism
\[
       \ShfE (BSU (d)^{\T})_{\T}\iso \left(\prod_{m: \T\to T} \O_{\X 
{m}}\right)^{W},
\]
or equivalently an isomorphism of $\cE$-spaces
\[
       \SpcE (BSU (d)^{\T})  \iso \Y.
\] \qed
\end{Proposition}

Let $\tSUd$ be the tautological bundle over $BSU (d).$  We turn now to
the Thom space $Y^{\tSUd}.$
Suppose that $m= (m_{1},\dotsc ,m_{d}) \in \cochars\subset \Z^{d}$  
labels a
component $BZ (m)$ of $Y=BSU (d)^{\T}$, and $x_{i}\in H^{2}(BT)$ are the
corresponding generators.
The equivariant Euler class of $\tSUd\restr{BZ (m)}$ in the
orientation of ordinary cohomology given by the sigma function is
\begin{equation} \label{eq:39}
     f_{m} (x,z) \eqdef  e_{\T} (\tSUd)\restr{BZ (m)}  = \prod_{i}  
\sigma
     (x_{i} + m_{i} z,\tau).
\end{equation}
This defines a holomorphic function
\[
     f_{m}: (\cochars\otimes \Jhat)/W (m) \times \C\rightarrow \C,
\]
but it does \emph{not} descend to $\X{m} = (\cochars\otimes \Jhat)/W
(m)\times \J $ as in \eqref{eq:21}.   Instead, it is a holomorphic  
section
of a line bundle $\Loo_{m}$ over $\X{m}$, as we now explain.

\subsection{The Loojienga line bundle}

There is a line bundle over
\[
\cochars\otimes \J = \cochars \otimes (\C^{\times}/q^{\Z})
\]
given by the formula
\begin{equation} \label{eq:24}
\frac{(\cochars \otimes \C^{\times})\times \C}
                {(u,\lambda) \sim
                 (uq^{m}, \lambda u^{-I (m)} q^{-\qf (m)})},
\end{equation}
and the Weyl-invariance of $I$ and $\qf$ \eqref{eq:25} imply that
this line bundle descends to a line bundle on $(\cochars\otimes
\J)/W,$ which we call $\Loo$.  As far as we know it was introduced by
Looijenga
\cite{Looijenga:RootSystems}; its role in elliptic cohomology was
first indicated by Grojnowski \cite{Grojnowski:Ell-new} and has been further
studied in \cite{Ando:EllLG,Ando:AESO}.  We
can regard $\sigma (u,q)$ as giving 
a $W$-invariant function on $(\cochars\otimes \C^{\times}),$ and the
functional equation
\[
     \sigma (uq^{n},q) = u^{-I (m)}q^{-\qf (m)} \sigma (u,q)
\]
of Lemma  \ref{t-fu-eq-sigma} implies that the product of sigma  
functions
\[
      \sigma (u,q) = \prod_i \sigma (u_{i},q)
\]
descends to a holomorphic section of $\Loo$.

Addition in the abelian group $\cochars\otimes \J$
induces a map
\begin{equation}\label{eq:34}
\mu_{m}: \X{m} = (\cochars\otimes \Jhat)/W (m)\times \J \xra{i\times m}
(\cochars \otimes \J)/W (m) \times (\cochars\otimes \J)^{W (m)} \to
     (\cochars\otimes \J)/ W (m) \xra{ }   (\cochars\otimes \J)/ W;
\end{equation}
if $(a_{1},\dotsc ,a_{d})\in \cochars \otimes \Jhat$ represents a
point of $(\cochars\otimes \Jhat)/W (m)$ and $z\in \J$, then
\[
   \mu_{m} (a_{1},\dotsc ,a_{d},z) = (a_{1} + m_{1}z,\dotsc ,a_{d} + m_ 
{d}z).
\]
This has the following relationship to topology.  The Borel
construction of $\tSUd$ is classified by a map
\[
     BZ (m)\times B\T \to BSU (d),
\]
and so provides a map
\[
   \mu_{m}^{\mathrm{top}} :  (\cochars\otimes\Jhat)/W
     (m)\times \Jhat \iso \spf HP^{0}(BZ (m)\times B\T) \rightarrow \spf
     HP^{0}(BSU) \iso (\cochars\otimes \Jhat)/W.
\]
\begin{Lemma}\label{t-le-mu-topology}
The diagram
\[
\begin{CD}
\spf HP^{0} (BZ (m)\times B\T) @> \mu_{m}^{\mathrm{top}}>> \spf HP^{0} 
(BSU (d)) \\
@VVV @VVV \\
(\cochars \otimes \Jhat)/W (m)\times \J @> \mu_{m} >>
(\cochars\otimes \J)/W
\end{CD}
\]
commutes.
\end{Lemma}

\begin{proof}
This is an expression of the fact
(see \eqref{eq:27}) that
\[
\ChBdot (\tSUd)=\prod_i(1+x_{i} + m_{i}z).
\]
\end{proof}

\begin{Definition} \label{def-loo-m}
Let $\Loo_{m}$ be the line bundle
\[
       \Loo_{m} \eqdef \mu_{m}^{*} \Loo
\]
over $\X{m}.$
Explicitly, $\Loo_{m}$ is obtained from the line bundle
\begin{equation} \label{eq:47}
     \frac{(\cochars \otimes \C^{\times}) \times \C^{\times} \times \C}
                {(u,z,\lambda) \sim
                 (u,zq^{k}, \lambda u^{-kI (m)}z^{-kI (m,m)} q^{-k^ 
{2}\qf (m)})}
\end{equation}
over $(\cochars\otimes \C^{\times}) \times \J$ by restriction to
$(\cochars\otimes \Jhat) \times \J$ and then descent to $\X{m}.$
\end{Definition}

It is easy to check that the functions $f_{m}$ of \eqref{eq:39}
descend to the holomorphic sections
\[
     f_{m} = \mu_{m}^{*}\sigma
\]
of $\Loo_{m}.$  By construction, these are
compatible as $m$ varies.

\begin{Proposition} \label{t-pr-Loo-m-and-f-m}
For $w\in W$, the diagram
\[
\xymatrix{
  {\X{m}}
   \ar[rr]^{w}
   \ar[dr]_{\mu_{m}}
  & &
  {\X{wm}}
   \ar[dl]^{\mu_{wm}} \\
  & {(\cochars\otimes \J)/W}
  }
\]
commutes, and so the maps
$\mu_{m}$ assemble to a map
\[
    \mu_{\T}: \X{\T} \to (\cochars\otimes \J)/ W.
\]
\qed
\end{Proposition}


By the Proposition, the $\Loo_{m}$ descend to a line bundle
\[
    \Loo_{\T} = \mu^{*}_{\T}\Loo \xra{} \X{\T},
\]
equipped with a section $f_{\T}= \mu^{*}_{\T}\sigma.$
Because the sigma function has zeroes,
$f_{\T}$ is not a trivialization of $\Loo_{\T}$.  Instead,
let $\cI (f_{\T})$ be
the ideal of zeroes of $f_{\T}.$

\begin{Corollary}\label{t-co-f-triv-what}
The section $f_{\T}$ is a trivialization
of the line bundle  $\Loo_{\T}\otimes \cI (f_{\T})$.   \qed
\end{Corollary}

\subsection{Trivializing line bundles of string bundles I: $\protect
\T$-fixed spaces}

We now turn to the sigma orientation.  We retain the abbreviation
  $Y=BSU(d)^{\T}$ from Subsection
\ref{sec:prot-space-assoc}, and we continue to write $\tSUd$ for the
tautological bundle over $Y$.   Since in Definition \ref{def-Y-m}, $\Y$
is just $\X{\T}$, considered as a $\cE$-space without thickening, we
may trivially regard $\Loo_{\T}, f_{\T},$ and $\cI (f_{\T})$ as living
over $\Y$.

\begin{Proposition}\label{t-pr-f-m-zeroes-thom}
After identifying
\[
    \SpcE (Y)  \iso \Y
\]
the map
\[
     \ShfE (Y^{\tSUd}) \xra{} \ShfE (Y)
\]
induced by the zero section $\zeta: Y\to Y^{\tSUd}$ induces an  
isomorphism
of line bundles over $\Y$
\[
\cI (f_{\T}) \iso \ShfE (Y^{\tSUd}).
\]
\end{Proposition}

\begin{proof}
Both $\cI (f_{\T})$ and $\ShfE (Y^{\tSUd})$ are ideal sheaves on $\Y,$
locally generated by a single element of $\ShfE (Y).$  We shall show
that, locally on $\Y$, the generators differ by a unit.


Let $a$ be a point of $\J$, of order $n$ with $1\leq n \leq \infty.$    
Let
\[
    \tSUd^{a} = \tSUd^{\T[n]}.
\]
Then
\[
     (\ShfE (Y)_{\T})^{\wedge}_{a} \iso HP^{0} (Y;\O^{\wedge}_{a}),
\]
while
\[
     (\ShfE (Y^{\tSUd})_{\T})^{\wedge}_{a}\iso HP^{0}_{\T}
     (Y^{\tSUd^{a}};\O^{\wedge}_{a}).
\]
Let
\[
m= (m_{1},\dotsc ,m_{d}) \in \cochars\subset \Z^{d}
\]
be a cocharacter, labelling a component $Z=BZ (m)$ of $Y$.  Then
$HP^{0}_{\T} (Z^{\tSUd^{a}};\O_{a}^{\wedge})$ is the ideal in
$HP^{0} (Z;\O_{a}^{\wedge})$ generated by its Euler class
\[
     e_{\T} (\tSUd^{a})  = \prod_{m_{j}\equiv 0 \mod n} \sigma (x_{j} +
     m_{j} z).
\]
On the other hand $\cI (f_{\T})^{\wedge}_{a}$ is generated by
\[
    f_{m} (\tSUd) = e_{\T} (\tSUd) = e_{\T} (\tSUd^{a}) e_{\T} (\tSUd/ 
\tSUd^{a}),
\]
where
\[
     e_{\T} (\tSUd/\tSUd^{a}) = \prod_{m_{j}\not\equiv 0\mod n} \sigma  
(x_{j} +
     m_{j}z).
\]
The argument in the proof of Proposition \ref{t-pr-delta-a-p-hol}
applies: the $x_{j}$ are topologically
nilpotent, and so $e_{\T} (\tSUd/\tSUd^{a})$ is a unit of $HP^{0}
(Z;\O_{a}^{\wedge})^{\times},$ and
$f_{m} (\tSUd)$ generates the same ideal as $e_{\T} (\tSUd^{a}).$
\end{proof}

We can display the situation described by the Proposition in the
following diagram, in which each square is a pull-back, and the curved
arrows are trivializations of the indicated sheaves.
\[
\xymatrix{
{\Loo (\tSUd)\otimes \ShfE (Y^{\tSUd})}
  \ar[r]^-{\iso}
  \ar[d]
&
{\Loo_{\T}\otimes \cI (f_{\T})}
  \ar[d]
  \ar[r]
&
{\Loo \otimes \cI (\sigma)}
  \ar[d]
\\
{\SpcE (Y)}
  \ar[r]^-{\iso}
  \ar@/_1pc/[u]_{\sigma (V)}
&
{\Y}
  \ar[r]^-{\mu_{\T}}
  \ar@/_1pc/[u]_{f_{\T}}
&
{(\cochars\otimes \J)/W}
  \ar@/_1pc/[u]_{\sigma}
}
\]
Thus a  $\T$-equivariant $SU (d)$ bundle $V$ over a $\T$-fixed space $Z$
gives rise to a map
\[
          h: \SpcE (Z) \to \Y,
\]
and we can form the line bundles
\begin{align*}
           \Loo (V) & = h^{*}\Loo_{\T}\\
           \cI (V) & =  h^{*} \cI (f_{\T})
\end{align*}
over $\SpcE (Z).$  The section
\[
\sigma (V) =  h^{*}f_{\T}
\]
is a trivialization of $\Loo (V)\otimes \cI (V)$,
and we have a canonical isomorphism of line bundles
\[
        \ShfE (Z^{V}) \iso \cI (V).
\]
In Section \ref{sec:ptsoffiniteorder}, we
explain how to handle the full space $BSU (d)$, and so spaces on
which $\T$ acts non-trivially.  Before doing so,
we discuss the sigma orientation for $\T$-fixed spaces $Z$.

Suppose that for $i=0,1$, $V_{i}$ is a $\T$-equivariant $SU
(d)$-bundle over a $\T$-fixed space $Z$, and let
$V=V_{0}-V_{1}$.  We then have two maps
\[
     h_{i}: \SpcE (Z) \to \X{\T},
\]
and we can form
the line bundles $\Loo (V_{i}) = h_{i}^{*}\Loo_{\T}$ with
sections $\sigma (V_{i})$ as above.  The ratio
\[
    \sigma (V) = \frac{\sigma (V_{0})}{\sigma (V_{1})}
\]
is a trivialization of
\[
\frac{\Loo (V_{0}) \otimes \cI (V_{0})}
      {\Loo(V_{1}) \otimes \cI (V_{1})} \iso
\frac{\Loo (V_{0}) \otimes \ShfE (Y^{V_{0}})}
      {\Loo(V_{1}) \otimes \ShfE (Y^{V_{1}})}.
\]
\begin{Proposition} \label{t-pr-ch-2-zero-Loo-triv}
If $\ChB_{2}
(V) = 0$ then
\[
      \Loo (V_{0})\iso \Loo (V_{1}).
\]
\end{Proposition}

Thus if $\ChB_{2} (V)=0,$ then $\sigma (V)$  gives a trivialization of
\[
      \ShfE (Z^{V_{0}}) \otimes \ShfE (Z^{V_{1}})^{-1}
\]
as a $\ShfE (Z)$-module.    This is precisely  our Thom class $\stror  
(V)$ from
Theorem \ref{t-th-sigma-equiv}.

\begin{proof}[Proof of Proposition \ref{t-pr-ch-2-zero-Loo-triv}]
We can factor the maps
\[
     h_{i}: Z \to BSU (d)
\]
as
\[
     h_{i}: Z\to BZ (m_{i}),
\]
where $m_{i}: \T\to T$ are cocharacters.  Then the formula
\eqref{eq:47} shows that $\Loo_{m}=\Loo_{m_{0}}\otimes
\Loo_{m_{1}}^{-1}$ is obtained from the line bundle
\[
     \frac{(\cochars \otimes \C^{\times})^{2} \times \C^{\times}  
\times \C}
                {(u_{0},u_{1},z,\lambda) \sim
                 (u_{0},u_{1},zq^{k},
         \lambda u_{0}^{-kI (m_{0})} u_{1}^{kI (m_{1})}
                 z^{-2k\qf (m_{0})+ 2k\qf (m_{1})} q^{-k^{2}\qf
		(m_{0}) + k^{2}\qf (m_{1})})}
\]
over
\[
      (\cochars \otimes \C^{\times})^{2} \times \J.
\]
by restricting to $(\cochars\otimes \Jhat)^{2}$ and then descending to
\[
      ((\cochars\otimes \Jhat)/W (m_{0})) \times ((\cochars\otimes
      \Jhat)/W (m_{1})) \times  \J.
\]
By Lemma \ref{t-le-equiv-ch-classes-for-covers}, if $\ChB_{2} (V) =
0$, then
\[
     \qf (m_{0}) = \qf (m_{1}),
\]
and  over $Z,$
\[
    u_{0}^{I (m_{0})} = u_{1}^{I (m_{1})},
\]
and so this line bundle is trivial.
\end{proof}

\section{Analytic geometry of the sigma orientation II: finite  
isotropy}
\label{sec:ptsoffiniteorder}

In Section \ref{sec:analyt-geom-sigma}, we considered the fixed point  
set
$Y=BSU(d)^{\T}$ and constructed a model
\[
     \SpcE (BSU (d)^{\T}) \iso \mathfrak{Y}.
\]
In this section we consider the entire space $X=BSU(d)$ and
extend the analysis to give a model
\[
     \SpcE (BSU (d)) \iso \mathfrak{X},
\]
beginning with
\[
    \X{\T}= \SpcE (BSU (d))_{\T} \iso \SpcE (BSU (d)^{\T})_{\T} \iso  
\Y_{\T}.
\]
In Section \ref{sec:analyt-geom-sigma}, we also constructed a map
\[
       \mu_{\T}: \SpcE (BSU (d)^{\T}) \xra{} (\cochars\otimes \J)/W
\]
of $\CatE$-spaces over $C$.
Over $(\cochars\otimes\J)/W$ we have a line bundle $\Loo,$ equipped
with a holomorphic section $\sigma$ defined by products of the  
Weierstrass
sigma function.  Turning to the Thom space of the tautological bundle
$\tSUd$ over $BSU(d), $  we showed that   $\ShfE ((BSU (d)^{\T})^ 
{\tSUd})$
is the ideal sheaf on $\SpcE (BSU (d)^{\T})$ of zeroes of
$\mu_{\T}^{*}\sigma.$
In this section, we give a similar analysis
of $\SpcE (BSU (d))_{\A}$  and $\ShfE (BSU (d)^{\tSUd})_{\A}$ for   
finite $\A$,
and assemble the information to give a global analysis of the
$\cE$-space line bundle  $\cF (BSU (d)^{\tSUd})$ over $\mathfrak{X}
=\SpcE (BSU(d))$.

\subsection{Overview}

It is not hard to describe the basic idea.
Let $m: \A\to T$ label a component $BZ (m)$ of $BSU (d)^{\A}$.
Since $H^{*}(BZ (m))$ is concentrated in even degrees,
Lemma \ref{t-le-completion} implies that there is
a (non-canonical) isomorphism
\begin{equation} \label{eq:37}
HP^{0}_{\T} (BZ (m); \OA) \cong HP^{0} (BZ (m); \OA).
\end{equation}

The right-hand side is the ring of functions on
\[
    (\cochars\otimes \Jhat)/W (m) \times \J^{\wedge}_{\A}.
\]
The idea is to construct a map
\begin{equation}\label{eq:38}
    \mu_{m}: (\cochars\otimes \Jhat)/W (m)\times \J^{\wedge}_{\A} \to
     (\cochars\otimes \J) / W,
\end{equation}
like the map in \eqref{eq:34}.  Then we can form the line
bundle
\[
\Loo_{m} = \mu_{m}^{*} \Loo
\]
with section
\[
    f_{m} = \mu_{m}^{*}\sigma,
\]
and identify $\ShfE (X^{\tSUd})_{\A}$ with $\cI (f_{m})$,
as in the $\T$-fixed case.

There are two related problems.  The first is that the isomorphism
\eqref{eq:37} is not canonical, and we must be able to construct the
map $\mu_{m}$ compatibly with restriction to $BSU (d)^{\T}$ in order
to extend the analysis of Section \ref{sec:analyt-geom-sigma}.    The  
second
is that the homomorphism
\[
m = (m_{1},\dotsc ,m_{d}):    \A\to T
\]
does not quite determine $\mu_{m}$ as in \eqref{eq:38}.  We do get a
homomorphism
\[
    \J[\A] \to (\cochars\otimes \J)^{W (m)},
\]
and so a map
\begin{equation} \label{eq:50}
    \mu_{m}^{\text{weak}}: (\cochars\otimes \J)/W (m) \times \J[\A]
    \xra{\id\times m} (\cochars\otimes \J)/W (m) \times
    (\cochars\otimes \J)^{W (m)}   \xra{} \cochars\otimes \J / W,
\end{equation}
analogous to  \eqref{eq:34}.  The problem is to extend this map to the
formal neighborhood $\J^{\wedge}_{\A}$ of $\Jlr{\A}$.

\subsection{The $\protect \cE$-space associated to $BSU(d)$}
The $\T$-action on $BZ (m) \subset BSU (d)$ (see Remark \ref{rem-7})
provides the extra information we need.  As we have already observed
in equation \eqref{eq:44}, the Borel construction
\[
    \tSUd\times_{\T}E\T \to BZ (m)\times_{\T} E\T
\]
is classified by a map
\[
    BZ (m)\times_{\T}E\T \to BZ (m),
\]
inducing a map
\begin{equation}\label{eq:48}
     \mu_{m,\e}: \spf HP^{0} ( BZ (m)\times_{\T}E\T) \to (\cochars 
\otimes\Jhat)/W (m).
\end{equation}
A choice of isomorphism
\begin{equation} \label{eq:49}
    HP^{0}(BZ (m)\times_{\T} E\T) \iso HP^{0}(BZ (m)\times B\T)
\end{equation}
permits us to view $\mu_{m,\e}$  as a map
\[
  (\cochars\otimes\Jhat)/W (m)\times \Jhat \rightarrow
   (\cochars\otimes \Jhat)/W (m) \rightarrow (\cochars\otimes \J)/W (m),
\]
giving us the desired map \eqref{eq:38} in a formal neighborhood of $ 
\e.$
We can then define $\mu_{m}$ at a point $a$ of exact order $n$ by
translation, noting that the diagram
\begin{equation} \label{eq:54}
\begin{CD}
(\cochars\otimes \Jhat)/W (m) \times \J[A] @> \mu^{\text{weak}}_{m}>>
(\cochars\otimes \J)/W (m) \\
@A 1\times T_{a} AA @AA T_{m (a)} A \\
(\cochars\otimes \Jhat)/W (m) \times \J[A] @> \mu^{\text{weak}}_{m}>>
(\cochars\otimes \J)/W (m)
\end{CD}
\end{equation}
commutes.

In doing so, there is little to be gained by choosing the isomorphism
\eqref{eq:49}.  Instead, we note that $\spf HP^{0}_{\T} (BZ (m))$ is
in any case a formal scheme over $\Jhat\iso \spf HP^{0}(B\T)$, and so
for $m:\A\to \cochars$ labelling a component
$BZ (m)$ of $BSU (d)^{\A}$, we define $\X{m}$ to be the formal scheme
over $\J^{\wedge}_{\A}$ given by
\[
    \X{m} \eqdef \coprod_{a\in \Jlr{\A}} T_{-a}^{*}\spf HP^{0}_{\T}(BZ  
(m)).
\]
By construction, for $w\in W$ we have natural isomorphisms
\[
       \X{m} \to \X{wm},
\]
and so setting
\[
     \X{\A} \eqdef \left(\coprod_{m:\A\to \T} \X{m} \right)/W,
\]
we have an isomorphism
\[
       \SpcE (BSU (d))_{\A} \iso \X{\A}.
\]
Since by Proposition \ref{t-pr-E-BZm-O-scr-X},
\[
            (\SpcE (BSU (d))_{\T})^{\wedge}_{\Jlr{A}} \iso  (\X{\T})^ 
{\wedge}_{\Jlr{A}},
\]
we get a map
\[
     \X{\A} \iso \SpcE (BSU (d))_{\A} \xra{}
     (\SpcE (BSU (d))_{\T})^{\wedge}_{\Jlr{A}} \iso
     (\X{\T})^{\wedge}_{\Jlr{A}},
\]
and so $\X{}$ is a $\CatE$-space.

For $m: A\to \cochars$, let
\[
  \mu_{m}:\X{m} \to (\cochars\otimes \J)/W (m)
\]
be the map which on the $a$ component of $\X{m}$ is given by
\[
   \spf HP^{0}_{\T} (BZ (m)) \xra{\mu_{m,\e}} (\cochars\otimes \J)/W (m)
   \xra{T_{m (a)}} (\cochars\otimes \J)/W (m).
\]
It is easy to check that the $\mu_{m}$ for $m:\A\to T$ induce a map
\[
    \mu_{\A}: \X{\A} \to (\cochars\otimes \J)/W.
\]
Then we have the following.

\begin{Proposition}\label{t-pr-mu-t-mu-a}
For $\A\subset \T$, the diagram
\[
\begin{CD}
   (\X{\T})^{\wedge}_{\Jlr{\A}} @>>>\X{\A}  \\
@VVV @VV \mu_{\A} V \\
\X{\T}  @> \mu_{\T}>> (\cochars\otimes \J)/W
\end{CD}
\]
commutes.  Thus the isomorphisms of ringed spaces above fit together  
to give
an isomorphism
\[
      \SpcE (BSU (d))  \iso \X{}
\]
of $\CatE$-spaces over $C$, and a map
\[
      \mu: \X{} \to (\cochars\otimes \J)/W.
\]
\end{Proposition}

\begin{proof}
We consider what is  happening over a
particular point $a\in C$ of exact order $|A|.$  We also work one
component at a time: fix a pair of homomorphisms
\[
\xymatrix{
{\T} \ar[r]^{\tm}
&
{T}
\\
{\A,}
  \ar[u]
  \ar[ur]_{m}
}
\]
labelling a component $BZ (m)$ of $BSU (d)^{\A},$ and a component $BZ
(\tm)$ of $BZ (m)^{\T}.$  We must show that the diagram
\begin{equation}\label{eq:55}
\begin{CD}
     (\X{\tm})^{\wedge}_{\Jlr{\A}} @>>> \X{m} \\
@VVV @VV \mu_{m} V \\
     \X{\tm} @> \mu_{\tm} >> (\cochars\otimes C)/W
\end{CD}
\end{equation}
commutes.

Note that the diagram
\[
\begin{CD}
(\cochars\otimes \J)/W (\tm)\times C^{\wedge}_{0} @> 1\times T_{a} >>
(\cochars\otimes \J)/W (\tm)\times C^{\wedge}_{a} \\
@V\mu_{\tm} VV @VV \mu_{\tm} V \\
(\cochars\otimes \J)/W (\tm) @> T_{\tm (a)} >>
     (\cochars\times \J)/W (\tm)
\end{CD}
\]
commutes: if $\tm= (m_{1},\dotsc ,m_{d}) \in \cochars
\subset \Z^{d},$ then either composition sends the element of
$(\cochars\otimes \J)/W (\tm)\times C^{\wedge}_{0}$ represented by $(x_ 
{1},\dotsc ,x_{d},z) \in (\cochars\otimes \J) \times
C^{\wedge}_{\e}$ to the class of
\[
    (x_{1} + m_{1} (z + a),\dotsc ,x_{d} + m_{d} (z+a))
\]
in $(\cochars\otimes \J)/W (\tm).$  Thus the counterclockwise  
composition in the
diagram \eqref{eq:55} at $a$ may be replaced by the top row in the
diagram
\[
\begin{CD}
(\X{\tm})^{\wedge}_{a} @> 1\times T_{-a} >>  (\X{\tm})^{\wedge}_{\e}
@> \mu_{\tm,\e} >> (\cochars\otimes \J)/W (\tm)
@> T_{\tm (a)} >>   (\cochars\otimes \J)/W (\tm) \\
@VVV @VVV @VVV @VVV \\
(\X{m})_{a}  @> T_{-a}^{*} >> (\X{m})_{\e} @> \mu_{m,\e}  >>
(\cochars\otimes \J)/W (m)  @> T_{m (a)} >>   (\cochars\otimes \J)/W  
(m).
\end{CD}
\]
The commutativity of the first square and third squares is
straightforward.   The commutativity of the second square follows from
Lemma \ref{t-le-mu-topology} and the commutativity of the diagram
\[
\xymatrix{
{BZ (\tm)\times_{\T} E\T}
  \ar[r]
  \ar[dr]
&
{BZ (m)\times_{\T} E\T}
  \ar[d] \\
&
{BZ (m).}
}
\]
\end{proof}

\subsection{Building the line bundle  of the Thom space over $BSU(d)$}
\label{sec:building-line-bundle}

Once again, let $\tSUd$ denote the tautological bundle over $BSU (d).$
  Having described the $\CatE$-space $\SpcE (BSU (d))$ over $C$, we
  turn to the $\CatE$-sheaf $\ShfE (BSU (d)^{\tSUd})$ over $C$.
As in Section \ref{sec:analyt-geom-sigma}, which treated the case
$\A=\T$, we can now define for finite $\A$ and $m: \A\to T$ a line  
bundle
\[
       \Loo_{m}  \eqdef \mu_{m}^{*}\Loo
\]
with section
\[
       f_{m} \eqdef \mu_{m}^{*}\sigma
\]
over $\X{m}$, and
\begin{align*}
    \Loo_{\A} & \eqdef \mu_{\A}^{*}\Loo \\
    f_{\A} & \eqdef \mu_{\A}^{*}\sigma
\end{align*}
over $\X{\A}.$  We have the following analogue for finite $\A$ of
Proposition \ref{t-pr-f-m-zeroes-thom}.

\begin{Proposition}\label{t-pr-IfA-FBSU}
The isomorphism
\[
    \SpcE (BSU (d))_{\A} \iso \X{\A}
\]
identifies $\ShfE (BSU (d)_{\A}^{\tSUd})$, considered as a sheaf of  
ideals on
$\SpcE (BSU (d))_{\A}$ via the zero section, with the ideal sheaf $\cI  
(f_{A}).$
\end{Proposition}

\begin{proof}
%
Let $m: A\to T$ be a homomorphism, labelling a component $BZ (m)$ of
$BSU (d)^{A}$.  Let $a \in \Jlr{A}.$  Suppose that $\ta$ is a lift of
$a$ to $\C$, and $\alpha=e^{\ta}$ is  the corresponding lift of $a$ to
$\C^{\times}$.  Let $k$ be the integer such that
\[
      \alpha^{n} = q^{k}.
\]
Let $\tm$ be a cocharacter
making the diagram
\[
\xymatrix{
{\T} \ar[r]^{\tm}
&
{T}
\\
{\A}
  \ar[u]
  \ar[ur]_{m}
}
\]
commute.  Recall that $\mu_{\A}$ is given near $a$ by
\[
     \spf HP^{0}_{\T} (BZ (m)) \xra{\mu_{m,0}} (\cochars\otimes \J)/W
     (m) \xra{T_{m (a)}} (\cochars\otimes C)/W (m).
\]
Recall from  \eqref{eq:24} that $\Loo \to (\cochars\otimes \J)/W$ is
obtained from the line bundle
\begin{equation}\label{eq:42}
\frac{(\cochars \otimes \C^{\times})\times \C}
                {(u,\lambda) \sim
                 (uq^{m}, \lambda u^{-I (m)} q^{-\qf (m)})}
\end{equation}
over $\cochars\otimes \J$, and in terms of this trivialization $\sigma$
is the section
\[
     u \mapsto (u, \sigma (u,q)).
\]
Recall that
\[
      f_{A} = \mu^{*}\sigma = \mu_{m,\e}^{*}T_{m (a)}\sigma,
\]
where the maps in question are
\[
\spf HP^{0}_{\T} (BZ (m)) \xra{\mu_{m,\e}} (\cochars\otimes \J)/W (m)
   \xra{T_{m (a)}} (\cochars\otimes \J)/W (m).
\]

Up to a unit multiple (depending on the choices $\tm$ and $\ta$),
$T_{m (a)}^{*}\sigma$ is given near the origin
in $\cochars\otimes \C^{\times}$ by
\[
       T_{\alpha^{\tm}}^{*}\sigma (u,q) = \sigma (u\alpha^{\tm},q).
\]
On the other hand, $\ShfE (BZ (m)^{\tSUd})_{\A}$ is the ideal in
\[
\ShfE (BZ(m))_{A} \iso HP^{0}_{\T} (BZ (m);\OA)
\]
generated $e_{\T} (\tSUd^{\A})$, and so $T_{a}^{*}\ShfE (BZ
(m)^{\tSUd})_{A,a}$ is the ideal in $HP^{0}_{\T} (BZ
(m),\J^{\wedge}_{\e})$ generated by $T_{a}^{*}e_{\T} (\tSUd^{\A}).$     
These
two classes can't be compared
directly, because $\sigma (u\alpha^{\tm},q)$ does not descend to
$(\cochars\otimes \C^{\times})/W (m)$.  It is for precisely this
reason that in Section \ref{sec:class-delta_a} we introduced the class
\[
     \deltaA (u,\tm,\ta) = u^{\frac{k}{n}I (\tm)}
     \alpha^{\frac{k}{n}\phi (\tm)} T_{\alpha^{\tm}}^{*}\sigma (u,q).
\]
We showed that this class does descend to a function $\deltaA
(u,m,\ta)$ on  $(\cochars\otimes \C^{\times})/W (m)$.  As such it
defines by restriction to $(\cochars\otimes \Jhat)/W (m)$ an element
$\deltaA (\tSUd,m,\ta) \in
HP^{0} (BZ (m),\Jhat_{\e}).$  Comparison of \eqref{eq:44} and \eqref 
{eq:48}
shows that the associated Borel class $\deltaBA$  from Section
\ref{sec:class-delta_a} is precisely
\[
\deltaBA = \deltaBA (\tSUd,m,\ta) =   \mu_{m,0}^{*}\delta_{A} (\tSUd,m, 
\ta)
\in HP^{0}_{\T} (BZ (m);\J^{\wedge}_{\e}).
\]

On the one hand, $\deltaA (\tSUd,\tm,\ta)$ is a unit multiple of
$T_{\alpha^{\tm}}^{*}\sigma$, and so $\deltaBA (\tSUd,m,\ta)$ is a unit
multiple of $f_{A}$.  On the other hand, Propositions
\ref{t-pr-delta-pp} and \ref{t-pr-delta-a-p-hol}
show that $\deltaBA = \deltaBAp \deltaBApp$
is a unit multiple of $T_{a}^{*}e_{\T} (\tSUd^{A}).$  Thus the two ideal
sheaves $\cI (f_{\A})$ and $\ShfE (BSU (d)^{\tSUd})_{\A}$ coincide, as  
required.
\end{proof}

We now turn to the analogue of Proposition
\ref{t-pr-ch-2-zero-Loo-triv}, and show that the line bundles
$\Loo_{A}$ are canonically trivialized after restriction to
$\SpcE (\bstringc)$.

The expression
\[
\deltaA (u,m,\alpha) = u^{\frac{k}{n}I (\tm)}\alpha^{\frac{k}{n}\phi  
(m)}\sigma (u\alpha^{\tm},q)
\]
defines a function on $(\cochars\otimes \C^{\times})/W (m)$ near the
origin, and as we show in Proposition \ref{t-pr-IfA-FBSU}, it gives a
trivialization of
\[
(T_{m (a)}^{*}I (\sigma))_{0} \iso I (\sigma)_{m (a)}
\]
over $(\cochars\otimes \C^{\times})/W (m)$ and so over
$(\cochars\otimes \Jhat)/W (m)$.   It follows that
\[
\deltaBA (\tSUd,m,\alpha) = \mu_{m,0}^{*}\deltaA (\tSUd,m,\alpha)
\]
is a trivialization of $\cI (f_{A})$ over $\X{m,a}.$

This trivialization depends on the choice of lift $\alpha$, but if
$\alpha'$ is another lift of $a$ related to $\alpha$ by
\[
      \alpha' = \alpha q^{s},
\]
then we showed in Lemma \ref{t-le-delta-A-dep-ta} that
\[
     \deltaBA (\tSUd,m,\alpha') = w (a,q^{1/n})^{s\phi (m)}\deltaBA  
(\tSUd,m,\alpha),
\]
where $w (a,q^{1/n})$ is the Weil pairing of $a$ and (the image in
$\J$ of) $q^{1/n}.$

If we let $\ell$ be a lift as in the diagram
\begin{equation}\label{eq:57}
\xymatrix{
&
{\C^{\times}}
  \ar[d]
\\
{\Jlr{A}}
  \ar@{-->}[ur]^{\ell}
  \ar@{>->}[r]
&
{\J}
}
\end{equation}
then we can assemble the resulting functions $\deltaBA$ into a
trivialization $\deltaBA (\tSUd,\ell)$ of $\cI (f_{\A})$ over $\X{\A}. 
$  The
section $f_{A}$ tautologically gives a trivialization
of $\Loo_{A}\otimes \cI (f_{A}),$ and so the ratio
\[
   \LooAtr (\tSUd,\ell) =   f_{A}/\deltaBA (\tSUd,\ell)
\]
is a trivialization of $\Loo_{A}$ over $\X{\A}.$

If $\ell$ and $\ell'$ are two lifts, let
\[
    s (\ell', \ell) : \Jlr{\A} \to \Z
\]
be the function such that
\[
       \ell' (a)= \ell (a) q^{s (\ell' ,\ell) (a)}: \Jlr{\A} \to \C^ 
{\times}.
\]
For $m: \A\to T$ labelling a  component of $BSU (d)^{\A}$ and for $a
\in \Jlr{\A}$, then, we have the $n$th root of unity
\begin{equation}\label{eq:46}
        c (m,a,\ell',\ell)  = w (a,q^{1/n})^{s(\ell',\ell,a) \qf (m)}.
\end{equation}
On the right, $\qf (m)$ means $\qf (\tm)$ for any lift $\tm: \T\to T$
of $m;$ since $w (a,q^{1/n})^{n} = 1$ the right hand side does not
depend on this choice of lift.    Also for $w\in W$, $\qf (\tm) = \qf
(w \tm)$,
and so as $m$ and $a$ vary we obtain a locally constant
function
\[
c (\tSUd,\ell',\ell):   \X{\A} \to \mathbf{\mu}_{\A}\eqdef \{\zeta\in  
\C^{\times} | \zeta^{|A|} = 1 \},
\]
such that
\[
     \deltaBA (\tSUd,\ell') = c (\tSUd,\ell',\ell) \deltaBA (\tSUd, 
\ell),
\]
and so we have the following.

\begin{Lemma}  \label{t-le-c-gamma}
If $\ell'$ and $\ell$ are two lifts, then
\[
         \LooAtr (\tSUd,\ell) = c (\tSUd,\ell',\ell)\LooAtr (\tSUd, 
\ell').
\]\qed
\end{Lemma}

The lemma implies that the trivialization $\LooAtr_{\ell}$ becomes
independent of $\ell$ upon restriction to $\stringc$ bundles.   One
way to formulate this is the following.  Let $V_{0}$ and $V_{1}$ be
two $\T$-equivariant $SU (d)$-bundles over a $\T$-space $X$, and let
$V=V_{0}-V_{1}.$  Let
\[
    h_{i}: \SpcE (X)_{\A} \to \SpcE (BSU (d))_{\A}
\]
be the map induced by the map $X\to BSU (d)$ classifying $V_{i}.$  We
can then form the bundles $\Loo (V_{i})_{\A}$ over
$\SpcE (X)_{\A}$ by pulling back as in the diagram
\[
\begin{CD}
\Loo (V_{i})_{\A} @>>> \Loo_{\A} @>>> \Loo \\
@VVV @VVV @VVV \\
\SpcE (X)_{\A} @> h_{i}>> \SpcE (BSU (d))_{\A} @> \mu_{\A}>> \X{}.
\end{CD}
\]
Let
\begin{align*}
    \LooAtr (V_{i},\ell) &= h_{i}^{*}\LooAtr (\tSUd,\ell) \\
    \LooAtr (V_{i},\ell')&= h_{i}^{*}\LooAtr (\tSUd,\ell')
\end{align*}
be the indicated trivialization of $\Loo (V_{i})_{\A}$, and let $c
(V_{i},\ell',\ell)$ be the composition
\[
     c (V_{i},\ell',\ell): \SpcE (X)_{\A} \xra{h_{i}} \SpcE (BSU
     (d))_{\A} \xra{c (\tSUd,\ell',\ell)} \mathbf{\mu}_{\A}.
\]
Finally, let
\begin{align*}
     \Loo (V)_{\A} & = \Loo (V_{0})_{\A}\otimes \Loo (V_{1})^{-1} \\
     \LooAtr (V,\ell) & = \LooAtr (V_{0},\ell)\otimes \LooAtr
     (V_{1},\ell)^{-1} \\
     \LooAtr (V,\ell') & = \LooAtr (V_{0},\ell')\otimes \LooAtr
     (V_{1},\ell')^{-1} \\
     c (V,\ell',\ell) & = c (V_{0},\ell',\ell) c (V_{1},\ell',\ell)^ 
{-1}.
\end{align*}
Then Lemma \ref{t-le-c-gamma} implies that
\[
    \LooAtr (V,\ell) = c (V,\ell',\ell)\LooAtr (V,\ell'),
\]
and we have the following.

\begin{Proposition}  \label{t-pr-c-restr-bstring-zero}
If $\ChB_{2} (V) = 0$ then $c (V,\ell,\ell') = 1,$ and so $\LooAtr
(V,\ell) = \LooAtr (V,\ell')$.  In particular, over
over $\SpcE (X)_{\A}$ we have a canonical trivialization
$\LooAtr_{\A} = f_{\A}/\deltaBA$ of $\Loo_{\A}$, independent of choice
of lift $\ell$.
\end{Proposition}

\begin{proof}
The argument is the same as in the proof of Proposition
\ref{t-pr-ch-2-zero-Loo-triv}.  Let $Z$ be a component of $X^{\A}.$
We can factor the maps
\[
    h_{i}: Z\to  X^{\A} \to BSU (d)^{\A}
\]
as
\[
     h_{i}: Z\to BZ (m_{i}),
\]
where $m_{i}: \A\to T$ are homomorphisms.  By \eqref{eq:46}, on $\SpcE
(Z)_{\A,a}$, $c (\ell,\ell',V)$ is given by the formula
\[
  c (V,\ell,\ell') =
w (a,q^{1/n})^{s(\ell,\ell',a) (\qf (m_{0})-\qf (m_{1}))}.
\]
By Lemma \ref{t-le-equiv-ch-classes-for-covers}, if $\ChB_{2} (V) =
0$, then
\[
     \qf (m_{0}) \equiv \qf (m_{1}) \mod |\A|.
\]
\end{proof}

\subsection{Trivializing line bundles of string bundles II: the global  
case}

Finally, we may draw the threads together, describing the line bundle  
given
by the  $\cE$-sheaf of a Thom space and showing that there is a  
canonical
trivialization for $String_{\C}$-bundles.  According to Proposition
\ref{t-pr-mu-t-mu-a}, we have an isomorphism of $\cE$-spaces over $C$
\begin{equation} \label{eq:35}
     \SpcE (BSU (d)) \iso \X{},
\end{equation}
and a map
\[
    \mu: \X{} \to (\cochars\otimes \J)/W.
\]
Over $(\cochars\otimes \J)/W$ we have the line bundle $\Loo$, equipped
with its section $\sigma$.  Let $f=\mu^{*}\sigma$ be the resulting
family of sections of $\mu^{*}\Loo$ over $\X{}$, and let $\cI (f)$ be
the associated ideal sheaf.
Combining Propositions
\ref{t-pr-f-m-zeroes-thom} and \ref{t-pr-IfA-FBSU}, we have the  
following.

\begin{Theorem}  \label{t-th-ShfE-BSU-V}
Let $\tSUd$ denote the tautological bundle over $BSU (d).$ Pull-back  
along the zero section identifies $\ShfE (BSU (d)^{\tSUd})$ with
a sheaf of ideals on $\SpcE (BSU (d)).$
The isomorphism \eqref{eq:35} carries the ideal sheaf $\ShfE (BSU
(d)^{\tSUd})$ to the ideal sheaf $\cI (f).$
\qed
\end{Theorem}

If $X$ is a $\T$-space, and $V$ is a $\T$-equivariant
$SU (d)$-bundle over $X,$ then the map
\[
     X \to BSU (d)
\]
classifying $V$ induces a map
\[
    h: \SpcE  (X) \xra{} \X{}
\]
of $\CatE$-spaces over $C$.  Thus we can form the line bundle
\[
\Loo (V) \eqdef h^{*}\mu^{*}\Loo
\]
with section $\sigma (V) = h^{*}\mu^{*}\sigma$.  Let
\[
\cI (V) \eqdef  h^{*}\cI (f):
\]
then $\sigma (V)$ is a trivialization of
\[
         \Loo (V)\otimes \cI (V) \iso \Loo (V)\otimes \ShfE (X^{V}).
\]
The situation is illustrated in the diagram
\[
\xymatrix{
{\Loo (V)\otimes \ShfE (X^{V})}
  \ar[r]
  \ar[d]
&
{\Loo\otimes \cI (f)}
  \ar[d]
  \ar[r]
&
{\Loo \otimes \cI (\sigma)}
  \ar[d]
\\
{\SpcE (X)}
  \ar[r]^-{h}
  \ar@/_1pc/[u]_{\sigma (V)}
&
{\X{}}
  \ar[r]^-{\mu}
  \ar@/_1pc/[u]_{f}
&
{(\cochars\otimes \J)/W.}
  \ar@/_1pc/[u]_{\sigma}
}
\]

Now suppose that $V_{0}$ and $V_{1}$ are two $\T$-equivariant $SU
(d)$-bundles over $X$.
Let $V=V_{0}-V_{1}$  and set
\[
        \Loo (V) = \Loo (V_{0})\otimes \Loo (V_{1})^{-1},
\]
and so forth.

\begin{Theorem} \label{t-th-anal-geom-sigma-orientation}
If $V=V_{0}-V_{1}$ is a difference of $SU(d)$-bundles over $X$ and
$\ChB_{2} (V) = 0$ then the bundle $\Loo (V)$ is trivialized,
and so $\sigma (V)$ may be viewed as trivialization of
\[
\cI (V) \iso \ShfE (X^{V}).
\]
This trivialization of the $\CatE$-sheaf $\ShfE (X^{V})$ coincides
with the Thom class $\stror (V)$ provided by Theorem
\ref{t-th-sigma-equiv}.
\end{Theorem}

\begin{proof}
The
identification $\cI (V) \iso \ShfE (X^{V})$ follows from Theorem
\ref{t-th-ShfE-BSU-V}.  Proposition \ref{t-pr-ch-2-zero-Loo-triv}
provides a trivialization of $\Loo (V)_{\T}$, and Proposition
\ref{t-pr-c-restr-bstring-zero}
provides a trivialization of $\Loo (V)_{\A}.$  The fact that $\ChB_{2}
(V) = 0$ implies, as in Proposition \ref{t-pr-delta-fixed-busix}, that
these trivilizations are compatible on $(\SpcE (X)_{\T})^{\wedge}_{\Jlr 
{\A}}$:
precisely,
\[
\LooAtr_{\A}\restr{(\SpcE (X)_{\T})^{\wedge}_{\Jlr{\A}}} =
(f_{\A}/\deltaBA)\restr{(\SpcE (X)_{\T})^{\wedge}_{\Jlr{\A}}} = 1.
\]
The formulae  for the sections $\sigma (V)$ are the
same as the formulae we have given for $\stror (V)$.
\end{proof}

We conclude this section by reformulating Theorem
\ref{t-th-anal-geom-sigma-orientation} in terms of the
universal bundle $\tStringc$ over $\bstringc$.  We briefly recall
the construction of Subsection \ref{subsec:univbundles}, where
  $\tSUd$ denotes the tautological bundle over $BSU (d)$,
  $W$ denotes a complex representation of $\T$
of rank $d$ and  determinant $1$  and   $i_{W}: BSU (d)\to BSU$
is the map classifying $\tSUd-W.$  As described in Subsections
\ref{sec:class-space-stable-SU} and \ref{subsec:univbundles},
$BSU$ is the colimit of maps of this kind.  Finally, we let
$\bstringc (W)$ be the pull-back in the diagram
\[
\begin{CD}
\bstringc (W) @>>> \bstringc  \\
@VVV @VVV \\
BSU (d) @> i_{W} >> BSU,
\end{CD}
\]
and we let
\[
\tStringc_{W} \iso (\tSUd - W)\restr{\bstringc (W)}
\]
denote the tautological $\stringc$ bundle over $\bstringc (W).$

We may now restate the theorem for the universal examples.

\begin{Corollary} \label{t-co-bstring-W-orientation}
The bundle $\Loo (\tStringc_{W})$ is canonically trivialized, and so $ 
\sigma
(\tStringc_{W})$ is a trivialization of
\[
\cI (\tStringc_{W}) \iso \ShfE (\bstringc (W)^{\tStringc_{W}}).
\] \qed
\end{Corollary}

\section{Cohomology of unitary groups and moduli spaces of divisors}
\label{sec:alg-geom-sigma}

We give an account in terms of divisors of the analysis in
Section \ref{sec:analyt-geom-sigma}.  It is illuminating to do
so for its own sake, and it indicates an approach to the
$\T$-equivariant sigma orientation for general elliptic curves over $\Q 
$-algebras (i.e., those for which analytic methods are not available).

\subsection{The classical non-equivariant description}

The starting point is the observation that if $BU (d)$ denotes the
\emph{nonequivariant} classifying space for $U (d)$-bundles, then
\[
    \spf HP^{0}(BU (d))\iso \Jhat^{d}/\Sigma_{d} \iso \Div_{+}^{d}  
(\Jhat)
\]
is the formal scheme of effective divisors of degree $d$ on $\Jhat$, the
formal group of $\J$.  (We continue to work with an elliptic curve  in
the form
\[
     \J = \C/\Lambda,
\]
so the projection $\C\to \J$ induces an isomorphism of formal groups
$\Jhat \iso \Gah.$)  Moreover,
the determinant
\[
        BU (d) \rightarrow BU (1)
\]
corresponds to the map
\[
       \Div_{+}^{d} (\Jhat) \rightarrow \Jhat
\]
which sends a divisor  $\sum (P)$ to $\sum^{\Jhat} P$,
so
\[
        \spf HP^{0}(BSU (d))\iso \Div_{+}^{d} (\Jhat)_{0}
\]
is the scheme of effective divisors which sum to zero in $\Jhat.$
(See Strickland \cite{MR1718087}).

\subsection{Centralizers}

Now we consider the $\T$-equivariant classifying space.  Let $T$ be
the maximal torus of $U (d)$, and let
\[
     m = (m_{1},\dotsc ,m_{d}):  \T\to T
\]
be a cocharacter, corresponding to a component $BZ (m)$ of $BU (d)$.
Let us suppose that we have arranged $m$ in increasing order, so it is
of the form
\begin{equation}\label{eq:36}
\begin{split}
   m_{e_{0}+1}=\dotsb =m_{e_{1}} & <
   m_{e_{1}+1} = \dotsb  =  m_{e_{2}} \\
     & < \dotsb \\
     & < m_{e_{r-1}+1} = \dotsb = m_{e_{r}},
\end{split}
\end{equation}
with $0=e_{0}<e_{1}<\dotsb <e_{r}=d$.  It is convenient also to number  
this
partition by setting $d_{i} = e_{i} - e_{i-1}$ for $1\leq i \leq r$, so
\begin{align*}
d_{i} & \geq 1\\
\sum_{i=1}^{r} d_{i} & = d.
\end{align*}
It is clear that every $m:\T\to T$ is conjugate to exactly one of this
form, and so these suffice to describe $BU (d)^{\T}$.  It is  also easy
to see that
\[
   Z (m) \iso U (d_{1})\times\dotsb \times U (d_{r}),
\]
and so
\[
    \spf HP^{0}(BZ (m)) \iso \Div_{+}^{d_{1}} (\Jhat)\times\dotsb
                           \Div_{+}^{d_{r}} (\Jhat):
\]
this is the scheme of $r$ effective divisors $D_{1},\dotsc ,D_{r}$,
with $\deg D_{i} = d_{i}$.  Another way to say this is that if we
write the tautological divisor $D$ over $\Div_{+}^{d} (\Jhat)\times \J$
as
\[
     D = \sum_{i}  P_{i},
\]
then the array $m$ labels each point $P_{i}$ with an integer $m_{i}$,
and $\spf HP^{0}(BZ (m))$ is the scheme of effective divisors labelled
with integers in this way.

\begin{Definition}\label{def-div-m}
Let $m: \T\to T$ be a cocharacter as in \eqref{eq:36}.  We define
\[
    \Div_{+}^{m} (\Jhat) \eqdef  \Div_{+}^{d_{1}} (\Jhat)\times\dotsb
                           \Div_{+}^{d_{r}} (\Jhat),
\]
and we write $D_{m}$ for the tautological divisor over $\Div_{+}^{m}
(\Jhat)\times \J$.
If $P$ is a point of $D_{m}$, then we write $m_{P}$ for its integer
label.  We write
\[
\Div_{+}^{m} (\Jhat)_{0} \eqdef \Div_{+}^{m} (\Jhat)\cap \Div_{+}^{d}
(\Jhat)_{0}
\]
for the subscheme consisting of divisors which sum to zero in $\Jhat$.
\end{Definition}

Strickland's ideas, applied to our
calculation of $H^{*}(BSU (d)^{\T})$ in Proposition
\ref{t-pr-splitting-princ-messy}, imply the following.

\begin{Proposition}\label{t-pr-co-bzm-divisors}
Let $T$ be the maximal torus of diagonal matrices in $SU (d)$, and let
\[
   m: \T\to T
\]
be a cocharacter, corresponding to a component $BZ (m)$ of $BSU
(d)^{\T}$.  Then
\[
     \spf HP^{0}(BZ (m))\iso \Div_{+}^{m} (\Jhat)_{0}.
\]\qed
\end{Proposition}

\begin{Definition}
If $P$ is a point of $\J$ and $n$ is an integer, let $D (P,n)$ be the
divisor
\[
D (P,n)	\eqdef  \sum_{\{a\in \J |na + P = 0\}} (a).
\]
\end{Definition}


\begin{Proposition} \label{t-D-P-m-and-f-m}
Let
\[
m= (m_{1},\dotsc ,m_{d}): \T\to T
\]
be a cocharacter of the form \eqref{eq:36}, labelling a component $BZ
(m)$ of $BSU (d)^{\T}$.  Let $\tSUd$ be the tautological bundle over  
$BSU (d)$.
If we write the tautological divisor over
\[
  \spf HP^{0}(BZ (m)) \times \J \iso Div_{+}^{m} (\Jhat)\times \J
\]
as
\[
     D_{m} = \sum_{i} (P_{i}),
\]
numbered so that $m_{i} = m_{P_{i}}$,
then $\EJc^{*} (BZ (m)^{\tSUd})$ (viewed as a sumbodule of
$\EJc^{*} (BZ (m))$ via pullback along the zero section)
is the cohomology of the ideal sheaf  of
the divisor
\[
         \sum_{i} D (P_{i},m_{i}).
\]
\end{Proposition}

\begin{proof}
By Proposition \ref{t-pr-f-m-zeroes-thom}, it is equivalent to show
that
\[
    \Div (f_{m}) =
         \sum_{i} D (P_{i},m_{i}).
\]
The reduction $m$ corresponds to decomposing $\tSUd$ into isotypical
summands according to the action of $\T$.  The choice of $P_{i}$
corresponds to using the splitting principle to decompose
the tautological bundle
$\tSUd$ over $BSU (d)$ as
\[
     \tSUd\restr{BZ (m)} \iso L_{1}\otimes \C (m_{1}) \oplus\dotsb
                           L_{d} \otimes \C (m_{d}).
\]
If $x_{i} = c_{1}L_{i}$, then
\[
     f_{m} (x,z) = \prod_{i} \sigma (x_{i} + m_{i} z,\tau).
\]
The result follows from the fact that $\sigma$ vanishes to first order
at the points of the lattice, and nowhere else.
\end{proof}

\subsection{The global equivariant picture}

We now ask, what is the failure of $D_{m}=\Div (f_{m})$ to be the  
divisor of
a function on $\J$?   The Riemann-Roch Theorem gives two conditions.

\begin{Proposition}
Let $m:\T\to T$ be a cocharacter, corresponding to a component of $BSU
(d)^{\T}$ and let  $\qf$ denote the function from  Subsection
\ref{sec:chern-classes}.   Then
\[
    \deg (f_{m}) = 2\qf (m),
\]
and if we write
\[
D_{m} = \sum_{i} P_{i}
\]
as in Proposition \ref{t-D-P-m-and-f-m}, then $D_{m}$ sums to
\[
      \sideset{}{^\J}\sum m_{i} P_{i}
\]
in the elliptic curve.
\end{Proposition}

\begin{proof}
Let us examine a typical summand $D (P,n)$ of $f_{m}$.  If $Q$ is any
point of $\J$ such that $nQ = P$, then
\[
    D (P,n) = \sum_{nb = 0} (Q+b).
\]
This shows that
\[
    \deg D (P,n) = n^{2},
\]
and  it follows that
\[
        \deg (f_{m}) = 2 \qf (m).
\]
Meanwhile,
\[
\sideset{}{^\J}\sum_{nb = 0} Q + b  = nP  + \sideset{}{^\J}\sum_{nb =  
0} b = nP,
\]
and so
\[
     \sideset{}{^\J}\sum_{i} D (P_{i},m_{i}) = \sideset{}{^\J}\sum_{i}  
m_{i} P_{i}.
\]
\end{proof}

Now suppose that for $i=0,1$, $V_{i}$ is a $\T$-equivariant $BSU
(d)$-bundle over a $\T$-fixed space $X$, and let
$V=V_{0}-V_{1}$.   Suppose for simplicity that $H^{*} (X)$ is
concentrated in even degrees, and let
\[
\spX  = \spf HP^{0}(X).
\]
Let $D^{i}$ be the divisor on $\spX\times \J$ which is obtained by
pulling back along the map
\[
      \spX \times \J \to (\cochars^{i}\otimes \Jhat)/W^{i} \times \J.
\]
Then we have the following.

\begin{Theorem}\label{t-th-divisorial-const}
If $\ChB_{1} (V) = 0 = \ChB_{2} (V)$, then $D^{0}-D^{1}$ is the
divisor of a meromorphic function on the elliptic curve $\spX\times
C$ over $\spX$, and this meromorphic function is a trivialization of
$\EJc^{*} (X^{V})$ as an $\EJc^{*} (X)$-module.
\end{Theorem}

\begin{proof}
Let $m^{i}$ be a reduction of the action of $\T$ on $V_{i}$.  By Lemma
\ref{t-le-equiv-ch-classes-for-covers}, the characteristic class
restrictions imply that
\[
       \qf (m^{0}) = \qf (m^{1})
\]
and that
\[
     \sideset{}{^{\J}}  \sum_{P\in D^{0}}  m^{0}_{P} P  =
     \sideset{}{^{\J}} \sum_{Q\in D^{1}} m^{1}_{Q}Q.
\]
\end{proof}

\appendix
\renewcommand{\thesubsection}{\arabic{subsection}}

\section{On the relationship between Borel homology and cohomology}
\label{sec-UCT}

In the appendix we work with coefficents in $k$, so that
the coefficient ring of Borel cohomology is
$H^{*}(B\T)=H^*(B\T ;k) \cong k[c]$.  In our applications, $k$ will
be a field, but we make this assumption explicitly where necessary.

The naive Kronecker pairing
\[
               H^{p}_{\T} (X;N) \xra{} \Hom (H_{p}^{\T}X,N)
\]
relating $\T$-equivariant Borel homology and cohomology with  
coefficients in
a $k$-module $N$ fails to take account of the coefficient ring
$H^{*}(B\T)=H^*(B\T ;k) \cong k[c]$.  In this section we
construct a Kronecker pairing which does reflect this structure.
To see what such a sequence might look like, we note that the
generator $c \in H^{2}B\T$
lowers degree in  $H_{*}^{\T}X$, and so all of $H_{*}^{\T}X$ is
$c$-torsion.  Thus in order to get a reasonable answer, one might hope
for a map
\[
\kappa:  H^{p}_{\T} (X) \xra{} \Hom_{H^{*}B\T}^{p} (H_{*}^{\T}X,H_{*}B 
\T),
\]
and we construct such a map.

\subsection{Some algebra}

Suppose that $M$ is a (graded) $H^{*}B\T$-module.   Let
\[
     \Gamma_{(c)}M = \{r \in M| c^{k} r = 0\text{ for some }k \}
\]
be the subgroup of $c$-power torsion in $M$.  The {\em local  
cohomology groups}
of $M$ are defined by the exact sequence
$$0 \lra H^0_{(c)}(M) \lra M \lra M[c^{-1}] \lra H^1_{(c)}(M)\lra 0,$$
and Grothendieck observed \cite{MR0224620} that they calculate the  
right derived
functors of $c$-power torsion:
$$H^*_{(c)}(M)=R^*\Gamma_{(c)}M.$$
For example, if $M$ is torsion free,
$$H^1_{(c)}(M)=M[1/c]/M, $$
and  $H^{0}_{(c)}M = 0$. A special case is instructive.

\begin{Example}
For $M=H^*(B\T)$ we have
\[
            H^{1}_{(c)}(H^*(B\T)) = k [c,c^{-1}]/k[c] \iso \Sigma^{2}H_ 
{*}B\T.
\]
However, note that the second isomorphism is not natural for ring  
automorphisms.
The natural statement, given by  the residue, involves the K\"ahler  
differentials:
$$H^1_{(c)}(H^*(B\T))\tensor_{H^*(B\T)} \Omega^1_{H^*(B\T)/k} \cong H_* 
(B\T).$$
In more concrete terms, if $\alpha$ is the automorphism multiplying $c 
$ by $\lambda$,
$\alpha$ multiplies $H_{2s}(B\T)$ by $\lambda^{-s}$, and the
part  of $H^1_{(c)}(H^*(B\T))$ in the corresponding degree (viz $2+2s 
$) by $\lambda^{-s-1}$.
\end{Example}

Note too that, if $k$ is a field, then graded $H^*(B\T)$-modules are
injective if and only if they are divisible.  If in addition $M$ is  
torsion
free, then the sequence
$$0 \lra M \lra M[1/c] \lra H^1_{(c)}(M) \lra 0$$
is an injective resolution, giving the following calculation.

\begin{Lemma}
If the coefficient ring $k$ is a field, $L$ is a torsion module
and $M$ is torsion free, we have
\[
     \Ext^1_{H^{*}B\T} (L,M) \iso
     \Hom_{H^{*}B\T} (L,H^1_{(c)}(M)) . \qqed
\]
\end{Lemma}

\subsection{The construction}

Let $Hk$ denote the inflation (in the sense of Elmendorf-May \cite 
{MR1437616})
  of the nonequivariant spectrum representing ordinary cohomology with
coefficients in the commutative ring $k$, and
let $Hb = F (E\T_{+},Hk)$ be the spectrum representing Borel cohomology
with coefficients in $k$.
These are both strictly commutative ring spectra, so we may consider the
triangulated homotopy category of modules over them. Let $HM$ be an $Hb 
$-module
spectrum with $\pi_{*}^{\T}(HM)=M$; existence and uniqueness are  
easily checked
when $k$ is a field, from the fact that the coefficient ring is of  
injective dimension 1 and
in even degrees. Borel cohomology with coefficients in $M$ is defined
as usual by
\[
     H^{p}_{\T} (X;M)  =  [ Hb \sm X_{+}, \Sigma^{p}HM ]_{Hb,\T}
                       \iso  [ X_{+}, \Sigma^{p}HM ]_{\T}.
\]

\begin{Remark}
\label{notationconflict}
To avoid confusion, we highlight some distinctions. Firstly, if
$N$ is a $k$-module one may consider the usual cohomology groups
$$H^*_{\T}(X;N)=H^*(E\T \times_{\T} X;N)$$
of the Borel construction. This does not coincide with
$H^*_{\T}(X;\epsilon^*N)$ where $\epsilon : k[c] \lra k$
is the augmentation, but in practice no confusion should arise.

The second distinction is more important.
If $I$ is an injective $k[c]$-module we may define a Brown-Comenetz
type cohomology theory
$$H^*_{\T}(X;I)_{BC}=\Hom_{H^*(B\T)}(H_*^{\T}(X),I).$$
Note that this is quite different from $H^*_{\T}(X;I)$. For
example
$$H^*_{\T}(pt ; H_*(B\T))=H_*(B\T) \not \cong H^*(B\T)=H^*_{\T}(pt;H_* 
(B\T))_{BC}.$$
Our present  notation conflicts with that used in \cite{ellT}, where
the Brown-Comenetz type theory was used without the subscript $BC$.
\end{Remark}

Given a map of $Hb$-module $\T$-spectra
\[
     f: Hb\sm X_{+} \xra{} \Sigma^{p}HM,
\]
we form
\begin{equation} \label{UCT-eq:4}
     Hb \sm X_{+}\sm ET_{+} \xra{f \sm 1} \Sigma^{p}HM \sm ET_{+}.
\end{equation}
Now apply $\pi_{*}^{\T}.$   The association $f\mapsto \pi_{*}^{\T}
(f\sm 1)$ gives a function
\[
          H^{p}_{\T} (X;M) \xra{}
\Hom_{H^{*}B\T}
     (\pi^{\T}_{*}(Hb\sm X_{+}\sm ET_{+}), \pi_{*}^{\T}(\Sigma^{p}HM  
\sm ET_{+}) ).
\]

To interpret the homotopy of the domain of \eqref{UCT-eq:4}, use the
Adams isomorphism
\cite{MR764596,LMS:esht}: if $A$ is $\T$-fixed and $B$
is $\T$-free, we have
\begin{equation}\label{eq:75}
   [A,B]^\T=[A,\Sigma B/\T].
\end{equation}

\begin{Remark}
The suspension in \eqref{eq:75} arises as smashing with $S^{ad}$, the  
Thom
space of the adjoint representation of $\T$ on its Lie algebra.
\end{Remark}

In our setting, the Adams isomorphism gives
\[
   \pi_{*}^{\T} (Hb\sm X_{+}\sm E\T_{+}) \iso
   H_{*-1}^{\T} (X) \iso \Sigma H_{*}^{\T} (X).
\]

To understand the homotopy of the target of \eqref{UCT-eq:4}, note
that, since $c$ is the Euler class of the natural representation,
  applying $\pi_{*}^{\T}$ to the cofibre sequence
\[
     HM\sm E\T_{+} \rightarrow HM\sm  S^{0} \xra{}HM\sm \widetilde{E}\T
\]
gives a triangle
\[
     \pi_{*}^{\T} (HM\sm E\T_{+}) \xra{} M  \xra{} M[c^{-1}].
\]
Thus, if $M$ is in even degrees we find
\[
     \pi_{*}^{\T} (HM\sm E\T_{+}) \iso H^{*}_{(c)}M,
\]
and we have a map
\[
\kappa: H^{p}_{\T} (X;M) \xra{} \Hom_{H^{*}B\T} (\Sigma H^{\T}_{*}X, H^ 
{*}_{(c)}M).
\]
It is easiest to sort out the gradings by example.  If $c$ is regular
in $M$, then
\[
  \pi_{k}^{\T} (\Sigma^{p}HM\sm E\T_{+}) \iso (M[c^{-1}]/M)_{k-p+1},
\]
and we have given a map
\[
     H^{p}_{\T} (X;M) \xra{} \Hom_{H^{*}B\T} (\Sigma H^{\T}_{*}X,  
\Sigma^{p-1}
(M[c^{-1}]/M))
\]
or
\begin{equation} \label{UCT-eq:3}
  H^{p}_{\T} (X;M) \xra{} \Hom_{H^{*}B\T}^{p} (H^{\T}_{*}X, \Sigma^{-2} 
(M[c^{-1}]/M)).
\end{equation}


\subsection{Isomorphisms}

Suppose now that $k$ is a field. Since  $M[c^{-1}]/M$ is an injective
$H^{*}B\T$-module, both the left and right sides of \eqref{UCT-eq:3}  
are cohomology
theories in $X$, and we have the following.

\begin{Proposition} \label{t-pr-kronecker-iso}
If $M$ is torsion free and complete, the Kronecker pairing is an  
isomorphism
\[
  \kappa: H^{p}_{\T} (X;M) \xra{\iso} \Hom_{H^{*}B\T}^{p} (H^{\T}_{*}X, 
\Sigma^{-2}(M[c^{-1}]/M)).
\]
\end{Proposition}

\begin{proof}
We must show that the map is an isomorphism when $X=S^{k}\sm
\T/\A_{+}$ for all subgroups $A$ and all integers $k$.
First, note that  the map is an isomorphism when
$X=S^{0}$: here we have the map
$$M \lra  \Hom_{H^{*}B\T}(H_*(B\T),\Sigma^{-2}(M[c^{-1}]/M)).$$
The fact that this is an isomorphism when $M$ is complete is essentially
local duality, but can be seen directly since the codomain is
$$\begin{array}{rcl}
  \Hom_{H^{*}B\T} (H_*(B\T),\Sigma^{-2}M[c^{-1}]/M)
  &\cong & \Ext^1_{H^{*}B\T} (H_*(B\T),\Sigma^{-2} M)\\
  &\cong & \Ext^1_{H^{*}B\T} (\colim_s(\ann(c^s,\Sigma^2H_*(B\T)),M)\\
  &\cong & \lim_s M/(c^s)M
\end{array}$$
since $\ann (c^s, \Sigma^2H_*(B\T))\cong \Sigma^{2s} k[c]/(c^s)$.
Using suspension isomorphisms on both sides, we obtain the result
for all spectra $X=S^k$.

Now, moving into topology, the Thom isomorphism implies that
we have an isomorphism when $X=S^{k}\sm S^{V},$ where $V$ is a
complex representation of $\T$.  If $\A$ is cyclic of finite order
$n$, then we have the cofibre sequence
\[
    \T/\A_{+} \xra{} S^{0} \xra{} S^{z^n},
\]
and so an isomorphism for cells of the form $S^{k} \sm \T/\A_{+}$
for general $\A$.
\end{proof}

The simplest example is when $M=H^*(B\T)$.

\begin{Example}
We have the isomorphism
\[
    H^{*}_{\T} (X;H^*(B\T))\iso \Hom_{H^{*}B\T} (H_{*}^{\T} (X),
H_*(B\T) \tensor_{H^*(B\T)}  (\Omega^1_{H^*(B\T)|k})^{-1})
\]
where the K\"ahler differentials can be omitted if only the $H^*(B\T)$- 
module
structure is relevant.\qqed
\end{Example}

Finally, we specialize to the case of interest to us, arising from the
geometry of an elliptic curve over a field of characteristic 0.

\begin{Example}
\label{BorelcOATA}
Now suppose that, as in Section \ref{sec:prop-equiv-ellipt}, we have an
elliptic curve $\J$ over a field $k$ of characteristic 0,
and coordinate data $t_{1}.$  For $n\geq 1$, the function  $t_{n}$
vanishes to the first order on points of exact order $n$ (and nowhere
else), so that if we let $c$ act on
$\OA\otimes \omega^{*}$ via $t_n/Dt$ we have
$$H^1_{(c)}(\OA\otimes \omega^{*}) \cong
H^1_{(t_n)}(\OA)\otimes \omega^{*} \cong T_{\A}\J \otimes \omega^{*}.$$
In any case,  $T_A\J \otimes \omega^*$ is a divisible torsion
$H^{*}B\T$-module, isomorphic to a finite sum of copies of
$H_{*}B\T$.  Applying Proposition \ref{t-pr-kronecker-iso}, we have the
natural isomorphism
\[
    H^{*}_{\T} (X;\OA\otimes \omega^{*})
  \iso \Hom_{H^{*}B\T}
    (H_{*}^{\T} (X),T_{\A}\J \otimes \omega^{*}).\qqed
\]
\end{Example}


\def\cprime{$'$} \def\cprime{$'$}

\end{document}